\documentclass[a4paper]{article}

\usepackage{amsfonts}
\usepackage{amssymb,amsthm}
\usepackage{graphicx}
\usepackage{amsmath}
\usepackage[a4paper]{geometry}
\usepackage{comment}
\usepackage{enumerate}
\usepackage{bigints}
\usepackage{dsfont}
\usepackage[colorlinks,citecolor=blue,urlcolor=blue]{hyperref}
\usepackage{mathtools}
\usepackage{enumitem}

\newtheorem{thm}{Theorem}
\newtheorem{rem}[thm]{Remark}

\newtheorem{corol}[thm]{Corollary}
\newtheorem{lem}[thm]{Lemma}
\newtheorem{prop}[thm]{Proposition}
\newtheorem{expl}[thm]{Example}
\newenvironment{assump}{\textbf{Assumption} \itshape}

\numberwithin{thm}{section}
\numberwithin{equation}{section}

\DeclareMathOperator{\E}{E}
\DeclareMathOperator{\R}{{\mathbb R}}
\DeclareMathOperator{\C}{{\mathbb C}}
\DeclareMathOperator{\Z}{{\mathbb Z}}
\DeclareMathOperator{\N}{{\mathbb N}}
\DeclareMathOperator{\Nz}{{\mathbb N} \cup \{0\}}

\DeclareMathOperator{\G}{{\mathcal G}}
\DeclareMathOperator{\J}{{\mathcal J}}
\DeclareMathOperator{\supp}{supp}

\renewcommand{\S}{\mathcal{S}}
\DeclareMathOperator{\twopi}{\mbox{$-$}}
\DeclareMathOperator{\twopii}{\mbox{$-$}\imath}

\providecommand{\Ree}[1]{\text{Re}\left(#1\right)}
\providecommand{\Imm}[1]{\text{Im}\left(#1\right)}

\renewcommand{\L}[1]{\textit{L}^{#1}\!\left({\mathbb R}\right)}
\DeclareMathOperator{\Lp}{\L{\it p}}
\DeclareMathOperator{\Lone}{\L{1}}
\DeclareMathOperator{\Ltwo}{\L{2}}
\DeclareMathOperator{\Linf}{\L{\infty}}
\DeclareMathOperator{\Pn}{\mathit{P}}
\DeclareMathOperator{\Pnhat}{\mathit{\widehat{P\,}_{\!\!n}}}
\DeclareMathOperator{\varphin}{\varphi}
\DeclareMathOperator{\varphintilde}{\widetilde{\varphi}}
\DeclareMathOperator{\varphinhat}{\widehat{\varphi}_{\mathit n}}
\DeclareMathOperator{\varphinhattilde}{\widetilde{\varphi}_{\mathit n}}
\DeclareMathOperator{\lambdanhat}{\hat{\lambda}_{\mathit n}}
\DeclareMathOperator{\setcomplement}{\mbox{\tiny \textit{C}}}
\DeclareMathOperator{\Tlambda}{\mathit{T}_{\lambda}}

\DeclareMathOperator{\Tlambdaunderline}{\mathit{T}_{ \underline{\lambda} \!}}
\DeclareMathOperator{\Xin}{2^{\mathit{J}}\Xi}
\DeclareMathOperator{\Un}{\mathit{U}_{\mathit{T}_\lambda}}
\DeclareMathOperator{\Hn}{\mathit{H}_{\mathit{T}_\lambda}}
\DeclareMathOperator{\UT}{\mathit{U}_{\mathit{T}}}

\DeclareMathOperator{\JTstar}{{\mathit{J}_{\Tlambda}^{*}}}
\DeclareMathOperator{\JThat}{{ \mathit{\hat{J}}_{\mathit T}}}
\DeclareMathOperator{\Jmax}{{\mathit{J}_{\mathit max}}}
\DeclareMathOperator{\twoJn}{2^{\mathit{J}_{\Tlambda}}}

\DeclareMathOperator{\twoJmax}{2^{\Jmax}}
\DeclareMathOperator{\twoJTstar}{2^{\JTstar}}

\providecommand{\twop}[1]{\left[2^{#1}\right]_{p}}

\providecommand{\ZJ}[1]{\mathit{Z}^{\mathit{J})}_{#1}}
\DeclareMathOperator{\varphinJ}{\varphin^{\mathit{J})}}
\providecommand{\varphinhatJ}[1]{\widehat\varphin^{\mathit{J})}_{#1}}
\DeclareMathOperator{\lambdaDn}{\lambda \Dn}
\DeclareMathOperator{\lambdanhatDn}{\lambdanhat \Dn}
\DeclareMathOperator{\lambdasupnDn}{\lambda \supnDn}
\DeclareMathOperator{\lambdabarsupnDn}{\bar \lambda \supnDn}
\DeclareMathOperator{\nunhat}{\hat \nu_{\Tlambda}}
\DeclareMathOperator{\nuThat}{\hat \nu_{\mathit T}}

\providecommand{\rightbar}[3]{\left. #1 \right|_{#2}^{#3}}
\providecommand{\rightbarHU}[1]{\rightbar{#1}{\mathit{H}}{\mathit{U}}}
\providecommand{\rightbarH}[1]{\rightbar{#1}{\mathit{H}}{}}
\providecommand{\rightbarU}[1]{\rightbar{#1}{}{U}}

\DeclareMathOperator{\Pntilde}{\mathit{\widetilde{P\,}}}
\DeclareMathOperator{\Pnhattilde}{\mathit{\widetilde{P\,}_{\!\!n}}}
\DeclareMathOperator{\Omegan}{\Omega_{\mathit{n}}}
\DeclareMathOperator{\Omegantilde}{\widetilde{\Omega}_{\mathit n}}
\DeclareMathOperator{\Omegantildecomplement}{\widetilde{\Omega}_{\mathit n}^{\setcomplement}}
\providecommand{\Prunder}[1]{\mathrm{Pr}_{\,#1}}
\DeclareMathOperator{\PrOmegantilde}{\Prunder{\Omegantilde\left( 2^j,r\right)}}

\DeclareMathOperator{\tildeK}{\mbox{{\tiny \,}}\widetilde{\mbox{{\tiny \!}}\mathit{K}{{ \,}}}\!}
\providecommand{\tildeKJ}[1]{\tildeK_J\left(#1\right)}
\providecommand{\KJ}[1]{K_J\left(#1\right)}
\DeclareMathOperator{\barK}{\mbox{{\tiny \,}}\overline{\mbox{{\tiny \!}}\mathit{K}\mbox{{\tiny \,}}}\mbox{{\tiny \!}}}
\DeclareMathOperator{\linspan}{span}

\DeclareMathOperator{\one}{{\mathds 1}}
\DeclareMathOperator{\onez}{{\mathds 1}_{\{0\}}}

\DeclareMathOperator{\Dn}{{\Delta_{\mathit{n}}}}
\DeclareMathOperator{\supnDn}{\bar{\Delta}}
\providecommand{\Dnsuper}[1]{\Delta_{\mathit{n}}^{#1}}

\DeclareMathOperator{\Dnn}{{\Delta_{\mathit{n}}} \mathit{n}}

\DeclareMathOperator{\Xs}{\mathit{X}_1, \mathit{X}_2, \ldots}
\DeclareMathOperator{\Zs}{\mathit{Z}_1, \ldots, \mathit{Z}_{\mathit{n}}}
\providecommand{\sump}[3]{\sum_{i=#1}^{\infty} \frac{\Delta_{#3}^{i} \nu_{#2}^{\ast i}}{i!}}

\providecommand{\abs}[1]{\left| #1 \right|}
\providecommand{\norm}[1]{\left\| #1 \right\|}

\renewcommand{\le}{\leqslant}

\DeclareMathOperator{\Log}{Log}
\providecommand{\qsub}{  \!\!\!\!\!\!\!\! }
\providecommand{\qqsub}{ \qsub \qsub \qsub \qsub }

\providecommand{\FT}{{\mathcal F}}
\providecommand{\FTT}[1]{{\mathcal F} \left[ #1\right]}
\providecommand{\FI}{{\mathcal F}^{-1}}
\providecommand{\FII}[1]{{\mathcal F}^{-1} \left[ #1 \right]}

\begin{document}

	\newenvironment{cond}{\mbox{ }\\ \noindent \textbf{Condition} \itshape}

\title{Adaptive nonparametric estimation for compound Poisson processes robust to the discrete-observation scheme
}


\author{Alberto J. Coca \footnote{The author is grateful to Mathias Trabs for an invitation to Hamburg, where the presentation of the results herein was greatly improved, and to Richard Nickl for repeated feedback on this manuscript. This project has received funding from the Engineering and Physical Sciences Research Council (EPSRC) under grant LFAG/141 (RG71269), and the European Research Council (ERC), 
		grant agreement No 647812.
	} \\
\\
\textit{University of Cambridge\footnote{Statistical Laboratory, Department of
Pure Mathematics and Mathematical Statistics, University of Cambridge, CB3 0WB,
Cambridge, United Kingdom. Email: alberto.j.coca@statslab.cam.ac.uk}}\\
\date{\today}}


\maketitle

\begin{abstract}
	A compound Poisson process whose jump measure and intensity are unknown is observed at finitely many equispaced times. We construct a purely data-driven 
	estimator of the L\'evy density $\nu$ 
	through the spectral approach using general Calderon--Zygmund integral operators, which include convolution and projection kernels. Assuming minimal tail assumptions, it is shown to estimate $\nu$ at the minimax 
	rate of estimation 
	over Besov balls under the losses $L^p(\mathbb{R})$, $p\in[1,\infty]$, and robustly to the observation regime (high- and low-frequency). To achieve adaptation in a minimax sense, we use Lepski\u{i}'s method as it is particularly well-suited for our generality.  
	Thus, novel exponential-concentration inequalities are proved including one for the uniform fluctuations of the empirical characteristic function. These are of independent interest, as are the proof-strategies employed to deal with general Calderon--Zygmund operators, to depart from the ubiquitous quadratic structure and to show  robustness 
	without polynomial-tail 
	conditions. 
	Part of the motivation for such generality is a new insight we include here too that, furthermore, allows us to unify the main two approaches to construct estimators used in related literature. 
\end{abstract}

\noindent\textit{Key words: adaptive nonparametric estimation, Calderon--Zygmund integral operator, compound Poisson process, Lepski\u{i}'s method, L\'evy density, low-, high- and arbitrary frequency regimes, non-linear inverse problem, robustness to observation schemes}

\noindent\textit{MSC 2010 subject classification}: Primary: 62G07; 
Secondary: 60F25, 
60G51, 62G20, 
62M05.

\section{Introduction}
\label{sec:introduction}

\subsection{The statistical non-linear inverse problem}
\label{sec:intro:snip}

Let $\Xs$ 
be a sequence of independent and identically distributed (i.i.d.) real-valued random variables with common law $\mu$ and let $N\!:=\!(N_t)_{t\geq 0}$ be 
an independent Poisson process on $\R$ with intensity $\lambda\!>\!0$. Then, $Y\!:=\!(Y_t)_{t\geq 0}$ with $Y_t\!=\!\sum_{i=1}^{N_t} X_i,  t\!\geq \!0,$
is a compound Poisson process 
with (finite) L{\'e}vy measure $\nu\!:=\!\lambda \mu$; we take the convention that an empty sum is zero so, in particular, the process 
starts at zero. We further assume that $\mu(\{0\})\!=\!0$, so that the model is fully identifiable through $\nu$ and 
$Y$ is a L{\'e}vy process. 

Compound Poisson processes are the building blocks of L{\'e}vy processes and are used in many applications within natural sciences, engineering and economics (cf. 
\cite{BNMR12} and \cite{CT04}, and references therein). In many of these, the underlying process is not perfectly observed but only discrete observations are available: e.g.,  $Y_{\Dn}, Y_{2\Dn}, ..., Y_{n\Dn}$ for some $\Dn\!>\!0$ and $n\!\in\! \N$. In view of the expression for $Y_t$ above
, the process may jump several times between two observations and we are effectively observing a random variable corrupted by a sum of a random number of copies of itself. More specifically, due to $Y$ being a L{\'e}vy process, the increments $Z_{j}:=Y_{j\Dn}- Y_{(j-1)\Dn}$, $j=1, \ldots,n$, are i.i.d. copies of $Y_{\Dn}$ with distribution
\vspace{-0.05cm}
\begin{equation}\label{eq:repPintro}
\Pn:=e^{-\lambdaDn} \sump{0}{}{n}, 
\vspace{-0.2cm}
\end{equation}
where $\ast$ denotes convolution, $\nu^{\ast i}\!=\!\overbrace{\nu \ast ... \ast \nu}^i$,  $i\geq 1$, and $ \nu^{\ast 0}\!=\!\delta_{0}$ is Dirac's delta at 0, so we are equivalently observing a sample from $\Pn\!=\!\Pn(\Dn, \nu)$. Consequently, estimating $\nu$ from discrete observations is a non-linear statistical inverse problem and, after linearisation, it is of deconvolution type with the role of the (unknown) error played by the increments themselves. Indeed, Nickl and Rei{\ss} \cite{NR12} call it an auto-deconvolution problem following the work of Neumann and Rei{\ss} in \cite{NR09}, where the ill-posedness of the problem for general L{\'e}vy processes was first studied. In the context of compound Poisson processes, the inverse problem is not ill-posed in terms of convergence rates because, denoting the Fourier transform of a finite measure by $\FT$, the characteristic function of $Y_{\Dn}$ satisfies
\begin{equation}\label{eq:repvarphiintro}
\varphi(\xi):=\FT \!\Pn(\xi) =e^{\Dn \left( \FT \nu (\xi) - \lambda \right)}, \qquad \xi \in \R,
\end{equation}
and, noting that $\abs{ \FT \nu(\xi)} \leq \lambda$ for all $\xi \in \R$, it satisfies $\inf_{\xi\in \R} |\varphi(\xi) | \geq e^{-2 \lambdaDn}>0$.

\subsection{Our result and related literature}
\label{sec:intro:resultandliterature}

Nonparametric inference for discretely observed L\'evy processes has received much attention in the last two decades. Roughly speaking, the literature can be split into two main blocks depending on the observation regime: the high-frequency scheme, where $\Dn \to 0$ and $T:=\Dnn\to \infty$ as $n\to \infty$ and any rates of convergence are given in terms of $T$; and the low-frequency regime, where $\Dn=\Delta>0$ 
is fixed and rates are in terms of $n$. In the former, part or all of the ill-posedness of the inverse problem vanishes asymptotically and estimators are built in two different ways: directly from small-time expansions such as that obtained by taking the first summand in \eqref{eq:repPintro} (see \cite{FL08,FLH09} for the expansions and \cite{FL09,FL11,KP16} and \cite{NRST16}, Section 2.2, for estimation); and, from \eqref{eq:repvarphiintro} and its generalisation 
given by the L{\'e}vy--Khintchine formula together with small-time approximations that simplify it (see \cite{BL15,CDGC14,CGC09,CGC11,D13,DM17}). Both have the drawback that $\Dn\to 0$ sufficiently fast is needed. In the low-frequency regime, the ill-posedness has to be dealt with in full and the 
route to construct estimators is the so-called spectral approach, initiated by Belomestny and Rei{\ss} \cite{BR06}, which uses the L{\'e}vy--Khintchine representation without any approximations (see \cite{BG03,CDH10,C18,CGC10,DK18,VEGS07,G09,G12,K14,NR09,NR12,T15}). The exceptions to this division are \cite{NRST16}, Section 2.3, and \cite{KK17}, where the spectral approach is used in the high-frequency regime with $\Dn\to 0$ sufficiently fast
, and \cite{KR10}, which we mention in more detail below. See \cite{BCGCMR15} for a comprehensive account of the existing results in this field.

However, in practice only finite samples are available and the statistician must choose an estimator constructed from one of the approaches above. 
Hence, estimators that can optimally deal with both regimes at once are of key importance and this is a recent topic of investigation in related problems too such as in estimation for discretely observed diffusions (cf. Chorowski \cite{Cho18} and Abraham \cite{A18}). In Chapter 2 in Coca \cite{C17}, we argued that, in the compound Poisson setting, the spectral route implicitly uses the small-time approximations considered in Duval \cite{D13}, Comte, Duval and Genon-Catalot \cite{CDGC14} and Duval and Mariucci \cite{DM17}, with the difference of being of arbitrary high-order in $\Dn$. Thus, it should provide estimators that deal with the high-frequency regime with no assumptions on how fast $\Dn \to 0$, to the extent that they should also be able to deal with the low-frequency regime. Indeed,   without knowledge of $\lambda$ and using general Calderon--Zygmund integral operators that include convolution and projection kernels such as wavelets, 
herein we build such robust estimator, $\hat \nu $, that, in addition, is adaptive in a minimax sense to an assumed Besov
-regularity of $\nu$ for all the losses $\Lp$, $p\in[1,\infty]$, under minimal assumptions; in particular, in the statistically-relevant case 
$p=\infty$, and also when $p\in[2,\infty)$, we only require a logarithmic-tail assumption. For adaptation we use Lepski\u{i}'s method as it is particularly well-suited for such generality. Furthermore, in Section \ref{sec:unification} we use this insight and our general estimator to unify the 
abovementioned approaches, which have so far been regarded as different. We note in passing that the rates of convergence we obtain will not be given in terms of $n$ but in terms of the natural quantity $\Tlambda:=\lambda T=\lambda \Dnn$, i.e. the expected number of jumps that $Y$ gives in the whole observation interval $[0,T]$. 

For more general L{\'e}vy processes observed discretely, the insight carries over and estimators constructed from the spectral approach are minimax-optimal and robust to the observation interval $\Dn$: 
this was first reported in Section 4.5 in Comte and Genon-Catalot \cite{CGC10} and in Kappus and Rei{\ss} \cite{KR10}, where the L{\'e}vy density and triplet are estimated under the $\Ltwo$ loss and a minimum distance criterion, respectively; and, it was subsequently remarked in the discussion at the end of Section 3 in Kappus \cite{K14} and in Remark 6 in Section 3 in Trabs \cite{T15}, where they  
estimate functionals of and the L{\'e}vy density under the losses $|\cdot|^2$ and $L^{\infty}(D)$, $D\!\subset \!\R\setminus\{0\}$ compact, respectively. Nevertheless, only 
\cite{CGC10} in Section 4.6 and 
\cite{K14} in Section 4 include adaptation procedures for their estimators and these are not proved 
to be robust to the discrete observation regime. Moreover, they are the only works in the literature in which adaptive estimation of L{\'e}vy densities from low-frequency observations is studied, and they do so for quadratic-type losses and under polynomial-tail assumptions.  
%
%
%
%
%
Therefore, our results fill in these three gaps and for general operators, starting with the simplest pure jump L{\'e}vy model. 

Right before completing this manuscript, the work of Duval and Kappus \cite{DK18} was released: 
they construct a convolution kernel-estimator of $\nu$ from the spectral approach that is adaptive and robust to arbitrary frequencies. Yet, 
they 
work  under the $\Ltwo$ loss with a fourth-moment condition and their rates are optimal up to logarithms 
assuming knowledge of $\lambda$. Due to these differences, our proof-strategies differ significantly too. 

Lastly, we must mention that a third block of literature considers nonparametric estimation from discrete but irregular or random observations (see \cite{K18,B11,CDGCK15}). 
Additionally, an alternative to the aforementioned purely frequentist approaches, the Bayesian approach, has recently been applied to make inference on $\nu$ or functionals of it: in a periodic setting and for low-frequency observations, 
Nickl and S{\"o}hl \cite{NS17} prove a Bernstein--von Mises theorem for its distribution, and non-adaptive minimax-optimal contraction rates of the Bayesian posterior under the supremum-norm follow. Previously, Gugushvili, van der Meulen and Spreij \cite{GvdMS15,GvdMS18} built a posterior for the forward operator in \eqref{eq:repPintro} that adaptively contracts at the optimal rate for the Hellinger distance on $\R$ for the low- and high-frequency regimes, respectively; in \cite{GvdMS18}, they report robustness to the observation interval of their approach. The concentration results developed in our proofs allow to construct a posterior for $\nu$ that contracts at the optimal rate robustly to the observation scheme for $L^p(D)$, $p\in[1,2]$ and $D\subset \R$ compact, using the constructive-testing approach introduced in Gin{\'e} and Nickl \cite{GN11}. Yet, this is out of the scope of the current work and, furthermore, it requires knowledge of the regularity of $\nu$, so the problem of finding a Bayesian posterior for $\nu$ that is adaptive on $\R$ remains open. 

\subsection{Adaptive and robust estimation in $\Lp$, $p\in[1,\infty]$}
\label{sec:intro:techinicalities}

To achieve the generality of our results 
in terms of loss functions and robustness under minimal assumptions, the routes taken in our proofs partly depart from those in existing literature on nonparametric 
estimation for discretely observed L\'evy processes. 
We now  discuss these differences, as they 
provide and constitute contributions of our work too.

In view of \eqref{eq:repvarphiintro}, the highly non-linear formal representation of $\nu$ in terms of $\Pn$ given by $\nu=\frac{1}{\Dn}\FII{\Log( e^{\lambdaDn}\varphi)}$ follows. 
Then, if  $\nu$ or functionals of it are estimated under an $L^2$ loss, the Parseval--Plancherel isometry can be readily used to work in the Fourier domain and the analysis simplifies. This strategy remains valid in estimation of 
more general L{\'e}vy processes and, indeed, it is used in most of the existing literature using the spectral approach. 
In particular, estimators are built by plugging the empirical counterparts of $\varphin$ and derivatives of it, namely
\begin{equation}\label{eq:varphinhat}
\varphinhat^{\!\!k)}(\xi) \!:= \!\frac{d^k}{d \xi^k} \left( \Pnhat e^{\twopii \xi \cdot}\right)\!:=\! \frac{1}{n} \sum_{j=1}^n (\twopii Z_j)^k e^{\twopii \xi Z_j},  \,\,\, \mbox{with } \,  k\geq 0 \, \mbox{ and}\, \Pnhat\!:=\!\frac{1}{n}\sum_{j=1}^{n}\delta_{Z_j},
\end{equation}
and the variance term (comprising linear and non-linear terms) can be dealt with by controlling $\E\left[(  \varphinhat-\varphin)^{k)}\right]^2$ 
pointwise. Model selection techniques can then be applied and this is the route is taken by the few works  
concerning adaptive estimation 
(cf. Section 3.2 in 
\cite{DK18}, 
Section 4.6 in \cite{CGC10} and Section 4 in \cite{K14}). 
However, $L^2$-convergence properties of Fourier transforms may not be informative about convergence in other $L^p$-norms, especially about the statistically significant $L^{\infty}$-norm; the only work dealing with $p=\infty$ that we are aware of is Trabs \cite{T15}, Section 3, where he deals with general L{\'e}vy processes, but his result is non-adaptive and restricted to a compact interval: the latter ultimately allows him to deal with the variance in the way 
we just described. 

To deal with the $\Lp$ loss for $p\in [1,\infty]$ without assuming $\nu\in \Ltwo$ necessarily, our analysis needs to be finer. 
In line with existing distributional results (cf. \cite{VEGS07,NR12,NRST16} and \cite{C18}), we work on the linear term to write it as an empirical process $\left( \Pnhat-\Pn\right)f$ with $f$  uniformly bounded,
and control the expectation of its  $\Lp$-norm by appropriately applying results from i.i.d. density estimation under general conditions on Calderon--Zygmund integral operators. This more precise treatment allows us to show concentration of its $\Lp$-norm too using Talagrand's inequality in the spirit of \cite{GN08,GN09,KNP12}, which is one of the main ingredients to apply (their modification of) Lepski\u{i}'s method to achieve adaptation; this method was originally introduced in Lepski\u{i} \cite{L91} and has only been applied in related literature by Trabs \cite{T15}, Section 4, in the abovementioned regression-type 
setting simpler than 
discrete observations. Similarly to Neumann and Rei{\ss} \cite{NR09}, where non-adaptive estimation of functionals was first studied, we reduce the problem of controlling the non-linear stochastic term to controlling $\E\sup_{\xi\in[-\Xi_n,\Xi_n]} | \varphinhat-\varphi|(\xi)$, $\Xi_n\!\to\! \infty$: we show this directly for $p\!\in\!\{1,2,\infty\}$ and use the Riesz--Thorin interpolation theorem to extend it to intermediate values and profit from the favourable $\Ltwo$ geometry. For adaptivity, we  use Talagrand's inequality again to show 
concentration of the uniform fluctuations of $\varphinhat$. 
This new result implies that, for $p\!\in\!\{1,\infty\}$, the non-linear term does not show exponential concentration but polynomial (of arbitrary high order) and, thus, using Talagrand's inequality is crucial to apply Lepski\u{i}'s method. In particular, for $p\!=\!1$ it  dominates over the linear term, 
a phenomenon  not yet  reported in related literature. 

The other main contribution of our work is that, remarkably, this finer analysis can be made robust to the observation regime with no polynomial-tail assumptions. Notice that $\E Y_{\Dn}^k =O(\Dn)$ if $\E X_1^k<\infty$, $k\geq 1,$ and that analogous identities hold for more general L{\'e}vy processes. Indeed, all literature reporting robustness or using the L{\'e}vy--Khintchine representation for the high-frequency regime use this observation to obtain the necessary extra acceleration
: 
e.g., it immediately implies that $\E\left[n(  \varphinhat-\varphi)^{k)}\right]^2=O(T)$, $k\geq 1$. However, our estimator only features $\varphinhat$ and no derivatives of it and, moreover, the underlying reason behind such acceleration is the collapse of $\Pn$ to $\delta_{0}$ at rate $\Dn$ (cf. \eqref{eq:repPintro}). Thus, we only exploit the latter to show that the variance of $n \left( \Pnhat-\Pn\right)f$  
is of order $\Tlambda$, and repeatedly use this in the linear term to derive robust optimal bounds and concentration. Proceeding in the same fashion, we also show that $\E\sup_{\xi\in[-\Xi_n,\Xi_n]} n| \varphinhat-\varphi|(\xi)=O\big(\sqrt{T_\lambda} \log^{1/2+\delta}(\Xi_n)\big)$ for any $\delta>0$
; note the key improvement this gives in the context of compound Poisson processes compared to Theorem 1 in \cite{KR10}, where it was shown that it is $O\big(\sqrt{n}\log^{1/2+\delta}(\Xi_n)\big)$
.  Optimal control of the non-linear term follows and, furthermore, it guarantees the right acceleration of the variance in Talagrand's inequality that allows us derive robust exponential concentration of the uniform fluctuations of $\varphinhat$. The latter, which states that they concentrate as in a parametric setting with sample size $\Tlambda$, 
is of independent interest and adds transparency to the proofs. 

\subsection{Outline}

We split the rest of the article into two main sections: Sections \ref{sec:mainresults} and \ref{sec:proofs}, devoted to our main results and part of their proofs, respectively. The former is divided into several subsections: Sections \ref{sec:settinganddefs} and \ref{sec:fnspacesandwavelets}, introducing the setting and notation used throughout together with Calderon--Zygmund integral operators and  relevant 
results from approximation theory; Section \ref{sec:assumptionsandtheestimator}, including the model assumptions and our estimator; and, Sections \ref{sec:minimaxoptandadapestimator} and \ref{sec:chfn}, containing our main results concerning our estimator and the empirical characteristic function, respectively. Section \ref{sec:minimaxoptandadapestimator} is further divided into four: Sections \ref{sec:minimaxoptimality} and \ref{sec:adaptation}, devoted to the minimax-optimality of our estimator and to adaptation; and, Sections   \ref{sec:unification} and \ref{sec:discussionDn}, where we include a novel unification of the existing literature on nonparametric estimation of discretely observed compound Poisson processes, together with several discussions on robustness to the observation scheme. Section \ref{sec:proofs} contains the proofs and is split into four: Section \ref{sec:proofs:prepresults}, devoted to some preparatory results; Sections \ref{sec:proofs:upperbound} and  \ref{sec:proofs:adaptation}, where we show 
the main theorems in Sections \ref{sec:minimaxoptimality} 
and \ref{sec:adaptation}, respectively; and, Section \ref{sec:proofs:varphivarphin}, devoted to proving the theorem in Section \ref{sec:chfn}. 

\section{Main results}
\label{sec:mainresults}

\subsection{Setting and definitions}
\label{sec:settinganddefs}

Throughout, all random quantities are defined on the probability space $\left(\Omega, \mathcal{A}, \Pr \right)$. Adopt the setting introduced in Section \ref{sec:intro:snip}:  $\mu$ and $\lambda>0$ are the jump measure and intensity of a compound Poisson process $Y:=(Y_t)_{t\geq 0}$ and we have discrete observations $Y_{i \Dn}$ for $i=1, \ldots,n$ and some $\Dn>0$. We define $\nu:=\lambda\mu$ and assume $\mu(\{0\})=0$ and $T:=\Dnn \to \infty$ as $n\to \infty$; we will only require $\supnDn :=\sup_n \Dn  <\infty$  when we explicitly state it. 
Then, $Y$ is a L{\'e}vy process and
%
%
the increments $Z_i:=Y_{i \Dn}-Y_{(i-1) \Dn}$, $i=1, \ldots,n$, are i.i.d. 
with common law $\Pn$ and empirical law $\Pnhat$  defined in \eqref{eq:repPintro} and \eqref{eq:varphinhat}, respectively. 

From now on, all functions and measures will be assumed to be Borel-measurable and Borel, respectively. For any function $f$ on $\R$ and function or measure $Q$ on $\R$, we define $Qf\!:=\!\int_{\R} f(x) Q(dx)$ and, for any 
$\xi,x\!\in \!\R$, we define the Fourier transform and its inverse,
\begin{equation}\label{def:FTFI}
\FT Q (\xi):= Qe^{ \twopii \xi \cdot} 
\qquad \mbox{and} \qquad \FI Q (x) := \frac{1}{2\pi}\FT \mu (-x)= \frac{1}{2\pi} \FT  Q^{-} (x),
\end{equation}
where $ Q^{-}(B)= Q(-B)$ for any Borel set $B\subseteq\R$.  We refer the reader to \cite{F99} for the standard results on Fourier analysis and convolution theory we use in what follows. We write $A_c\lesssim B_c$ if $A_c\le C B_c$ holds with a uniform constant $C$ in the parameter $c$ and $A_c\thicksim B_c$ if $A_c\lesssim B_c$ and $B_c\lesssim A_c$. We use the standard notation $A_c = O(B_c)$ and $A_c = o(B_c)$ to denote that $A_c/B_c$ is bounded and $A_c/B_c$ vanishes, respectively, as $c\to \infty$. We write $r_c'=O_{\Pr}(r_c)$ to denote that $r_c' r_c^{-1}$ is bounded in $\Pr$-probability.

\subsection{Function spaces and Calderon--Zygmund operators}
\label{sec:fnspacesandwavelets}

In this section we review the relevant material from Chapter 4 in \cite{GN16} that we use throughout. Let $\Lp$, $p\in [1,\infty]$, denote the usual Lebesgue space on $\R$ with norm $\norm{\cdot}_p$,  let $C_u(\R)$ be the space of uniformly continuous functions on $\R$ and let $\S(\R)$ denote the Schwartz space on $\R$. We say that a  class of functions is VC if it is of Vapnik--\u{C}ervonenkis type (cf., e.g., Definition 3.6.1 in \cite{GN16}). 
Let $f:\R\to \R$, $K :\R\times\R \to \R$ integrable, which is generally referred to as a kernel function, and $K_h(\cdot,\cdot)=h^{-1}K(\cdot /h, \cdot /h)$ for $h>0$. Then, to construct our estimators we consider integral operators of the form
\begin{equation*}
f \mapsto K_h(f):=\int_{\R}  f(y) K_h(\,,y) dy = \left\langle f, K_h(\, ,\cdot)\right\rangle,
\end{equation*}
where 
$\left\langle f, g\right\rangle:= \int_{\R} f \bar{g}$ with $\bar g$ denoting the complex conjugate of $g$. 
In the theory of singular integrals, these are sometimes called Calderon--Zygmund type-of operators and, as we show in the examples below, they comprise convolution and projection kernels. For all $x,\xi \in \R$ we define $\tildeK(x,\xi)=\tildeK_1(x,\xi)$, where
\begin{equation*}
\tildeK_h(x,\xi):= K_h\big(e^{ \twopii \xi \cdot} \big)(x)=\FTT{K_h(x,\cdot)}(\xi), \quad \mbox{and} \quad 
\tildeK_h(f)(x):= \left\langle f, \tildeK_h(x ,\cdot)\right\rangle. 
\end{equation*}
\begin{cond} \label{condition}
	On the kernel $K$ we impose the 
	following conditions: 
	\begin{enumerate}[label= \arabic*., ref=\arabic* ]
		\item  for every $x \in \R$ and $k\in \N$, \label{cond:1}
		\begin{equation*}
		\int_{\R}K(x,x+y) \,dy=1 \qquad \mbox{and} \qquad \int_{\R}K(x,x+y) \,y^k \,dy=0\,;
		\end{equation*}
		\item for every $y \in \R$ and $k\in \N$, there exists a bounded function $\barK: \R^{+} \to \R^{+}$ such that 
		\label{cond:2}
		\begin{equation*}
		\sup_{x\in \R}|K(x,x-y)| \leq \barK(|y|) \lesssim (1+|y|)^{-k}\,;
		\end{equation*}
		\item 
		for every $h>0$ there exists a countable and uniformly bounded VC class of functions $\mathcal{K}\!:=\!\{\kappa_t\!:\!\R \!\to\! \R\!: t \!\in\! \mathcal{T}\}$ 
		such that, for 
		any finite measure $Q=Q_{d}+Q_{ac}$ on $\R$ with $Q_{d},Q_{ac}$ discrete and absolutely continuous with respect to Lebesgue such that $Q_{ac}$ has density $q\in \Linf$, 
		\label{cond:3}
		\begin{equation*}
		\norm{hK_h(Q)}_{\infty} \lesssim  \sup_{t \in \mathcal{T}} |Q\kappa_t| \qquad \mbox{and} \qquad \sup_{t \in \mathcal{T}}
		\int_{\R} {\kappa_t}^2(x) q(dx) \lesssim h \norm{q}_{\infty},
		\end{equation*}
		where the constants hidden in the notation $\lesssim$ depending on $K$ only; and,
		\item 
		$\supp(\tildeK(x,\cdot))$ does not change with $x\!\in\! \R$ and, for some $\Xi\!>\!0$, it is in $ [-\Xi,\Xi]$. \label{cond:4}
	\end{enumerate}
\end{cond}


Conditions \ref{cond:1} and \ref{cond:2} are standard: due to $k\leq \infty$, they guarantee that the integral operator has infinite order, roughly meaning that $K_h(f)$ captures the regularity of $f$ regardless of how high it is; more precisely, it allows us to control the approximation error $\norm{K_h(f)-f}_p$ optimally 
when $f\in \Lp$, $p\in [1,\infty]$, 
lives in a space of arbitrary smoothness in a sense to be made precise in \eqref{eq:boundLrnormKjpernu} at the end of the section.  Condition \ref{cond:2} also implies that for any $f\in \Lp$ and $h>0$, 
\begin{equation}\label{eq:boundnormKJfpbynormfp}
\norm{K_h(f)}_p \lesssim \norm{f}_p,
\end{equation}
where the constant  in $\lesssim$ only depends on $\barK$. 
Condition \ref{cond:3} 
will be used to control the $\norm{\cdot}_p$-norm of the empirical process-form of 
the linear term (see discussion in Section \ref{sec:intro:techinicalities}) when $p\!=\!\infty$. 
And, 
Condition \ref{cond:4} is a band-limitedness 
condition inducing a spectral cut-off required throughout to deal with the inverse nature of the problem. 

\begin{expl}[Convolution operators]\label{ex:1}
	For some $K\in \Lone$
	, let
	\begin{equation*}
	K(x,y)=K(x-y), \qquad x,y\in \R.
	\end{equation*} 
	%
	Functions $K$ satisfying Conditions \ref{cond:1}-\ref{cond:4} 
	are well-known: e.g., so-called `flat-top kernels', such as the one introduced in \cite{MP04}, defined for any $b>0$ and $c\in(0,1)$ through
	\begin{equation}\label{eq:flattopkernel}
	\FT K (\xi)=\left\{
	\begin{array}{ll}
	1, & \mbox{ if } |\xi|\leq c,\\
	\exp\left(-\frac{b}{(|\xi|-1)^2}\exp\left(-\frac{b}{(|\xi|-c)^2}\right)\right), & \mbox{ if } c<|\xi|< 1,\\
	0, & \mbox{ otherwise;}
	\end{array}	
	\right.
	\end{equation} 
	indeed, $K$ satisfies Condition \ref{cond:4} 
	trivially in view of $\tildeK(x,\xi)=e^{ \twopii \xi x} \FT K(-\xi)$; furthermore, $\FT K$ is infinitely differentiable on $\R$, so $K\in \mathcal{S}(\R)$ and Conditions \ref{cond:1}-\ref{cond:2} follow due to $\int_{\R}K(y) \,dy=\FT K(0)$ and $\int_{\R}K(y) \,y^k \,dy=(\twopii)^{-k}\frac{d^k}{d\xi^k}\FT K(\xi) \mid_{\xi=0}$ for all $k\in \N$, and Condition \ref{cond:3} by, e.g., Remark 5.1.4 and Case 4 in the proof of Theorem 5.1.5 in \cite{GN16}.
\end{expl}

\begin{expl}[Projection operators]\label{ex:2}
	For some $\phi\in\Ltwo$
	, let
	\begin{equation*}
	K(x,y)=\sum_{k\in \Z} \phi(x-k) \phi(y-k), \qquad x,y\in \R.
	\end{equation*}
	In the terminology of wavelet analysis, $\phi$ is generally referred to as the scaling function. It can be chosen to satisfy Conditions \ref{cond:1}-\ref{cond:4} and, moreover, so that it generates a multiresolution analysis of $\Ltwo$
	: $V_0:=\{\phi_k\!:=\!\phi(\cdot-k)\!: k\in \Z\}$ forms an orthonormal system in $\Ltwo$; $V_{j-1} \subset V_j$ for all $j\in \N$, where $V_j:=\linspan(\{\phi_{j,k}\!:=\!2^{j/2}\phi(2^j\!\cdot-k)\!: k\in \Z\})$; and, $\cup_{j\geq 0} V_j$ is dense in $\Ltwo$. Writing $h=2^J$ we have that $K_J(f):=K_{h(J)}(f)$ is the projection of $f$ onto $V_J$. Equivalently, a band-limited wavelet $\psi$ can be constructed from $\phi$ (cf. end of Section \ref{sec:implementation} below), i.e. a function $\psi\! \in \!\Ltwo$ such that $\{\psi_{j,k}\!:=\!2^{j/2}\psi(2^j\!\cdot-k)\!: k\!\in\! \Z\}$ is an orthonormal basis for the orthogonal complement $W_j$ of $V_j$ in $V_{j+1}$, so that $\Ltwo\!=\!V_0\oplus (\bigoplus_{j=0}^{\infty}W_j)$ and
	\begin{equation}\label{eq:defKjper}
	K_{J}(f)(x)=\sum_{k\in \Z} \left\langle f, \phi_{k}  \right\rangle  \phi_{k}(x)+ \sum_{j=0}^{J-1} \sum_{k\in \Z}\left\langle f, \psi_{j,k} \right\rangle \psi_{j,k}(x)  , \quad J\in \N \cup \{0,\infty\}, x\in \R.
	\end{equation}
	In the historical development of orthogonal wavelet basis, scaling functions satisfying all of the above conditions were the first to appear after Haar and Shannon: Meyer \cite{M86} constructed $\phi$ through
	\begin{equation}\label{eq:FTphi}
	\FT \phi (\xi)=\left\{
	\begin{array}{ll}
	1, & \mbox{ if } |\xi|\leq \frac{2\pi}{3},\\
	\cos\left[\frac{\pi}{2}\chi\left(\frac{3}{2\pi}|\xi|-1\right)\right], & \mbox{ if } \frac{2\pi}{3}\leq |\xi|\leq \frac{4\pi}{3},\\
	0, & \mbox{ otherwise,}
	\end{array}	
	\right.
	\end{equation} 
	where $\chi:\R\to \R$ is any infinitely differentiable function satisfying $\chi(\xi)=0$ for $\xi\leq 0$, $\chi(\xi)+\chi(1-\xi)=1$ when $\xi\in (0,1)$ and $\chi(\xi)=1$ for $\xi\geq 1$; 
	indeed, $\phi$ satisfies Condition \ref{cond:4} 
	trivially in view of $\widetilde{K}(x,\xi)=\FT \phi (\xi) \sum_{k\in \Z}\phi(x-k) e^{\twopii \xi k}$; furthermore, $\FT \phi$ is infinitely differentiable on $\R$, so $\phi\in \mathcal{S}(\R)$ and we refer the reader to, e.g., Propositions 4.2.6 and 4.2.5, and Remark 5.1.4 in \cite{GN16} for further details as to why it satisfies Conditions \ref{cond:1}-\ref{cond:3}. 	If we denote the even function in \eqref{eq:flattopkernel} by $\upsilon$, we can simply take 
	\begin{equation}\label{eq:chi}
	\chi(\xi)=\frac{1}{2}\upsilon(2\xi-1)\one_{ (-\infty,1/2]}(\xi)+\left(1-\frac{1}{2}\upsilon(2\xi-1)\right)\one_{ [1/2,\infty)}(\xi), \qquad \xi\in \R.
	\end{equation}
	We remark that, despite our results apply to wavelet estimators, the adaptation strategy we use in Section \ref{sec:adaptation}, namely Lepski\u{i}'s method, does not exploit their ability to extract local structure of inhomogeneous jump measures. This is a compromise we make to achieve the generality of our results.

\end{expl}


With the notation and conditions of Example \ref{ex:2}, if $f\in \Lp$, $p\in[1,\infty)$, or $f\in C_u(\R)$, for $J=\infty$ it holds that $K_{J}(f)=f$ in $\Lp$ or in $C_u(\R)$, respectively. We are now ready to introduce Besov spaces. We denote them by $B_{p,q}^s(\R)$ with $p\in [1,\infty], q\in [1,\infty]$ and $s>0$, and we shall only be concerned with Nikolski\u{i} classes, i.e. $q=\infty$, for simplicity. 
Using the band-limited basis introduced in Example \ref{ex:2}, we define
\begin{equation*}
B_{p,\infty}^s(\R)\!:= \! \left\{
\begin{array}{ll}
\{f\in \Lp\!: \norm{ f}_{B_{p,\infty}^s}<\infty\}, & \mbox{ if } p\in[1,\infty), \\
\{f\in C_u(\R)\!: \norm{ f}_{B_{p,\infty}^s}<\infty\}, & \mbox{ if } p=\infty, 
\end{array}	
\right.
\end{equation*}
where, letting $\norm{ a }_{\ell^p}\!:=\!  (\sum_{k\in \Z}\abs{a_k}^p)^{1/p}$ if $p\!\in\! [1,\infty)$ and $\norm{ a }_{\ell^{\infty}}\!:=\!\sup_{k\in\Z} |a_k|$ for  any real sequence $a\!=\!(a_k)_{k\in \Z}$, 
\begin{equation*}
\norm{ f}_{B_{p,\infty}^s}\!\!:= \norm{ \left\langle f, \phi_{\cdot}  \right\rangle}_{\ell^p} + \sup_{j\geq 0}  2^{j(s+1/2-1/p)} \norm{\left\langle f, \psi_{j,\cdot} \right\rangle}_{\ell^p}.
\end{equation*}
Then, for any $f\!\in\! B_{p,\infty}^s(\R)$ and for $K_{J}(f)$ as in Example \ref{ex:2}, it is clear that 
\begin{equation}\label{eq:boundLrnormKjpernu}
\norm{K_{J}(f) -f }_{p} 
\lesssim \norm{f }_{B_{p,\infty}^s} 2^{-Js}, 
\end{equation}
where the constant in $\lesssim$ 
depends on $K$. 
From now on we write $h\!=\!2^J$ and $K_J\!=\!K_{h(J)}$ for general operators for conciseness and because in the adaptation results of Section \ref{sec:adaptation} we have to discretise $h$ for any integral operator so it will be a natural transformation. With this notation and using Littlewood--Paley theory, it can be shown that any integral operator satisfying Conditions \ref{cond:1} and \ref{cond:2} satisfies the last display  with the hidden constant possibly depending on $s$ too (cf. Proposition 4.3.8 in \cite{GN16}).

\subsection{Assumptions and the estimator} \label{sec:assumptionsandtheestimator}

\subsubsection{Assumptions}

\begin{assump} \label{assumption}
	We assume that the unknown 
	$\mu=\nu/\lambda$ has a density with respect to Lebesgue's measure, denoted by $\mu$ too, satisfying 
	the following assumptions:
	\begin{enumerate}[label= \arabic*., ref=\arabic* ]
		\item $\mu\in B_{p,\infty}^s(\R)$ for some $p\in [1,\infty]$ and $s>0$; and,\label{ass1}
		\item depending on the result, one of the following tail conditions is satisfied:\label{ass2}
		\begin{enumerate}
			\item $\norm{\mu}_{q,\alpha}:=\norm{(1+|\cdot|)^{\alpha} \mu}_q<\infty$ for some $\alpha>0$ and $q\in[1,\infty)$; or, \label{ass2a}  
			\item $	\norm{ \log^{\beta} (\max\{|\cdot|,e\}) \mu }_1<\infty$ 
			for some $\beta>1$. \label{ass2b}
		\end{enumerate}
	\end{enumerate}
\end{assump}


As it is customary, Assumption \ref{ass1} will be used to control part of the bias of the estimator via \eqref{eq:boundLrnormKjpernu}. Assumptions \ref{ass2a} and \ref{ass2b} presume $\mu$ to have some polynomial and logarithmic moment, respectively,  and, further down in this section, we will comment on their roles. To show Theorems \ref{thm:upperbound} and \ref{thm:adaptation} in the next section when Assumption \ref{ass1} is satisfied for some $p\in[1,2)$, we will have to strengthen slightly condition $\alpha>0$ in Assumption \ref{ass2a}, whilst for $p\in[2,\infty]$ we will only require Assumption \ref{ass2b}. We refer the reader to the discussion after Theorem \ref{thm:upperbound} for more details on the mildness of these assumptions to prove our results.


\subsubsection{The estimator}

To construct our estimator we define 
%
\begin{equation*}\label{eq:varphintilde}
\varphintilde(\xi):=e^{\lambdaDn} \varphin(\xi)
= e^{\Dn \FT \nu (\xi) }, \qquad \xi \in  \R,
\end{equation*}
where the second equality follows from the L{\'e}vy--Khintchine formula  \eqref{eq:repvarphiintro}. Consequently, 
\begin{equation}\label{eq:identityLK2}
\FT \nu (\xi) = \frac{1}{\Dn} \Log \varphintilde \, (\xi) \qquad  \mbox{for all } \xi\in  \R,
\end{equation}
where $\Log \varphintilde$ is the distinguished logarithm of $\varphintilde$,  i.e. the unique continuous function satisfying $\exp \big(\Log(\varphintilde)(\xi)\big)=\varphintilde(\xi)$ for all $\xi \in \R$
; we remark that it coincides with the principal branch of the logarithm 
if $\varphintilde$ does not cross the negative real line, e.g. when $\lambdaDn < \pi$, which we need not assume. It follows that, formally, 
\begin{equation}\label{eq:identityestimator}
\nu(x) = \frac{1}{\Dn}  \FII{ \Log \varphintilde} (x), \qquad x \in \R. 
\end{equation}
Prior to proposing a data-driven approximation, the right hand side can be approximated with the help of the Calderon--Zygmund operators from the previous section; we recall that we use the notation introduced at the end of it throughout. 
Then, for any $J\in \Nz$, a formal approximation to $\nu(x)$, $x\in \R$, is given by 
\begin{align}\label{eq:innerprodPlancharel}
\KJ{ \frac{1}{\Dn}  \FII{ \Log \varphintilde}  }(x)&:= \frac{1}{\Dn} \int_{\R} \KJ{x,y} \FII{ \Log \varphintilde}\!(y) dy \notag \\
&=\frac{1}{2\pi\Dn} \left\langle    \Log \varphintilde ,\tildeKJ{x,\cdot} \right\rangle=:\frac{1}{2\pi}\tildeKJ{\frac{\Log \varphintilde}{\Dn}  }\!(x) ,
\end{align}
where in the second equality we used Fubini's theorem formally. 
Note that, 
\begin{equation}\label{eq:FTphikFTpsiJk}
\tildeKJ{x,\xi}=\tildeK\left(2^Jx, 2^{-J}\xi\right), \qquad \mbox{for all }J\in \Nz \mbox{ and } x,\xi\in \R,
\end{equation}
and, using 
Conditions \ref{cond:2} and \ref{cond:4}, it is simple to show that, for all $J\in \Nz$, 
\begin{equation}\label{eq:suptildeKJ}
\sup_{x\in \R} |\tildeKJ{x,\cdot}|\in  \Linf  \quad \mbox{and} \quad \supp\left(\sup_{x\in \R} |\tildeKJ{x,\cdot}|\right) \subseteq [-\Xin,\Xin].
\end{equation}
Then, whilst the left side in \eqref{eq:innerprodPlancharel} 
is valid only formally, the right  side therein 
is well defined for all $x\!\in\! \R$ 
in view of \eqref{eq:identityLK2} and  $\nu\! \in\! \Lone$. 
Consequently, to estimate $\nu$ we can estimate the latter by plugging in the empirical version of $\varphintilde\!:=\! e^{\lambdaDn} \varphin$  obtained by replacing $\lambda$ by an estimator of it and  $\varphin$ by its empirical version defined in \eqref{eq:varphinhat} when $k=0$. 
Due to the assumption that the L{\'e}vy measure $\nu$ of the compound Poisson process $Y$ 
has a Lebesgue density, $\lambda$ can be estimated by
\begin{equation}\label{eq:lambdanhat}
\lambdanhat:= - \frac{1}{\Dn} \log\left( \frac{1}{n}\sum_{i =1}^n \onez(Z_i)\right). 
\end{equation}
Hence, we estimate $\varphintilde$ by
\begin{equation}\label{eq:hatwidetildevarphin}
\varphinhattilde(\xi):=e^{\lambdanhat \Dn} \varphinhat(\xi)=\frac{1}{n_0} \sum_{i=1}^n e^{\twopii \xi Z_i}, \qquad \xi \in \R,
\end{equation}
where $n_0$ is the number of zero increments in the sample $\Zs$. While estimating $\lambda$ by $\lambdanhat$  results in a simple expression for $\varphinhattilde$, it also implies that both estimators are not well defined when $n_0=0$. The probability that this occurs 
is bounded above by $\left(1-\exp(-\lambdasupnDn)\right)^n$, which tends to $0$ exponentially fast as $n\to \infty$, so 
if $n_0=0$  the final estimator for $\nu$ will be set to zero; see Remark \ref{rem:n0=0} below for more details on how to estimate $\nu$ when $n_0=0$.  Another complication may arise from using $\varphinhattilde$ to estimate the right hand side in \eqref{eq:innerprodPlancharel}: a priori, $\varphinhattilde$ could cross the origin, and 
the distinguished logarithm is not defined if the function it acts on takes the value zero. Yet, in view of \eqref{eq:suptildeKJ}, 
our estimation strategy will be justified in the sets
\begin{equation}\label{eq:defOmegan}
\Omegan=\Omegan\left(2^J\right):= \left\{ \omega \in \Omega: \, \inf_{ |\xi| \leq \Xin } |\varphinhat(\xi)| >0 \, \,\, \mbox{ and } \,\,\, n_0>0 \right\};
\end{equation}
it will follow from 
Lemma \ref{lem:OmegantildesubseteqOmegan} below that, under Assumption \ref{ass2b}, for the familiar choice
\begin{equation}\label{eq:2J}
2^J=\twoJn:= \left(\frac{\Tlambda}{(\log \Tlambda)^{\one_{\{\infty\}}(p)}}\right)^{1/(2s+1)}
\end{equation} 
and some constant $L>0$,  $\Pr(\Omegan)\geq  1-e^{- L \Tlambda}$, which tends to $1$ exponentially fast as $\Tlambda \to \infty$
. On the complement of $\Omegan$ we let the estimator be identically zero and remark that even in the extensive simulations included in Section 4.4 in \cite{C17} where a similar estimation strategy was used, no experiments falling into these sets have been reported.  
Lastly, to prove our results, $ \Dnsuper{-1}\Log\varphinhattilde$ 
cannot be arbitrarily large
, so we truncate it 
to an increasingly large threshold when its modulus exceeds it; in view of \eqref{eq:identityLK2}, such modification is harmless asymptotically. In turn, it will mean that for $p\in[1,2)$ we have to truncate the tails of the estimator to zero
---which we also need to control the non-linearities
---, resulting in the appearance of Assumption \ref{ass2a}. Hence, for $H
, U
\to \infty$ and on $\Omegan$,
we propose the estimator
\begin{align}\label{eq:estimatornu}
\hat{\nu}(x)&:=  
\frac{1}{2\pi} \tildeKJ{\rightbarU{\frac{\Log \varphinhattilde}{\Dn} }}\!(x)
\one_{[-H,H]}(x), \qquad x\in \R,
\end{align}
where $\rightbarU{z}:= z \one_{|z|\leq U}+U\one_{|z|> U}$ for any $z\in\C$ and $U\geq 0$, and on $ \Omegan^{\!\!\!\!\setcomplement}$ we set $\hat{\nu}\equiv 0$. 
When estimating $\nu$ in the $\norm{\cdot}_{p}$-norms for $p\in[2,\infty]$, we will not need the indicator function or, equivalently, we will be able to take $H=\infty$. 

\begin{rem}
	\label{rem:n0=0}
	The estimator $\hat{\nu}$ is identically zero when $n_0=0$ but, as mentioned after \eqref{eq:hatwidetildevarphin}, the likelihood of this occurring tends to $0$ exponentially fast as $n\to \infty$. This case (and $n_0>0$ too) can be dealt with by using the estimators for $\lambda$ proposed in \cite{C18} and \cite{C17} instead of $\lambdanhat$ to estimate $\varphintilde$,  as they both satisfy a central limit theorem under Assumptions \ref{ass1} and \ref{ass2b}. Furthermore, their use, together with that of the estimator of the drift proposed in \cite{C18}, guarantees the extension of the results herein to the setting considered in \cite{C18} where $Y$ has an unknown, possibly non-zero, drift and $\nu$ has a discrete component, possibly allowing for perfect cancellations of unobserved jumps in between observations. Nevertheless, the intricate expressions of these estimators complicate parts of the proofs, so this extension is beyond the scope of the present article.
\end{rem}

\subsubsection{Implementation of the estimator}\label{sec:implementation}


In practice, $\hat{\nu}$ is computed at finitely many points, so the indicator $\one_{[-H,H]}$ in its expression may be interpreted as the window on which $\hat{\nu}$ is computed and no actual truncation  needs to be implemented. Similarly, the key quantity $\varphinhat$ featuring in the estimator can be evaluated only at a finite grid of $[-\Xin,\Xin]$. Since we show that $\varphinhat$ converges uniformly to $\varphi$ on such intervals, $\varphinhat$ and $\varphi$ are continuous, and $\inf_{\xi\in \R}|\varphi(\xi)|\geq e^{-2\lambda \Dn}>0$, if the grid is fine enough, the condition for $\varphinhat$ in the definition of $\Omegan$ in \eqref{eq:defOmegan} can be relaxed to be satisfied only at the grid; let us denote this new probability set by $\Omega_n'$. As argued in Section 2.3.1 in \cite{C17}, $\Log \varphintilde$ coincides with the principal branch of the logarithm of $\varphintilde$ if $\lambdaDn<\pi$. Then, due to the uniform convergence of $\varphinhat$ to $\varphi$, when on $\Omega_n'$ and if $\lambdaDn<\pi$ (in fact, $\lambdaDn\lesssim 5$ suffices in many practical situations as observed in Section 4.4.1.2 in \cite{C17}) the computation of $\Log \varphintilde$ is immediate generally; 
if $\lambdaDn\geq \pi$, the grid may have to be further refined to compute $\Log \varphintilde$ following the construction of the distinguished logarithm in Theorem 7.6.2 in \cite{C01}. Either way, it is clear that the truncation induced by $U$ in the discretised version of the inner product in $\hat \nu$ 
can be ignored. The resulting quantity can be computed using fast Fourier transforms (FFT) and we know spell out the details for the choices of operators in Examples \ref{ex:1} and \ref{ex:2}. \\

\noindent\textbf{Example 2.1} (Continued). \textit{Without the truncations and before discretisation, 
	$\hat \nu$ is
	\begin{equation*}
	\int_{\R}  \frac{\Log \varphinhattilde(\xi)}{\Dn} \FII{K_J(x,\cdot)}(\xi) d\xi  =  \frac{1}{2\pi \Dn} \int_{\R} e^{\twopii \xi x} \Log \varphinhattilde(\xi) \FT K \left(2^{-J}\xi\right) d\xi,
	\end{equation*}
	where we recall that $\FT K$ is as in \eqref{eq:flattopkernel}. Therefore, if $[-H,H]$ is discretised into a grid of size $M\sim 2^J$
	, the discretised versions of all integrals can be easily and efficiently computed with a single FFT and we refer the reader to Section 4.3.1 in \cite{C17} for more details (note that the expression on the left hand side above falls under the general expression on the right hand side of (4.3.1) therein).}\\

\noindent\textbf{Example 2.2} (Continued). \textit{Due to Conditions \ref{cond:2} and \ref{cond:4}, the corresponding quantity is 
	\begin{align*}
	\sum_{k\in \Z}\phi_{J,k}(x) \!\int_{\R}  \!\frac{\Log \varphinhattilde(\xi)}{\Dn}  \FII{\phi_{J,k}}\!(\xi) d\xi  \!=\! \sum_{k\in \Z}\frac{\phi(2^Jx\!-\!k)}{2\pi \Dn} \!\int_{\R} \!e^{\twopii \xi k2^{-J}} \!\Log \varphinhattilde(\xi) \FT \phi (2^{-J}\xi) d\xi,
	\end{align*}
	where we recall that $\FT \phi$ is as in \eqref{eq:FTphi} with $\chi$ given in \eqref{eq:chi} so, in particular, it is even. Then, by the uniform convergence of $\varphinhat$ to $\varphi$ on compact intervals, the regularity of $\nu$ and the localisation of the Meyer-type wavelets, only finitely many integrals are nonzero within machine precision. 
	Furthermore, if the tails of $\nu$ decay sufficiently fast, the number of these is $K\sim 2^J$ and all integrals can be computed easily and efficiently with a single FFT as in the previous example. 
	Due to $\phi\in \S(\R)$ and the Fourier inversion theorem, 
	\begin{equation*}
	\phi(x)=\frac{1}{2\pi } \int_{\R} e^{\imath \xi x } \FT \phi (\xi) d\xi \qquad \mbox{for all }x\in \R, 
	\end{equation*}
	so if $[-H,H]$ is discretised into a grid of size $M\sim 2^J$, the values of the functions $\phi_{J,k}(x)$ can be approximated again with a single FFT. A more classical alternative to this route once the approximation of the integrals in the penultimate display has been computed goes as follows: apply the decomposition and reconstruction (or Laplace pyramidal, filter-bank style) technique of Mallat \cite{M89} with the low-pass filter 
	\begin{equation*}
	m_0(\xi/2)=\left\{
	\begin{array}{ll}
	\FT\phi(\xi), & \mbox{ if } \xi \leq 4\pi/3,\\
	0, & \mbox{ otherwise },
	\end{array}
	\right.
	\qquad\mbox{(extended periodically to $\R$)}
	\end{equation*}	
	to find approximations to 
	\begin{align*}
	\hat \alpha_k:=
	\left\langle  \FII{\frac{\Log \varphinhattilde}{\Dn} \one_{ \supp(\phi_k)} }
	\!,   \phi_{k} \right\rangle \,\,\mbox{ and }\,\, \hat \beta_{j,k}:=
	\left\langle  \FII{\frac{\Log \varphinhattilde}{\Dn} \one_{ \supp(\psi_{j,k})}}
	\!,   \psi_{j,k} \right\rangle;
	\end{align*}
	and, for $K\sim 2^J$ as above, compute
	\begin{align*}
	\sum_{|k|\leq K} \hat \alpha_k  \phi_{k}(x) + \sum_{j=0}^{J-1} \sum_{|k| \leq K}\hat \beta_{j,k} \psi_{j,k}(x),   
	\end{align*}
	where $\phi$ and $\psi$ are evaluated at a dyadic discretisation of $[-H,H]$ with, e.g., a cascade algorithm using 
	$m_0$. This is an infinite impulse filter but, due to $\phi \in \S(\R)$, the tail coefficients in its discrete Fourier transform-representation can be truncated safely. Indeed, decomposition and reconstruction, and cascade algorithms for Meyer-type wavelets are inbuilt in the wavelet packages of most contemporary programming languages. }


\subsection{Asymptotic guarantees for the estimator} \label{sec:minimaxoptandadapestimator}

\subsubsection{Minimax-optimality}
\label{sec:minimaxoptimality}


Let $\FT_{p,s,\alpha}(B_1,B_2)\!:=\!\{f\geq 0\!: \norm{f}_1=1, \norm{f}_{B_{p,\infty}^s} \!\leq \!B_1 \mbox{ and, if $p\in[1,2)$, } \norm{f}_{1,\alpha}\!\leq\! B_2 \}$ for $p\!\in \![1,\infty]$, $s\!>\!0$, $\alpha>2/p-1$ and $B_1,B_2\!\geq \!1$. By the classical work in \cite{DJKP96}, the \textit{minimax rate of estimation} for a probability density function  $f\in\mathcal{F}_{p,s,\alpha}(B_1,B_2)$ from an i.i.d. sample of it of size $m\!>\!0$ under the norm $\norm{\cdot}_p$ is 
\begin{equation}\label{eq:varepsilonm}
\varepsilon_{m}:= \left(\frac{m}{(\log m)^{\one_{\{\infty\}}(p)}}\right)^{-s/(2s+1)},
\end{equation} 
i.e.,  there exist 
constants $\underline{C}=\underline{C}_{p,s,\alpha}(B_1,B_2),\overline{C}=\overline{C}_{p,s,\alpha}(B_1,B_2)>0$ and an estimator $\hat f_m$ such that 
\begin{equation*}
\liminf_{m} \varepsilon_{m}^{-1} \inf_{\tilde f_m}\sup_{  \FT_{p,s,\alpha}(B_1,B_2) 
} E_{f}\norm{\tilde{f}_m-f }_p \geq \underline{C} \quad \mbox{ and} \quad \sup_{ \FT_{p,s,\alpha}(B_1,B_2) 
}E_{f}\norm{\hat{f}_m-f }_p \leq \overline{C} \varepsilon_{m},
\end{equation*}
where the infimum is over all possible estimators $\tilde f_m$. An estimator that attains the rate $\varepsilon_{m}$ is called \textit{minimax-optimal}, and if it further satisfies the property for $\hat f_m$ above without using knowledge of $s$ in its construction, \textit{adaptive in the minimax sense}. Note that, in an interval of time $T$, the number of (independent) jumps that the compound Poisson process $Y$ gives follows a Poisson distribution with parameter $\lambda T=:\Tlambda$. Thus, if we continuously observed $Y$ on $[0,T]$ for large $T$, we would expect to see around $\Tlambda$ i.i.d. random variables with common distribution $\mu$. Intuitively, estimation in the continuous observation regime should be no harder from a statistical point of view than in any discrete observation regime and, consequently, the minimax rate of estimation 
should be no faster than $\varepsilon_{\Tlambda}$. Using this intuition
, \cite{D13} showed 
in Theorem 2.2  that, indeed, for the high-frequency regime, $T^{-s/(2s+1)}$ is a lower bound for the rate if $p\in [1,\infty)$, $\lambda \in \Lambda$ and $\mu\in B^s_{p,\infty}(A)$ for some compact intervals $\Lambda\subset (0,\infty)$ and $A\!\subset\! \R$. Inspection of her proof reveals that, if $\Lambda=\{\lambda\}$, an extra $\lambda$ term features naturally giving $\Tlambda^{-s/(2s+1)}$ and, moreover, the bound is valid for arbitrary frequencies without assuming $\Dn \to 0$ or even $\supnDn<\infty$. Furthermore, her ideas can readily be used when $p=\infty$ and $A=\R$ too to conclude that, 
for any $p,s,\alpha,B_1,B_2$ as above,
there exists a constant $\underline{C}'=\underline{C}_{p,s,\alpha}'(\Lambda,B_1,B_2)$ such that
\begin{equation*}
\liminf_{T} \varepsilon_{T}^{-1} \inf_{\tilde\nu_{T}}\sup_{\substack{ \nu=\lambda \mu:\, \lambda\in \Lambda \\ \mu\in \mathcal{F}_{p,s,\alpha}(B_1,B_2)
}}E_{P_{\nu}}\norm{\tilde{\nu}_{T}-\nu }_p \geq \underline{C}', 
\end{equation*}
where the infimum is over all possible estimators $\tilde\nu_{T}$ based on $n$ discrete $\Dn$-equispaced observations. Note that if $\Lambda=\{\lambda\}$, one can write $\Tlambda$ in place of $T$ in the last display and take the supremum simply over $\mu$; for such case,
it follows from the next theorem that the rate $\varepsilon_{\Tlambda}$ is attained by our estimator $\hat \nu_{\Tlambda}$ if $\supnDn<\infty$ 
and, thus,  $\hat \nu_{\Tlambda}$ is {minimax-optimal}. 
In the next section we modify it so that it attains the rate $\varepsilon_{T}$ uniformly over the compact functional class for $\nu$ (with, possibly, $\{\lambda\} \subsetneq \Lambda$) without knowledge of $s, \alpha$ and $\lambda$. We will require $\nu$ to satisfy slightly stronger tail assumptions that do not change the dimension of the functional space and, hence, the lower bound in the last display  remains valid. In view of the definition above, this will imply that $\varepsilon_{T}$ is the minimax rate of estimation for the given class and that the improved estimator is
{adaptive in the minimax sense}. 

\begin{thm} \label{thm:upperbound}
	
	Adopt the setting and notation introduced in Section \ref{sec:settinganddefs} with $\supnDn <\infty$. Assume that $\mu=\nu/\lambda$ satisfies Assumption \ref{ass1} for some $p\in [1,\infty]$ and $s>0$; if $p\in [1,2)$, let it satisfy Assumption \ref{ass2a} for some $\alpha>2/p-1$ and for both $q=1$ and $q=p$; and, if $p\in [2,\infty]$, assume instead that it satisfies Assumption \ref{ass2b} for some $\beta>1$. 
	Let $\nunhat$ be as in \eqref{eq:estimatornu} with $2^J=\twoJn$ given by \eqref{eq:2J}, 
	\begin{equation}\label{eq:Un}
	U=\Un= \Tlambda^{\vartheta_U}, \quad 0< \vartheta_U<\infty,
	\end{equation}
	\begin{equation}\label{eq:Hn}
	H=\Hn= \Tlambda^{\vartheta_H}, \quad \vartheta_H \in 
	\left( \frac{ 1}{\alpha }\frac{s}{2s+1}, \frac{ 1}{2/p-1}\frac{1+s-1/p}{2s+1}\right), \quad  \mbox{ if } p\in[1,2),\\
	\end{equation}
	and $H=\infty$ if $p\in[2,\infty]$. Then, for $\Tlambda$ large enough,
	\begin{align}\label{eq:upperbound}
	\E \norm{\nunhat\!-\nu}_p \! \leq  &   \lambda \!\left(2 e^{2\lambdasupnDn} 
	C^{(MS)} \sqrt{ \frac{\twoJn}{\Tlambda}}+ C^{(B)}\norm{\mu}_{B_{p,\infty}^s} 2^{-sJ_{\Tlambda} }\right)   \\
	& + \! \lambda\!\left(  1\!+\! \left(e^{3\lambdasupnDn}\!+\!1\right) \norm{\mu}_p \!+\!  C^{(R)}\lambda \supnDn e^{4\lambdasupnDn}  
	\!+\!\norm{  \mu}_{p,\alpha} \one_{[1,2)}(p)    \right) o\left(\varepsilon_{\Tlambda}\right)\!, \notag
	\end{align}
	where, for $\barK$ 
	as in Condition \ref{cond:2}, 
	$C^{(MS)}=C^{(MS)}_{p,\alpha}(\lambda, \mu,\barK)$ is defined in Remark \eqref{rem:ctsminimax}, 
	$C^{(B)}=C^{(B)}_s(K)$
	and $C^{(R)}=C^{(R)}_p( K)$.

\end{thm}

The first two terms on the right hand side of \eqref{eq:upperbound} correspond to the usual variance-bias bound. They arise from part of the linear stochastic term and the bias term and, in view of \eqref{eq:2J},  are of order $\varepsilon_{\Tlambda}$. Consequently, Theorem \ref{thm:upperbound} states that we can estimate the L{\'e}vy density at the minimax-optimal rate using a spectral estimator under minimal assumptions: if $\mu=\nu/\lambda$ satisfies Assumption \ref{ass1} for some $p\in [1,2)$ and $s>0$, then already in the standard i.i.d. sampling model Assumption \ref{ass2a} with $\alpha>2/p-1$ and $q=1$ is required to ensure the compactness of $\mathcal{F}_{p,s,\alpha}(B_1,B_2)$  needed 
to show that \eqref{eq:varepsilonm} is the minimax rate and, in such sense, it is necessary
; Assumption \ref{ass2b} is minimal and guarantees convergence of the uniform fluctuations of $\varphinhat$ about $\varphin$ on $|\xi|\leq (2\pi/3) \twoJn$; and, Assumption \ref{ass2a} with $\alpha>2/p-1$ and $q=p\in[1,2)$ allows us to circumvent the fact that the Hausdorff--Young inequality does not hold for the H{\"o}lder conjugate of $p\in[1,2)$ when controlling the non-linearities. The last condition is very mild and, in fact, implied by the first  for reasonable choices of $\mu$: 
let $\mu$ satisfy all the assumptions of Theorem \ref{thm:upperbound} for $p\in[1,2)$ 
and, moreover, assume it is such that, for some $\gamma,M\geq 0$, 
\begin{equation}\label{eq:ass2equiv}
\abs{{\mu}}(x) \lesssim \left(1+|x|\right)^{-\gamma} \qquad \mbox{for all } |x|\geq M;
\end{equation}
then, Assumption \ref{ass2a} implies $\gamma>\alpha+1/q$ necessarily, which is strongest for $q=1$. In turn, we note that, when $\mu$ satisfies these conditions, $\gamma>2/p$ is needed due to 
$\alpha>2/p-1$. 

Theorem \ref{thm:upperbound} assumes knowledge of $\lambda$ only through the dependency of  $J, U$ and $H$ on $\Tlambda$. These quantities  can be chosen to depend on $T$  
as we do in the next section and such change only affects the constants in \eqref{eq:upperbound}. 
Indeed, we write $J, U$ and $H$ in terms of $\Tlambda$ as it results in the multiplicative factor of $\lambda$ in the leading term in \eqref{eq:upperbound} that is natural due to 
$\nu=\lambda \mu$. We also note that the interval to which $\vartheta_H$ in \eqref{eq:Hn} belongs 
is never empty due to $\alpha>2/p-1$ and $s\leq 1+s-1/p$ for all $p\in[1,2)$. 

\begin{rem}\label{rem:ctsminimax}
	We have that $C^{(MS)}:=L_pM_p$, where
	\begin{align*}\label{eq:Lp}
	L_p&:=\left\{
	\begin{array}{ll}
	\left(\frac{ 2(2-p)}{\alpha p-(2-p)}\right)^{1/p-1/2} \norm{ \barK^2}_{1,\alpha}^{1/2}, & \mbox{ if } p\in[1,2),\\
	\norm{\barK}_2, & \mbox{ if } p=2,\\
	(p-1)^{1+1/p} \max\{ 2^{4+3/p}\norm{\barK}_2, 2^{1-1/p}\norm{\barK}_p  \}, & \mbox{ if } p\in(2,\infty),\\
	C_{\infty}= C_{\infty}(\phi), & \mbox{ if } p=\infty, 
	\end{array}	
	\right. \:\mbox{ and,}\\
	M_p&:=\left\{
	\begin{array}{ll}
	\!e^{\lambdasupnDn \left( \norm{ \mu }_{1,\alpha/2}-1\right) }\!\left(  \left(e^{\lambdasupnDn \norm{  \mu}_{1,\alpha}} \!-\!1\right)\!/\!\left(e^{\lambdasupnDn}\! -\!1\right) \right)^{\!1/2}\!, & \mbox{if } p\!\in\![1,2),\\
	\!\one_{[2,\infty)}(p)+\norm{\mu}_{p/2}^{1/2}\one_{(2,\infty)}(p)+1\vee \norm{\mu}_{\infty}^{1/2}\one_{ \{\infty\}}(p)
	, & \mbox{if } p\!\in\![2,\infty]. 
	\end{array}	
	\right.
	\end{align*}	
	Due to $\norm{\mu}_1\!=\!1$ and $\norm{\mu}_p\! \leq\! C \!\norm{\mu}_{B^s_{p,\infty}}$, 
	$\norm{\mu}_{p/2}^{1/2}\!\leq  \! \norm{\mu}_{p}^{\frac{p-2}{2(p-1)}} \!\leq  \!C_p'\norm{\mu}_{B^s_{p,\infty}}^{\frac{p-2}{2(p-1)}}\!\!<\!\infty\,\,\forall p\!\in\![2,\infty]$. 
	
\end{rem}
\subsubsection{Adaptation in the minimax sense}
\label{sec:adaptation}

We now proceed as briefly mentioned right before Theorem \ref{thm:upperbound} and achieve adaptation. In such theorem, when $p\in[1,2)$ the choice 
of $\Hn$ featuring in $\nunhat$ depends on the generally unknown parameter  $\alpha$ and also on $\lambda$. 
To avoid these limitations we may write $H$ in terms of $T$ instead of $\Tlambda$ and strengthen slightly the tail assumptions on $\mu$:  
assume it satisfies Assumption \ref{ass2a} for some $\alpha>2/p-1$ and $q=1$ and also
, e.g., for some $\alpha'>
3/p-3/2$ and $q=p$. 
Then, $
3/p-3/2>2/p-1$ for all $p\in [1,2)$, and we can choose 
\begin{equation*}
H=H_{T}= T^{\vartheta_H}, \qquad \vartheta_H\in 
\left(\frac{ 1}{
	3(1/p-1/2) }\frac{s}{2s+1}, \frac{ 1}{2(1/p-1/2)}\frac{1+s-1/p}{2s+1}\right).
\end{equation*}
If $\mu$ further satisfies \eqref{eq:ass2equiv} together with Assumptions \ref{ass1} and \ref{ass2b}, we have that, in view of the discussion right after \eqref{eq:ass2equiv},  $\gamma>
\max\{2/p,4/p-3/2\}$ 
necessarily
. Therefore, for such jump 
densities, only when $p\in [1,3/4)$ the extra assumptions on $\alpha'$ strengthen the tail condition and, for this case, it is only by adding $
2/p-3/2$. Similarly, to avoid the dependency of $U$ 
on the unknown constant $\lambda$, we may take $U=\UT$ with $\vartheta_U$ as in \eqref{eq:Un}.

A more delicate issue is circumventing the dependency on $s$ of both $2^J=2^{J_T}$ and $H_T$. To achieve this, we first write the latter in terms of the former: note that the $s$-dependent expression for 
$2^{J_T}$ in \eqref{eq:2J} together with the last display suggest that, for $p\in[1,2)$, we may choose
\begin{equation}\label{eq:Hn'}
H=H_{T,J}=\left(\!\sqrt{\frac{T}{2^J}}\:\right)^{\theta_H} 2^{J \theta_H'} = T^{\theta_H/2} 2^{J\left( \theta_H' - \theta_H/2 \right)}, 
\end{equation}
\begin{equation}\label{eq:thetaHthetaH'}
\mbox{where} \quad \theta_H\in 
\left(\frac{ 1}{3(1/p-1/2) }, \frac{ 1}{2(1/p-1/2)}\right)
\quad \mbox{ and } \quad
\theta_H'\in 
\left[0,\frac{ 1-1/p}{2(1/p-1/2)}\right]
;
\end{equation}
we remark that we will not be able to prove our results for all choices of  $ \theta_H$ and $ \theta_H'$ in these intervals and we will need to impose some extra restrictions on them. 
Hence, we are left with choosing $J=J_T$ appropriately without knowledge of $s$. For this, we adapt Lepski\u{i}'s method ---first introduced in \cite{L91}--- to our setting. Due to the fact that the upper bound in \eqref{eq:upperbound} is of the usual form in density estimation plus a negligible term, together with the fact that the integral operators we use 
are of arbitrary order, we proceed almost like in the modification of the classical method introduced in \cite{KNP12}: let $\Jmax\in\N$ be such that
\begin{equation}\label{eq:Jmax}
\frac{T}{\log^{\eta}(T\vee e)} \sim 2^{\Jmax} \leq \frac{T}{\log^{\eta}(T\vee e)} , \quad \mbox{where }\:  
\left\{\!
\begin{array}{ll}
\eta>2/\left(1\!-\!\theta_H\right), & \mbox{if } p=1,\\
\eta>0 & \mbox{if } p\!\in\!(1,\infty),\\
\eta>1, & \mbox{if } p=\infty,
\end{array}	
\right.
\end{equation}
and define $\J:=\{j\in \N: j\in[1,\Jmax]\}$ and 
\begin{equation}\label{eq:twopl}
\twop{l}:=2^l \left(1+ (l+\log T)\one_{\{\infty\}}(p)\right), \qquad l\in\J, p\in[1,\infty];
\end{equation}
then, writing $\nuThat(J)$ for the estimator defined pointwise in \eqref{eq:estimatornu} with $J\in \Nz$ and $ U=U_T$, and, if $p\in[1,2)$, $H=H_{T,J}$ whilst, if $p\in[2,\infty]$, $H=\infty$, we take, for some $\tau>0$ to be determined,
\begin{equation}\label{eq:JThat}
\JThat=\JThat(p,\tau):=\min \left\{ j\in \J\!: \norm{ \nuThat(j) -\nuThat(l)}_p \leq \tau  \sqrt{\frac{\twop{l}}{T}} \:\: 
\forall l >j, \:l\in \J\right\}
\end{equation}
and $\JThat=\Jmax$ if the set is empty
.  
The bound on $\eta$ in \eqref{eq:Jmax} for $p=\infty$ is nearly as in the standard problem of density estimation from i.i.d. random variables, where $\eta=1$. The bound for $p=1$ and $p\in[2,\infty)$ is necessary and particular to the problem at hand: when $p\in[1,1+\epsilon]$ for some $\epsilon\in[0,1)$, the concentration of the non-linear term dominates that of the linear and  $p=1$ is worst, requiring $\eta>2/\left(1-\theta_H\right)$, where we note that $2/\left(1-\theta_H\right)>6$ for $p=1$; when $p\in[2,\infty)$, the condition $\eta>0$ arises from the fact that in the proofs we need that $2^{\Jmax} \leq c(\lambda) \lambda T$, where $c(\lambda)<1$ and we recall that $\lambda$ is unknown. This condition on $2^{\Jmax}$ is intuitive, as it is effectively stating that the number of parameters we can estimate should be no larger than the effective sample size. The bound on $\eta$ for $p\in(1,2)$ is not necessary and can be replaced by assuming $\theta_H'>0$.
%
%
%
%
%


In order to prove that the data-driven estimator $\nuThat(\JThat)$ is adaptive in the minimax sense (cf. discussion before Theorem \ref{thm:upperbound}
), we need to restrict estimation to a compact subclass of $\mathcal{F}_{p,s,\alpha}(B_1,B_2)$.  For any $p\in[1,\infty]$, $s,\alpha, \alpha'>0$, $\beta>1$, $0<\underline{\lambda} \leq \bar{\lambda}<\infty$  
and $B_1, B_2,B_3, B_4\geq 1$, we define $ \mathcal{V}= \mathcal{V}_{p,s,\alpha, \alpha',\beta}\left(\underline{\lambda}, \bar{\lambda}, (B_i)_{i=1}^4\right)$ as 
\begin{align}\label{eq:mathcalV}
& \mathcal{V}:= \left\{\nu=\lambda \mu
\!: \lambda
\in [\underline{\lambda},\bar{\lambda}], 
\mu \in \mathcal{F}_{p,s,\alpha}(B_1,B_2) \mbox{ and, if } p\in[1,2), 
\,\norm{\mu}_{p,\alpha'}\leq B_3, \right. \notag \\
& \qquad \qquad \: \qquad \qquad \qquad \quad \:\mbox{whilst, if } p\in[2,\infty], \, \left. \! \norm{ \log^{\beta} (\max\{|\cdot|,e\}) \mu }_1<B_4 
\right\}.
\end{align}	 
Requiring $\lambda\leq \bar{\lambda}$ guarantees that the compound Poisson process does not have arbitrarily large activity, and $\lambda\geq \underline{\lambda}>0$ that, for $T$ large enough, the  effective sample size is such that $T_{\lambda}\geq 1$ for all $\nu\in\mathcal{V}$, as otherwise it would be arbitrarily small.  Lower bounding $B_1,B_2, B_3, B_4$ by $1$ is natural  due to $\int_{\R}\mu=1$. 

\begin{thm} \label{thm:adaptation}
	
	Adopt the setting and notation introduced in Section \ref{sec:settinganddefs} with $\supnDn <\infty$.  Assume that $\nu\in \mathcal{V}_{p,s,\alpha, \alpha',\beta}\left(\underline{\lambda}, \bar{\lambda}, (B_i)_{i=1}^4\right)$ for some $p\in[1,\infty]$, $s>0$, $\alpha>2/p-1$, $\alpha'>3/p-3/2$, $\beta>1$, $0<\underline{\lambda} \leq \bar{\lambda}<\infty$ and $B_1,B_2, B_4,B_3\geq 1$. 
	Let $\nuThat(\JThat)$ be $\nuThat(J)$ as defined right before \eqref{eq:JThat} with $J=\JThat$ and, 
	for $p\in(1,2)$, suppose that $\theta_H$ and $\theta_H'$ in \eqref{eq:thetaHthetaH'} further satisfy $\theta_H-2 \theta_H' -1\leq 0$ and $\theta_H'<\frac{1-1/p}{2(1/p-1/2)}$.
	Then, if  
	\begin{equation}\label{eq:tau}
	\tau > 
	\frac{1}{\sqrt{\underline\lambda}} 
	+  \max\left\{1,\left(C_2^{(MS)}\right)^2\right\},
	\end{equation} 
	we have that, for $T$ large enough,
	\begin{align}\label{eq:adaptation}
	\sup_{ \nu\in  \mathcal{V}}\E\left[\norm{\nuThat(\JThat) - \nu}_p \right] &\leq 2\left(\tau\sqrt{\bar\lambda} +2 { \bar\lambda}  e^{2\lambdabarsupnDn}   C_1^{(MS)}+1  \right) \left( \bar\lambda  C^{(B)} B_2 \right)^{\frac{1}{2s+1}} \varepsilon_{\Tlambdaunderline} \notag \\
	& \quad +   C^{(o)} o\left(\varepsilon_{\Tlambdaunderline}\right),
	\end{align}
	where 
	$ C_1^{(MS)}, C_2^{(MS)}$ are defined in Remark \ref{rem:cts}, $C^{(B)}$ is as in Theorem \ref{thm:upperbound} and $ C^{(o)}$ is defined in \eqref{eq:Co}. 

\end{thm}

In comparison to Theorem \ref{thm:upperbound}, Theorem \ref{thm:adaptation} states that, for $\tau$ large enough, the purely data-driven estimator $\nuThat(\JThat)$ converges to $\nu$ at least as fast as an estimator that knew the regularity $s$ in the model with smallest effective sample size within $\mathcal{V}$. 
For $p\in[1,2)$, a valid choice of $H$ satisfying the additional conditions in Theorem \ref{thm:adaptation} is obtained by taking, e.g.,   $\theta_H=\frac{3}{8}\frac{1}{1/p-1/2}$ and $\theta_H'=\theta_H (1-1/p)$. 

\begin{rem}\label{rem:cts}
	
	We have that $ C_1^{(MS)}:=L_p \bar M_p$, where for $L_p$ and $C'$  as in Remark \ref{rem:ctsminimax},
	\begin{equation}
	\!\!\!\bar M_p\!:=\!\left\{\!
	\begin{array}{ll}
	\!e^{\lambdabarsupnDn \left( B_2-1\right) }\!\left(  \left(e^{\lambdabarsupnDn B_2} \!-\!1\right)\!/\!\left(e^{\lambdabarsupnDn}\! -\!1\right) \right)^{\!1/2}\!, & \mbox{if } p\!\in\![1,2),\\
	\!\one_{[2,\infty)}(p)+C_p'B_1^{\frac{p-2}{2(p-1)}}\one_{(2,\infty)}(p)+1\vee (C_p'B_1^{1/2})\one_{ \{\infty\}}(p)
	, & \mbox{if } p\!\in\![2,\infty], 
	\end{array}	
	\right.  \label{eq:defbarMp} \\
	\end{equation}	
	and
	\begin{align}\label{eq:C2MS}
	C_2^{(MS)}\!:=\!4 
	e^{2\lambdabarsupnDn} \!\left(\frac{3}{2}   \sqrt{\bar\lambda}C_1^{(MS)}\! + \!   \sqrt{ 2 \bar\lambda C B_1\! \norm{\barK}_{2\left(1 \!-\! \frac{1}{1+p}\right)}^2 } \!+\!\frac{14 \norm{\barK}_p   }{3}\sqrt{2\!\left(1
		\!+\! \lambdabarsupnDn\right)\!}\right)\!. 
	\end{align}
\end{rem}

\begin{rem}\label{rem:tau}
	
	Depending on the value of $p$, the lower bound for $\tau$ in \eqref{eq:tau} depends on most or all of the following unknown quantities: on $\underline{\lambda}$ explicitly; and, on $\bar \lambda$, $B_1$ and $B_2$  implicitly through $C_1^{(MS)}$ and $C_2^{(MS)}$ defined in Remark \ref{rem:cts}. Thus, a question 
	not addressed by Theorem \ref{thm:adaptation} is how to choose a minimal $\tau$ to compute $\JThat$ when implementing $\nuThat$ in practice (see Section \ref{sec:implementation} for more details on implementation); this can be achieved by estimating each of the quantities as we conclude at the end of this remark. Prior to it, we discuss their estimation in decreasing order of difficulty. 
	
	The quantity $B_1$ appears explicitly and multiplied by $\bar{\lambda}$ in $C_2^{(MS)}$ and, for $p\in(2,\infty]$, implicitly and multiplied by $\bar{\lambda}$ too through the term
	\begin{equation*}
	\sqrt{\bar\lambda}C_1^{(MS)} \!=\! L_{p}\!\left(\left(\sqrt{\bar\lambda} \!+\!\bar{\lambda}^{\frac{1}{2(p-1)}}C_p'(\bar{\lambda} B_1)^{\frac{p-2}{2(p-1)}}\right)\!\one_{(2,\infty)}(p)\!+\!\sqrt{\bar\lambda}\vee\left(C_p'\left({\bar\lambda} B_1\right)^{\!\frac{1}{2}}\right) \!\one_{ \{\infty\}}(p)\right)\!.
	\end{equation*}
	Note that we arrived at $C_2^{(MS)}$ and $C_1^{(MS)}$ using the inequalities at the end of Remark \ref{rem:ctsminimax}, effectively using $\norm{\nu}_{p} \leq \bar{\lambda} C B_1$ and $\norm{\nu}_{p}^{\frac{p-2}{2(p-1)}}\leq C_p'(\bar{\lambda} B_1)^{\frac{p-2}{2(p-1)}}$
	. Therefore, we may focus on estimating $\norm{\nu}_{p}$ 
	adapting the ideas in \cite{GN09}, where the case $p=\infty$ was considered: for any $p\in[1,\infty]$, we estimate it by $\norm{\nuThat(J')}_{p}$, where $J'$ be such that 
	\begin{equation*}
	2^{J'} \sim \frac{T}{\log^{\eta'}(T\vee e)} \qquad \mbox{for}  \quad \eta'=\eta'(p)>\eta+1 \quad \mbox{and} \quad  \eta=\eta(p) \mbox{ as in \eqref{eq:Jmax}.}
	\end{equation*}
	From Proposition \ref{prop:mainstochasticconcentration}, Lemmata \ref{lem:remainderconcentration} and \ref{lem:OmegantildesubseteqOmegan:b}, and proceeding similarly to the proof of Lemma \ref{lem:JhatT}, a concentration inequality for $\norm{\nuThat(J') - \nu}_{p}$  follows immediately and, for $T$ large enough,  concentration of  $\norm{\nuThat(J') }_{p}$ around the unknown $\norm{\nu}_{p}$ can be inferred. By the same reasoning, the remaining instances of $\bar \lambda$ together with $\underline{\lambda}$ in \eqref{eq:tau} can be 
	substituted by $\lambdanhat$ or the other estimators 
	in Remark \ref{rem:n0=0} which, as mentioned therein, also satisfy a central limit theorem and, moreover, concentration inequalities for all can be shown. Lastly, $B_2$ is only present in the case $p\in[1,2)$ and it depends on moments of $\mu$; if no upper bound for it is available, it can be immediately bounded as in \cite{GN08} 
	by making stronger assumptions such as $\mu$ being compactly supported on, e.g., $[-1/2,1/2]$. 

	Due to the exponential-concentration results for the estimates of $\norm{\nu}_{p}$ and $\lambda$ we just discussed, the proof of Theorem \ref{thm:adaptation} goes through with the resulting random choice of $\tau$ too (multiplied by a constant larger than $1$ due to the strict inequality in \eqref{eq:tau}); we presented it with the bounds for $\norm{\nu}_{p}$ and $\lambda$ known 
	for simplicity. We remark that the lower bound in \eqref{eq:tau} does not depend on $B_3$ or $B_4$ and, thus, our data-driven estimator does not require knowledge of these. It follows that when $p\in[2,\infty]$ we need no knowledge of the tails of $\mu$: for $p=2$, it agrees with the rest  of existing literature on adaptive estimation of  L{\'e}vy processes observed discretely in the low-frequency regime; and, for  $p=\infty$, which is the most statistically-relevant case, it means that our estimator is the first purely data-driven one optimally and uniformly estimating $\nu$ in the low-frequency regime.  The case $p=1$ is the natural framework for probability densities but is less relevant in practice, similarly to $p\in(1,2)$ and $p\in(2,\infty)$. 
	We include them because they have not been studied in existing literature using the spectral approach and, therefore, the proofs have theoretical value. In particular, they allow us to be the first to report the  phenomenon of the concentration of the non-linearities dominating over that of the linearities when $p=1$. Moreover, for $p\in[1,2)$ they can be used to derive optimal Bayesian contraction rates via the concentration-of-measure approach developed in \cite{GN11}.

\end{rem}

%
%

\subsubsection{Unification of existing literature on the high- and low-frequency regimes and resulting insights
}\label{sec:unification}

Whilst our starting point \eqref{eq:identityLK2} is valid for any $\Dn>0$, 
if $\lambdaDn<\pi$ then $\FT \nu=\frac{1}{\Dn} \log(e^{\lambdaDn}\varphi)$, where $\log$ is the principal branch of the complex logarithm. Note, furthermore, that if $\lambdaDn<\log 2$, then $\abs{e^{\lambdaDn}\varphi-1}<1$ and $\log$ admits the usual power series expansion. Consequently, taking the Fourier inverse on both sides and using that 
\begin{equation*}
e^{\lambdaDn}\varphi\!-\!1\!=\!\FTT{e^{\lambdaDn}P\!-\!\delta_{0}}\!=\!( e^{\lambdaDn}\!-\!1)\FT\! \Pn^{J)}\!, \,\,\, \mbox{where}\,\, \Pn^{J)}\!:=\! \frac{ P\!-\!e^{-\lambdaDn}\delta_{0}}{ 1-e^{-\lambdaDn}}\!=\!\frac{\sum_{i =1}^{\infty}\frac{\Dnsuper{i} \nu^{\ast i}}{i!}}{e^{\lambdaDn}-1}
\end{equation*}
is the distribution of $Y_{\Dn}|_{Y_{\Dn}\neq 0}$, we obtain that, under very mild assumptions on $\nu$, 
\begin{equation}\label{eq:repnu}
\frac{1}{\Dn} \FII{\Log(e^{\lambdaDn}\FT P)}=\nu =  
\sum_{i =1}^{\infty} \frac{(-1)^{i-1} (e^{\lambdaDn}-1)^i}{\Dn i} \big( \!\Pn^{J)}\!\big)^{\ast i}.  
\end{equation}
Moreover, similar arguments, including
\begin{equation*}
\varphinhattilde\!-1\!=\! \FTT{e^{\lambdanhatDn}\Pnhat\!-\delta_{0}}
\!=\!(e^{\lambdanhatDn}\! -\!1)\FT\widehat{\Pn}_{n}^{J)}\!, \,\,\, \mbox{where} \,\, \widehat \Pn_{n}^{J)}\!\!:=\!\frac{\sum_{Z_i\neq 0} \delta_{Z_i}}{\sum_{i=1}^n\! \one_{Z_i\neq0}}\!=\!\frac{1}{n\!-\!n_0}\!\sum_{i=1}^{n-n_0
}\! \delta_{Z_i^{J)}}
\end{equation*}
is the empirical distribution of the non-zero increments $Z_1^{J)}, \ldots, Z_{n-n_0}^{J)}$, and Plancherel theorem imply that, for $\lambda\Dn<\log 2$ and with $\Pr$-probability approaching $1$ as $n\to \infty$, 
\begin{equation}\label{eq:repnuhat}
\frac{1}{2\pi}\tildeKJ{ \frac{\Log \varphinhattilde}{\Dn} }
= \sum_{i =1}^{\infty} \frac{(-1)^{i-1} (e^{\lambdanhatDn}-1)^i}{\Dn i}\KJ{\left( \widehat\Pn^{J)}_{n}\right)^{\ast i}}.
\end{equation}
Using the left side to construct estimators is generally referred to as the spectral approach, which, until now, has been considered as a different approach from building estimators using the right side (cf. p. 673 in \cite{VEGS07}
, p. 3964 in \cite{D13} 
and p. 165 in \cite{CDGC14}
). These derivations together with our unified treatment of integral operators show that they are in fact equal and provide a unification of all existing literature on non-Bayesian estimation of discretely observed compound Poisson processes. 

Several insights follow. Clearly, the spectral approach uses all the small-time approximations corresponding to truncations of \eqref{eq:repnu} up to $i\leq I$, $I\in \N$, to the extent that it gives a well-defined generalisation of the representation of $\nu$ in terms of $\Pn^{J)}$ for any low-frequency setting. Indeed, ignoring all the adaptation procedures, the estimator of $\nu$ in Comte et al. \cite{CDGC14} is a small-time approximation to that in van Es et al. \cite{VEGS07} which\footnote{They assumed $\lambda$ known so did not use $\lambdanhat$, but the corresponding expressions follow immediately.}, in turn, is a special case of our estimator taking a convolution operator with finite order; assuming $\nu$ has a density, the estimators for its distribution in Coca \cite{C18} and Buchmann and Gr{\"u}bel \cite{BG03} are the integral of \cite{VEGS07} and its `non-regularised' version---which explains all the differences in their results (cf. Chapter 2 in Coca \cite{C17} for more details)---; and, the estimator of $\nu$ in Duval \cite{D13} is a small-time approximation of a particular case of our estimator taking a compactly supported wavelet projection operator. As a consequence of the last link, we conjecture that a thresholding strategy to adaptation  using the spectral approach should work.

More refined insights are the following. In \cite{D13} and \cite{CDGC14}, and in \cite{BG03}, they show that their estimators attain optimal rates assuming $\Dnsuper{2I+1}n\to 0$ and $\lambdaDn<\log 2$, respectively. Thus, the unification above gives a transparent reason as to why our spectral estimator is optimal in and robust to the high- and low- frequency regimes. In fact, it implies the same properties for the estimators in \cite{VEGS07} and \cite{C18} too. Indeed, in  \cite{C17} we argued that in the high-frequency regime, the estimator in  \cite{C18} satisfies a functional central limit theorem with the same limiting Gaussian process as in classical Donsker's theorem; the latter is optimal for the i.i.d. sampling model. We can also link the usual estimators in such model with all the above estimators in the model herein for $\Dn \to 0$: 
from \eqref{eq:repnuhat} it follows that, with $\Pr$-probability approaching $1$ as $n\to \infty$, 
\begin{equation*}
\frac{1}{2\pi}\tildeKJ{\frac{\Log \varphinhattilde}{\Dn} }
= \lambdanhat  \KJ{\widehat \Pn_{n}^{J)}}	
+  O(\Dn), 
\end{equation*}
and we conclude that, as $\Dn \to 0$ and $n\to \infty$, the leading term in all the abovementioned estimators 
is $\lambdanhat$\footnote{Or, in line with the previous footnote, $\lambda$ instead of $\lambdanhat$ in \cite{VEGS07} and\cite{BG03}.} times the usual i.i.d. integral operator-based 
estimator. Consequently, in the high-frequency regime spectral estimators and their small-time approximations should recover all aspects of the i.i.d. sampling model. This  agrees with the fact that the first line in \eqref{eq:upperbound} corresponds to the usual variance-bias bound and, dividing by $\lambda$ and letting $\supnDn \to 0$, we obtain the constants of i.i.d. density estimation (cf. Remark 5.1.6 in \cite{GN16}); the resulting rate is the minimax one for a sample of size $\Tlambda$, which is the natural quantity asymptotically as $\Tlambda$ is the expected number of jumps $Y$ gives in $[0,T]$. 

\subsubsection{Robustness to other sampling schemes and potential implications 
}\label{sec:discussionDn}

Theorems \ref{thm:upperbound} and \ref{thm:adaptation} assume $\supnDn<\infty$ in order to guarantee convergence at the optimal rates $\varepsilon_{\Tlambda}$ and $\varepsilon_{T_{\underline{\lambda}}}$, respectively, for the high- and low-frequency regimes simultaneously. However, convergence at slower rates still holds if $\Dn\to \infty$ slowly enough: if $\Dn\sim c \log \log T$ for some $c>0$ only a logarithmic penalty is paid, whilst if $\Dn\sim c \log T$ for some $c>0$ the deterioration is polynomial for $c$ small enough and consistency no longer holds for $c$ beyond a certain threshold; for $p=2$ this is at most $c\leq 1/(4 \lambda)$, 
as recently reported in Duval and Kappus \cite{DK18} through a different methodology ($\lambda=1$ known therein). This `slow macroscopic  regime' has not been studied in related literature and, thus, we ignore what is the optimal critical threshold $c$ to ensure consistency or whether the resulting slower rates are suboptimal. In a  `fast macroscopic  regime' in which  $\Dn\sim  T^{-c}$ for some $c\in(0,1)$ our estimator is not consistent, which agrees with the fact showed by  Duval \cite{D14} that in such setting it is not possible to estimate $\nu$ nonparametrically. 

A potential implication that has not yet been reported is the following. In our calculations, the quantity $\Dn$ is (nearly) always multiplied by $\lambda$ and this allows us to write the rates of converge in terms of the natural quantities $\Tlambda$ and $T_{\underline{\lambda}}$. Therefore, an alternative view of the properties discussed in the previous paragraph is that, if $\supnDn<\infty$ or even $\Dn\to \infty$ sufficiently slowly, consistency is still guaranteed by our results when $\lambda=\lambda (T)\to \infty$ slowly enough. In addition, note that, assuming $\Dnsuper{2}n\to 0$, Duval and Mariucci \cite{DM17} built a non-adaptive density estimator for general L{\'e}vy processes exploiting their construction as limits of compound Poisson processes: using the  increments with absolute value larger than $\epsilon_T\to 0$, they estimate the L{\'e}vy density by estimating its restriction to $[-\epsilon_T, \epsilon_T]^{\setcomplement}$ regarding it a compound Poisson processes with intensity $\lambda=\lambda(\epsilon_T)\to \infty$. Hence, these remarks suggest that a similar strategy may allow to extend our results to construct robust estimators that are optimal in losses other than $\Ltwo$ for more general classes of L{\'e}vy processes without polynomial-tail assumptions. 

\subsection{Control of the empirical characteristic function for arbitrary frequencies}\label{sec:chfn}

Recall that, besides $\lambdanhat$, the source of randomness in $\nunhat$ defined in \eqref{eq:estimatornu} 
is the empirical characteristic function $\varphinhat(\xi)$ restricted to $|\xi|\leq \Xin$. Indeed, in the proofs of Theorems \ref{thm:upperbound} and \ref{thm:adaptation}, we  need to control $\sup_{|\xi|\leq  \Xin } |\varphinhat(\xi) - \varphin(\xi)|$ appropriately and the following theorem, which is of independent interest, contains the precise results we need.

\begin{thm}\label{thm:varphivarphin}
	Adopt the setting and notation introduced in Section \ref{sec:settinganddefs}. 
	Let $\varphin$ and $\varphinhat$ be as in \eqref{eq:repvarphiintro} and \eqref{eq:varphinhat} with $k=0$, respectively. Then, for any $n\in \N$ and $\Xi'\geq 0$,
	\begin{enumerate}		[ref=\thethm \,(\alph*)]
		
		\item \label{thm:varphivarphin:a}
		if $a:=\int_{\R} \log^\beta (\max\{|x|,e\}) \mu(dx)<\infty$ for some $\beta>1$, we have that for any 
		$\delta>0$,
		\begin{equation*}
		\E \sup_{|\xi|\leq \Xi'} n|\varphinhat(\xi) - \varphin(\xi)|  \leq C  \sqrt{\Tlambda \log^{1+2\delta}(e+\Xi')},
		\end{equation*}
		where, for some  universal constants $b,C_1, C_2>0$,
		\begin{equation*}
		C:= \left(C_1 +	C_2 \left( {  \left(e^{a b^\beta \lambdaDn}-1\right)}/{\left(b^\beta \left(e^{\lambdaDn}-1\right)\right)} \right)^{\frac{1}{2\beta}}\right) \sqrt{\min\{1/(\lambdaDn), 1\}},
		\end{equation*}
		and we remark that the quantity $a$ does not depend on $\lambda$ or $\Dn$;
		
		\item \label{thm:varphivarphin:b}
		we have that for all  $t\geq 0$
		\begin{align*}
		\Pr&\left(\sup_{|\xi|\leq \Xi'}n|\varphinhat(\xi) - \varphin(\xi)| \geq 2 E\sup_{|\xi|\leq \Xi'}n|\varphinhat(\xi) - \varphin(\xi)| + t \right) \\
		&\leq 2  \exp\left(-\frac{ t^2}{16E\sup_{|\xi|\leq \Xi'}n|\varphinhat(\xi) - \varphin(\xi)| + 32\Tlambda+ 8t/3 }\right);  \quad \mbox{and},
		\end{align*}
		
		\item \label{thm:varphivarphin:c}
		suppose Assumption \ref{ass2b} is satisfied for some $\beta>1$, $\log^{1+2\delta}(e+ \Xi') \leq c \Tlambda$ and $e+\Xi' \geq 2^{\frac{1}{4L(1/(2\sqrt{c}) -2C)^2}}$ for some $\delta >0$ and $0<c< (4 C)^{-2}$, where  $C$ is as in part (a) and $L=3/464$; it follows that if $\sqrt{\log^{1+2\delta}(e+\Xi') /(c \Tlambda) } \leq  r\leq 1$, 	
		\begin{equation}\label{eq:concentrationsupr}
		\Pr\left(\sup_{|\xi|\leq \Xi'}n|\varphinhat(\xi) - \varphin(\xi)|  > \Tlambda r \right) \leq  e^{- L \Tlambda r^2  } .
		\end{equation}

	\end{enumerate}
	
\end{thm}

Notice that $C$ defined in Theorem \ref{thm:varphivarphin:a} depends on $\Dn$. Under the assumption that $\supnDn <\infty$, 
and due to $e^x-1\geq x$ for any $x\geq 0$ together with $(e^{c  x} -1)/ x$ being increasing for any $c\geq 0$ fixed, we have that
\begin{equation}\label{eq:CleqbarC}
C\leq \bar C := \left(C_1 +	C_2 \left( {  \left(e^{a b^\beta \lambda \supnDn  }-1\right)}/{\left(b^\beta \lambda \supnDn\right) } \right)^{\frac{1}{2\beta}}\right).
\end{equation}
Consequently, Theorem \ref{thm:varphivarphin:a} and Theorem \ref{thm:varphivarphin:c}  hold for $\bar C$ in place of $C$, and from now on we interpret them as if they were written in terms of $\bar C$. Then, in the high-frequency regime and for the compound Poisson setting, Theorem \ref{thm:varphivarphin:a} crucially improves upon existing related results (cf. Theorem 1 in \cite{KR10}, Theorem 4.1 in \cite{NR09} and Theorem 3.1 in \cite{C18}), which guarantee that
\begin{equation*}
\E \sup_{|\xi|\leq \Xi'} n|\varphinhat(\xi) - \varphin(\xi)|  \lesssim \sqrt{n \log^{1+2\delta}(e+\Xi')}.
\end{equation*}
The fact that this bound can be improved when $\Dn\to 0$ 
should be intuitive: as the  sample size  increases, the expected number of jumps the compound Poisson process gives between every discrete observation decreases, and the estimation problem becomes easier. Theorem \ref{thm:varphivarphin:a} makes this precise by stating  that the order of the supremum improves  at least by a factor of $\sqrt{\Dn}$ as $\Dn\to 0$ compared to that predicted by the  usual maximal inequalities; such factor is natural once one notices that the variance satisfies
\begin{equation*}
\E \left(n|\varphinhat(\xi) - \varphin(\xi)|\right)^2 \leq 2 \min\{1,4\lambdaDn\} , \qquad \xi \in \R,
\end{equation*}
which is shown in the course of the proof of Theorem \ref{thm:varphivarphin:b}. To take advantage of this acceleration of the variance when proving concentration of $\sup_{|\xi|\leq \Xi'} n|\varphinhat(\xi) - \varphin(\xi)| $ around its expectation, we use Talagrand's inequality. This results in the term $32\Tlambda$ on the right hand side of Theorem \ref{thm:varphivarphin:b}, which would otherwise be of the worse order $n$. The improvement of this term is, in turn, necessary to obtain Theorem \ref{thm:varphivarphin:c} which, compared to the usual concentration inequalities is, essentially, optimal. To the best of our knowledge, the three results in Theorem \ref{thm:varphivarphin} are new to the literature, and they are necessary to prove  robustness of spectral estimators to the low- and high-frequency regimes without tail assumptions and their optimality in such setting. They are also crucial to adapt Lepski\u{i}'s method to our problem in order to show Theorem \ref{thm:adaptation}.

Recall that estimator $\nunhat$ has non-trivial values in $\Omegan$ defined in \eqref{eq:defOmegan}. Then, we will have to control $\sup_{|\xi|\leq \Xi'} n|\varphinhat(\xi) - \varphin(\xi)| $, where $\Xi' = \Xin$ and $2^J=\twoJn=o(\Tlambda)$ is as in \eqref{eq:2J}. Furthermore, when applying Lepski\u{i}'s method we will need to control this quantity for $2^J\leq 2^{J_{\max}}\leq T$.  This follows immediately from Theorem \ref{thm:varphivarphin:c}.

\begin{corol}\label{corol:varphivarphin}
	Adopt the setting and notation introduced in Section \ref{sec:settinganddefs} with $\supnDn <\infty$. 
	Then, if Assumption \ref{ass2b} is satisfied for some $\beta>1$, $J=J(\Tlambda)\to \infty$ with $2^J=O(T_\lambda)$ and $\sqrt{\log^{1+2\delta}(e+2^J) / \Tlambda } \leq  r\leq 1$ for some $\delta>0$
	, we have that, for $L$ as in Theorem \ref{thm:varphivarphin:c} and $\Tlambda$ large enough,
	\begin{equation*}
	\Pr\left(\sup_{|\xi|\leq\Xin}n|\varphinhat(\xi) - \varphin(\xi)|  > \Tlambda r \right) \leq  e^{- L \Tlambda r^2  } .
	\end{equation*}	
\end{corol}

\section{Proofs}
\label{sec:proofs}

In Section \ref{sec:proofs:prepresults} we give some preparatory results to show Theorems  \ref{thm:upperbound} and  \ref{thm:adaptation}, which are proved in Sections  \ref{sec:proofs:upperbound} and \ref{sec:proofs:adaptation}. Lastly, we show Theorem \ref{thm:varphivarphin} in Section \ref{sec:proofs:varphivarphin}. Throughout we adopt the setting and notation introduced in Section \ref{sec:settinganddefs}.

\subsection{Preparatory results}\label{sec:proofs:prepresults}

Let us introduce some additional notation: recalling that the meaning of $\rightbarU{z}$ was introduced right after \eqref{eq:estimatornu}, we define, for any complex-valued function $f$ and any $U\geq 0$,  
\begin{equation*}
\rightbarHU{f}:= \left\{\begin{array}{ll}
\rightbarU{f}\one_{[-H,H]}, & \mbox{if } H\geq 0,\\
\rightbarU{f}, & \mbox{if } H=\infty,
\end{array} 
\right.
\qquad \mbox{and} \qquad \rightbarH{f}:=\rightbar{f}{H}{\infty}; 
\end{equation*}
then, in view of \eqref{eq:estimatornu}, on $\Omegan$ we write $\hat{\nu}=(2\pi)^{-1}\!\rightbarH{\tildeKJ{ \rightbarU{\left( \Dnsuper{-1}\Log \varphinhattilde\right)}}}$.
Until next section we keep writing $\hat{\nu}$ instead of $\nunhat$ or $\nuThat$. 
To prove both Theorem  \ref{thm:upperbound} and  \ref{thm:adaptation}, we split $\hat{\nu}-\nu$ into a stochastic and a bias term, and, due to the non-linear nature of the problem and to the need to set $\hat{\nu}\equiv 0$ on $\Omegan^{\!\!\!\!\setcomplement}$ defined in \eqref{eq:defOmegan}, two non-standard terms appear. 
To follow this strategy we first introduce the sets 
\begin{equation*}\label{eq:defOmegatilden}
\Omegantilde\left(2^J, r\right)\!:=\! \left\{ \omega \!\in\! \Omega\!: \! \sup_{|\xi|\leq \Xin}\!n \left|\varphinhat(\xi)\! -\! \varphin(\xi)\right| \!\leq\! \Tlambda r \, \, \mbox{ and } \,\,n\left|e^{\lambdanhatDn} \!-\! e^{\lambdaDn}\right| \!\leq \! e^{\lambdaDn} \!\Tlambda r \right\}\!.
\end{equation*}

\begin{lem}\label{lem:OmegantildesubseteqOmegan}
	We have that,
	\begin{enumerate}[ref=\thethm \,(\alph*)]
		\item for any $2^J,r\geq 0$ and $n\in \N$,
		\begin{equation*}
		\Omegantilde\!\left(2^J\!, r\right)\!\subseteq\! \left\{ \omega \!\in\!\Omega:\! \sup_{|\xi|\leq \Xin}\!n \left|\varphinhattilde(\xi)\!-\!\varphintilde(\xi)\right| \leq 2 e^{\lambdaDn}\Tlambda r \right\} \,\,\, \mbox{and} \,\,\,\, \Omegantilde\left(2^J\!, r\right)\!\subseteq \Omegan;
		\end{equation*}\label{lem:OmegantildesubseteqOmegan:a}
		
		\item  and, if Assumption \ref{ass2b} is satisfied for some $\beta>1$, $J=J(\Tlambda)\to \infty$ with $2^J=O(\Tlambda)$, 
		$\sqrt{\log^{1+2\delta}(e+2^J) / \Tlambda } \leq r\leq 1$ for some $\delta>0$, and $\supnDn <\infty$, we have that, for $\widetilde{L}:=Le^{-3\lambda \supnDn}$, with $L$ as in Theorem \ref{thm:varphivarphin:c}, and any $\Tlambda$ large enough,
		\begin{equation*}
		\Pr\left(\Omegantildecomplement\left(2^J, r\right)\right) \leq  3 e^{- \widetilde{L} \Tlambda r^2  } .
		\end{equation*}	\label{lem:OmegantildesubseteqOmegan:b}
	\end{enumerate}	
\end{lem}

\begin{proof}
	
	\begin{enumerate}
		\item	The first inclusion of this part of the lemma follows immediately by the definitions of $\varphinhattilde$ and $\varphintilde$ because, in view of them,
		\begin{equation*}
		\varphinhattilde-\varphintilde= \left( e^{\lambdanhatDn} -e^{\lambdaDn}\right) \varphinhat + e^{\lambdaDn} \left( \varphinhat -\varphin\right),
		\end{equation*}
		where $\left|\varphinhat\right|\leq 1$. For the second inclusion, we start by showing that, for any $2^J,r\geq 0$ and $n\in \N$,
		\begin{equation}\label{eq:inclusioncharacteristicfunctions}
		\left\{  \omega \in \Omega: \sup_{|\xi|\leq \Xin} n \left|\varphinhat(\xi)- \varphin(\xi)\right| \leq  \Tlambda r \right\} \subseteq \left\{  \omega \in \Omega: \inf_{ |\xi| \leq \Xin } \left|\varphinhat(\xi)\right| >0\right\}
		\end{equation}
		Trivially, for any $2^J,r\geq 0$ and $n\in \N$, the set on the right hand side contains
		\begin{equation*}
		\left\{ \omega \in \Omega:  \inf_{ |\xi| \leq \Xin } |\varphinhat(\xi)| \geq e^{-2 \lambdaDn}-\lambdaDn r \right\}.
		\end{equation*}
		Then, since $ \inf_{ |\xi| \leq \Xin } |\varphin(\xi)| \geq e^{-2 \lambdaDn}$ for all $2^J\geq 0$ and $n\in \N$, and $\Tlambda:=\lambdaDn n$, the exact same arguments employed at the end of the proof of Theorem 3.2 in \cite{C18} guarantee that the set on the left hand side of \eqref{eq:inclusioncharacteristicfunctions} is contained in the set of the last display, and \eqref{eq:inclusioncharacteristicfunctions} follows. The second inclusion of this part of the lemma is then justified because, in view of the definition of $\lambdanhat$ in \eqref{eq:lambdanhat}, we have that $e^{\lambdanhatDn}=n/n_0$ and, thus, for any $r\geq 0$ and $n\in \N$,
		\begin{equation*}\label{eq:inclusionestimatorslambda}
		\left\{\omega \in \Omega:n\left|e^{\lambdanhatDn} - e^{\lambdaDn} \right| \leq e^{\lambdaDn}\Tlambda r \right\} \subseteq \left\{  \omega \in \Omega: n_0>0\right\}.
		\end{equation*}

		\item  We start by noting that, in view of the present assumptions and Corollary \ref{corol:varphivarphin}, for $\Tlambda$ large enough
		\begin{equation*}
		\Pr\left(\Omegantildecomplement\left(2^J, r\right)\right) \leq e^{- {L} \Tlambda r^2  } + \Pr\left( n\left|e^{\lambdanhatDn}- e^{\lambdaDn}\right| > e^{\lambdaDn}\Tlambda r   \right),
		\end{equation*}
		and, thus, we have to analyse the probability on the right hand side. Recalling that $e^{\lambdanhatDn}=n/n_0$ and by simple algebra, we have that for all $\lambda,\Dn>0, n\in \N$ and $r\geq 0$,
		\begin{equation}\label{eq:elambdanhatDn-elambdaDn>stuff}
		n\left|e^{\lambdanhatDn}- e^{\lambdaDn}\right| >e^{\lambdaDn}\Tlambda r \quad \Longrightarrow \quad \left| n_0-En_0  \right| >  \Tlambda r \frac{ e^{-\lambdaDn}}{1+ \lambdaDn r},
		\end{equation}
		where we note that $En_0=ne^{-\lambdaDn}$ due to $n_0 \sim \text{Bin}(n,e^{-\lambdaDn})$. Therefore, we need to use a concentration inequality for $n_0$; we use Bernstein's inequality to exploit that $\mbox{var}(n_0)=ne^{-\lambdaDn}(1-e^{-\lambdaDn})=O(\lambdaDn)$ if $\Dn\to 0$, in a similar spirit to what we did when dealing with the characteristic function. By Theorem 3.1.7 in \cite{GN16}, it follows that, for all $u\geq 0$,
		\begin{align*}
		\Pr\left( \left| n_0\!-\!En_0  \right| \!>\! u  \right)\! &\leq \!2 \exp\left( \!-\frac{u^2}{2ne^{-\!\lambdaDn}(1\!-\!e^{-\!\lambdaDn}) + 2u/3}\right) \\
		&\!\leq \!2 \exp\left( \!-\frac{u^2}{2e^{-\!\lambdaDn}\!\Tlambda\!+ 2u/3}\right),
		\end{align*}
		where the second inequality follows because $(1-e^{-x})\leq x$ for all $x\geq 0$. Then, using that $0\leq  r\leq 1$ by assumption together with $1+x \leq e^{x}$ for all $x\geq 0$, 	and defining $L':=3/8 >>L=3/464$, we conclude that
		\begin{equation*}
		\Pr\left(\Omegantildecomplement\left(2^J, r\right)\right) \leq e^{- {L}  \Tlambda r^2 }+ 2 e^{  - L' e^{-3\lambdaDn} \Tlambda r^2 } \leq 3 e^{  - \widetilde{L} \Tlambda r^2 }.
		\end{equation*}

	\end{enumerate}
	
\end{proof}

On the sets $\Omegantilde$, the following simple, yet key, lemma holds.

\begin{lem}\label{lem:remindertermLog}
	If $0\!\leq \!r\!<\! e^{-2\lambdaDn}/(4\lambdaDn)$, $2^J\!\geq\! 0$ and $n\!\in\! \N$, we have that on $\Omegantilde\left(2^J, r\right)$
	\begin{equation*}
	\Log\left(\frac{\varphinhattilde}{\varphintilde}\right)(\xi)  = \left( {\varphinhattilde(\xi)}{-\varphintilde(\xi)}\right) \, \varphintilde^{-1}(\xi) + R_n(\xi) \qquad \mbox{for any} \quad|\xi| \leq \Xin,
	\end{equation*}
	where $\sup_{|\xi|\leq \Xin} |R_n(\xi)| \leq \sup_{|\xi|\leq \Xin} \big|\left( {\varphinhattilde(\xi)}{-\varphintilde(\xi)}\right) \, \varphintilde^{-1}(\xi)\big|^2\leq  4(\lambdaDn r)^2 e^{4\lambdaDn}.$
\end{lem}

\begin{proof}
	Let us first note that, for a function $f:\R \to \C\setminus\{0\}$ such that $f(0)=1$, the distinguished logarithm of it as defined in Section \ref{sec:assumptionsandtheestimator} coincides with the principal branch of the complex logarithm if the path described by $f(\xi)$ when started from $\xi=0$ does not cross the negative real line. In particular, this is the case if the path is of the form $1+z$ for some complex number $|z|<1/2$, and, in such a disc, the principal branch of the complex logarithm	satisfies that $|\Log(1+z) -z|\leq |z|^2$. Notice also that $\inf_{\xi \in \R} |\varphin(\xi)| \geq e^{-2\lambdaDn}$ so, in view of the first inclusion in Lemma \ref{lem:OmegantildesubseteqOmegan:a},  if $\omega \in \Omegantilde\left(2^J, r\right)$ with $2^J\geq 0$, $0\leq r< e^{-2\lambdaDn}/(4\lambdaDn)$ and $n\in\N$,
	\begin{equation*}
	\sup_{|\xi|\leq \Xin} \left|\frac{\varphinhattilde(\xi) - \varphintilde(\xi)}{\varphintilde(\xi)} \right| \leq 2 \lambdaDn r \, e^{2\lambdaDn}  < 1/2.
	\end{equation*}
	Then,  the conclusion  follows by taking, for any $n\in\N$ and $\xi\in \R$,
	\begin{equation*}
	R_n(\xi):=  \Log\left(1+z_n(\xi)\right)  - z_n(\xi) \qquad \mbox{with} \quad z_n(\xi):=\frac{\varphinhattilde(\xi) - \varphintilde(\xi)}{\varphintilde(\xi)}.
	\end{equation*}
\end{proof}

Due to the basic property that the sum of distinguished logarithms is the distinguished logarithm of the product together with the triangle inequality, it immediately follows that, on $\Omegantilde\left(2^J, r\right)$ with $n,2^J$ as in Lemma \ref{lem:remindertermLog} and $0\leq r< \min\{1,e^{-2\lambdaDn}/(4\lambdaDn)\}$,
\begin{align*}\label{eq:Dn-1|logvarphitilde|boundOmegatilde}
\sup_{|\xi|\leq \Xin}\abs{\frac{\Log \varphinhattilde(\xi)}{\Dn}} \!\leq\! \sup_{|\xi| \leq \Xin}\abs{\frac{1}{\Dn}{\Log\left(\frac{\varphinhattilde}{\varphintilde}\right)\!(\xi)} }\!+\! \sup_{|\xi|\leq \Xin}\abs{\frac{\Log \varphintilde(\xi)}{\Dn}(\xi)}\! \leq\! 3 \lambda e^{2\lambdasupnDn} 
\!+\! \lambda\!<\!\infty.
\end{align*}
Therefore, due to the Condition \ref{cond:4} 
and that $U\to \infty$ in both Theorem \ref{thm:upperbound} and \ref{thm:adaptation}, the truncation by $U$ is innocuous on $\Omegantilde\left(2^J, r\right)$ and we have that, for such $n,r$, $J\!\in\! \Nz$ and  $T$ or $\Tlambda$ large enough, respectively,
\begin{equation}\label{eq:presplitnunhat-nu}
\hat{\nu}-\nu=\left((\hat{\nu}-\nu)^{(MS)} + (\hat{\nu}-\nu)^{(R)} + (\hat{\nu}-\nu)^{(B)}\right) \one_{\Omegantilde} + \left|\hat{\nu}-\nu\right|\one_{\Omegantildecomplement},
\end{equation}
where, in view of the second inclusion in Lemma \ref{lem:OmegantildesubseteqOmegan:a} and due to the linearity of $\tildeK_J$, the main stochastic or linear term is
\begin{align}\label{eq:defmainstochastic}
(\hat{\nu}-\nu)^{(MS)}:=\frac{1}{2\pi} \rightbarH{\tildeKJ{\frac{\varphinhattilde-\varphintilde}{\Dn \varphintilde}}},
\end{align}
whilst the remainder or non-linear term and the bias term are, respectively,
\begin{align}\label{eq:defremainderbias}
(\hat{\nu}-\nu)^{(R)}:= \frac{1}{2\pi}\rightbarH{\tildeKJ{\frac{R_n}{\Dn}}}\quad \mbox{and} \quad (\hat{\nu}-\nu)^{(B)}:= \frac{1}{2\pi} \rightbarH{\tildeKJ{\frac{\Log {\varphintilde}}{\Dn}}} - \nu.
\end{align}

In order to show our results we rewrite the main stochastic term as follows. 
Defining $\Pntilde:=e^{\lambdaDn} \Pn$ and $\Pnhattilde:=e^{\lambdanhatDn} \Pnhat$, we have that, on $\Omegantilde$ and for all $x\in \R$,
\begin{align*}
\tildeKJ{\frac{\varphinhattilde-\varphintilde}{2\pi \varphintilde} }\!(x)= \left\langle \FTT{ \Pnhattilde - \Pntilde} ,    \frac{\tildeKJ{x,\cdot}}{2\pi \varphintilde(-\cdot)} \right\rangle = \left( \Pnhattilde- \Pntilde\right)\FII{ \varphintilde^{-1}(-\cdot)  \tildeKJ{x,\cdot}},
\end{align*}
where the last equality follows by Fubini's theorem because $\Pnhattilde - \Pntilde$ is a finite measure on $\Omegantilde$ due to $n_0\!>\!0$, and $\varphintilde^{-1}(-\cdot)  \tildeKJ{x,\cdot} \!\in\! \Lone$ for all $x\!\in\! \R$ due to $\norm{\varphintilde^{-1}}_{\infty} \!\!\leq\! \exp( \lambdaDn)$ and \eqref{eq:suptildeKJ}. 
Recalling the notation $ \nu^{-} (B) = \nu (-B)$ for any Borel set $B \subseteq \R$, we define 
\begin{align}\label{eq:FIvarphintilde-1(-cdot)}
\FII{ \varphintilde^{-1}(-\cdot)}:=  \sum_{i=0}^{\infty} \frac{(-\Dn)^i}{i !}  \left(\nu^{-}\right)^{\ast i},
\end{align}
and remark that this is a finite (signed) measure on $\R$
. Then, arguing as in the proof of Lemma 3.3 in \cite{C18} using Conditions \ref{cond:2} and \ref{cond:4}, we have, for all $x\in \R$,
\begin{align}
(\hat{\nu}-\nu)^{(MS)}(x)&=\frac{1}{\Dn} \rightbarH{ \left(\left( \Pnhattilde - \Pntilde\right)\left(\FII{ \varphintilde^{-1}(-\cdot)}  \ast \KJ{x,\cdot}\right)\right)}\label{eq:mainstochastic1} \\
&=\frac{1}{\Dn} \rightbarH{ \left(\FII{ \varphintilde^{-1}(-\cdot)} \left( \left( \Pnhattilde - \Pntilde\right)\ast \KJ{x,-\cdot} \right)\right)},\label{eq:mainstochastic2}
\end{align}
where the second equality follows due to the uniform boundedness of $K_J(\cdot,\cdot)$ by Condition \ref{cond:2}. Lastly, using the notation  in Section \ref{sec:unification} and $e^{\lambdanhatDn}=n/n_0$ with $n_0 \sim \text{Ber}(n,e^{-\lambdaDn})$, 
\begin{align}
\Pnhattilde - \Pntilde &=  \left( e^{\lambdanhatDn} -e^{\lambdaDn}\right) \Pnhat + e^{\lambdaDn} \left( \Pnhat -\Pn\right)\label{eq:tildePn-tildeP1}\\
&=\left(\frac{n-n_0}{n_0}\right) \left( \widehat{\Pn \,}^{J)}_{\!n} - \Pn^{J)}\right) + e^{\lambdaDn}\left(\frac{\E n_0 - n_0}{n_0}\right) \Pn^{J)}.\label{eq:tildePn-tildeP2}
\end{align}

\subsection{Proof of Theorem \ref{thm:upperbound}}\label{sec:proofs:upperbound}

In view of \eqref{eq:presplitnunhat-nu}  with $\hat{\nu}=\nunhat$ and by the triangle inequality, we have that, for all $p\in [1,\infty]$, $H\geq 0$ if $p\in[1,2)$ and $H=\infty$ if $p\in [2,\infty]$, $n\in\N$, $J\in \Nz$, $0\leq r < \min\{1,e^{-2\lambdaDn}/(4\lambdaDn)\}$ and 
$\Tlambda$ large enough,
\begin{align}\label{eq:splitnunhat-nu}
\E \norm{\nunhat - \nu}_p  \leq &\E\left[ \norm{\left(\nunhat- \nu\right)^{(MS)}}_p \one_{ \Omegantilde\left( 2^J, r\right)}\right] + \E\left[ \norm{\left(\nunhat - \nu\right)^{(R)}}_p \one_{ \Omegantilde\left( 2^J, r\right)}\right] \notag\\
&+ \E\left[ \norm{\left(\nunhat - \nu\right)^{(B)}}_p \one_{ \Omegantilde\left( 2^J, r\right)}\right] + \E\left[ \norm{\nunhat - \nu }_p \one_{ \Omegantildecomplement\left( 2^J, r\right)}\right].
\end{align}
In the following four results we  bound each of these terms in reverse order 
and, at the end, put them together to show Theorem \ref{thm:upperbound}.

\begin{lem}\label{lem:Omegantildecomplement}
	Let $\supnDn <\infty$ and assume that $\nu\in \Lp$ for some $p\in [1,\infty]$ and that $\mu=\nu/\lambda$ satisfies Assumption \ref{ass2b} for some $\beta>1$. Furthermore, if $p\in[1,2)$, let $H\geq 0$ and, if $p\in[2,\infty]$ let $H=\infty$. Then, if $2^J=\twoJn$ as in \eqref{eq:2J} and $\sqrt{\log^{1+2\delta}(e+2^J) / \Tlambda } \leq  r\leq 1$ for some $\delta>0$, we have that, for $\widetilde L$ as in Lemma \ref{lem:OmegantildesubseteqOmegan:b}, any $U\geq 0$ and $\Tlambda$ large enough,
	\begin{equation}\label{eq:Omegantildecomplementbound}
	\E\left[\norm{\nunhat-\nu}_p \one_{\Omegantildecomplement\left(2^J, r\right)} \right]\lesssim \left( 2^J U H^{1/p}\one_{[1,2)}(p) +2^{J(1-{1}/{p})} U\one_{[2,\infty]}(p) +\norm{\nu}_p\right)  e^{- \widetilde L \Tlambda r^2  },
	\end{equation}
	where the constant hidden in the notation $\lesssim$ only depends on $K$.
\end{lem}

\begin{proof}
	By the  triangle inequality,  the definition of $\nunhat$ in \eqref{eq:estimatornu} and its expression with the notation at the beginning of Section \ref{sec:proofs:prepresults} 
	valid on $\Omegan$, 
	\begin{equation}\label{eq:Omegantildecomplementboundpre}
	\norm{\nunhat-\nu}_p \one_{\Omegantildecomplement\left(2^J, r\right)} \leq  \norm{\frac{1}{2\pi}\!\rightbarH{\tildeK_J\left(\rightbarU{\frac{\Log \varphinhattilde}{\Dn}}\right)} }_p \!\one_{\Omegan \cap \Omegantildecomplement\left(2^J, r\right)} + \norm{\nu}_p \one_{\Omegantildecomplement\left(2^J, r\right)}.
	\end{equation}
	In view of  Lemma \ref{lem:OmegantildesubseteqOmegan:b}, the second summand immediately gives rise to the last term in \eqref{eq:Omegantildecomplementbound}. For the first summand we note that, in view of \eqref{eq:suptildeKJ}
	, Fubini's theorem and the definition of $\Omegan$ in \eqref{eq:defOmegan}, 
	for all $U\geq 0$ and $x\in \R$, 
	\begin{equation}\label{eq:ftildeK=KFIf}
	\frac{1}{2\pi}\left\langle   \rightbarU{\frac{\Log \varphinhattilde}{\Dn}}\!\!, \tildeKJ{x,\cdot} \right\rangle \!=\!\KJ{  \FII{\rightbarU{\frac{\Log \varphinhattilde}{\Dn}}\!\! \one_{\widetilde{S}}}}\!(x),
	\end{equation}
	where $\widetilde{S}\!:=\!{\supp\big(\tildeKJ{x,\cdot}\big)}$, and we recall that, by the first part of Condition \ref{cond:4}, $\widetilde{S}$ does not depend on $x$. To bound the first summand in the penultimate display 
	we proceed differently depending on whether $p\!\in\! [1,2)$ or $p\!\in\![2,\infty]$. In the former case  it is bounded above by
	\begin{align*}
	(2H)^{\frac{1}{p}}\norm{K_J\!\left(\!\FII{\!\rightbarU{\frac{\Log \varphinhattilde}{\Dn}}\!\!\one_{\widetilde{S}}}\right)}_\infty \lesssim (2H)^{\frac{1}{p}}\norm{{\!\rightbarU{\frac{\Log \varphinhattilde}{\Dn}}\!\!\one_{\widetilde{S}}}}_1  \leq (2H)^{\frac{1}{p}} \Xin U,
	\end{align*}
	where in the first inequality we used \eqref{eq:boundnormKJfpbynormfp} and $\FII{\Lone}\subseteq \Linf$, so the constant in $\lesssim$ only depends on $\barK$, and in the second inequality follows from the second part of Condition \ref{cond:4}. 
	For $p\in[2,\infty]$, we use \eqref{eq:boundnormKJfpbynormfp} again to justify that 
	the quantity of interest is bounded above, up to a constant only depending on $\barK$, by
	\begin{equation*}
	\norm{\FII{\rightbarU{\frac{\Log \varphinhattilde}{\Dn}}\one_{\widetilde{S}}}}_p\leq \frac{1}{2\pi} \norm{{\!\rightbarU{\frac{\Log \varphinhattilde}{\Dn}}\one_{\widetilde{S}}}}_{(1-1/p)^{-1}} \leq \frac{1}{2\pi} \left(\Xin\right)^{1-1/p} U,
	\end{equation*}
	where in the first inequality we used the Hausdorff--Young inequality. 
	These two bounds together with Lemma \ref{lem:OmegantildesubseteqOmegan:b} show the conclusion of the present lemma for any $p\in[1,\infty]$. 
\end{proof}

\begin{lem}\label{lem:bias}
	Assume that $\mu=\nu/\lambda$ satisfies Assumption \ref{ass1} for some $p\in [1,\infty]$ and $s>0$, and,  if  $p\in [1,2)$, suppose it satisfies Assumption \ref{ass2a} 
	for some $\alpha>0$ and $q=p$. 
	Let $H\geq 0$ if $p\in[1,2)$ and $H=\infty$ if $p\in[2,\infty]$. 
	Then, for any $J\in \Nz$, 
	\begin{equation}\label{eq:biasbound}
	\norm{(\nunhat-\nu)^{(B)}}_p  \leq \lambda C^{(B)}  \norm{\mu}_{B_{p,\infty}^s} 2^{-Js} +  \lambda \norm{ \mu}_{p,\alpha} H^{-\alpha}   \one_{[1,2)}(p),
	\end{equation}
	where $C^{(B)}=C^{(B)}(s,K)$. 
\end{lem}

\begin{proof}

	Note that, due to \eqref{eq:suptildeKJ} and Condition \ref{cond:2}, 
	the Fourier inversion theorem (cf. result 8.26 in \cite{F99}) guarantees that, 
	for all 
	$J\in \Nz$ and $x,y\in \R$, $\KJ{x,y}= \FII{  \tildeKJ{x,\cdot}}(y)$.
	Then, in view of \eqref{eq:identityLK2}, 
	$\nu\in \Lone$ and  Fubini's theorem,
	\begin{equation*}
	\tildeK_J\left(\frac{\Log {\varphintilde}}{\Dn}\right)= \tildeK_J\left(\FT \nu\right) = 
	K_J(\nu),
	\end{equation*} 
	and we can write
	\begin{equation*}
	(\nunhat-\nu)^{(B)}= \rightbarH{\left(K_J(\nu)-\nu\right)} - \nu \one_{[-H,H]^{\setcomplement}},
	\end{equation*}
	where in the case $p\in[2,\infty]$ the second summand on the right hand side is not present since $H=\infty$. Then, the statement of the present lemma follows because $\mu=\nu/\lambda$ satisfies Assumption \ref{ass1} for some $p\in [1,\infty]$ and $s>0$, and  Assumption \ref{ass2a} for some $\alpha>0$ and $q=p$ if  $p\in [1,2)$, together with \eqref{eq:boundLrnormKjpernu}.
\end{proof}


\begin{prop}\label{prop:remainder}
	For all $p\in[1,\infty]$, $H\geq 0$ if $p\in[1,2)$ and $H=\infty$ if $p\in[2,\infty]$, $J\in \Nz$,	$0\leq r< e^{-2\lambdaDn}/(4\lambdaDn)$ and $n\in \N$, we have that 
	\begin{align}\label{eq:remainderbound}
	\E \left[\norm{(\nunhat-\nu)^{(R)}}_p \one_{\Omegantilde\left(2^J, r\right)}  \right] 	 \leq C^{(R)}  \lambda^2 \Dn e^{4\lambdaDn} r^2 &\left(  H^{\left(\frac{2}{p}-1\right)}2^{\frac{J}{p}} \one_{[1,2)}(p)\right. \notag \\
	& \left.\,\,\,\,+ 2^{J\left(1-\frac{1}{p}\right)} \one_{[2,\infty]}(p) \right),
	\end{align}	
	where $C^{(R)}=C^{(R)}\left(p,K\right)$.
\end{prop}

\begin{proof}
	
	We start with the case $p\in[1,2)\cup \{\infty\}$. In view of \eqref{eq:defremainderbias} and \eqref{eq:ftildeK=KFIf}, and arguing as thereafter using the second bound for $|R_n|$ in Lemma \ref{lem:remindertermLog},
	\begin{equation*}
	\norm{(\nunhat-\nu)^{(R)}}_p \leq \frac{(2H)^{\frac{1}{p}}}{\Dn} \norm{K_J\!\left(\!\FII{{R_n}\one_{\widetilde{S}}}\right)}_\infty \lesssim 
	\frac{(2H)^{\frac{1}{p}}}{\Dn}4 (\lambdaDn r)^2 e^{4\lambdaDn}\Xin ,
	\end{equation*}
	where the constant in $\lesssim$ only depends on $K$. Taking the expectation on both sides 
	gives the upper bound in \eqref{eq:remainderbound} when $p=1$ and $p=\infty$. 
	%
	%
	To obtain \eqref{eq:remainderbound} for all $p\in(1,\infty)$, we now improve the bound in the last display for the case $p=2$ and interpolate. Arguing as in the last display,	
	we have that, for all $J\in \Nz$,
	\begin{equation*}
	\norm{(\nunhat-\nu)^{(R)} }_2=  \frac{1}{\Dn}\norm{{K_J\left( \FII{{R_n} \one_{\widetilde{S}}}\right)}}_2  \lesssim \frac{1}{\Dn} \norm{\FII{{R_n}\one_{\widetilde{S}}}}_2   ,
	\end{equation*}
	where the constant in $\lesssim$ only depends on $K$ and, by Plancherel theorem,  the right  side equals
	\begin{align*}\label{eq:reminderboundp=2}
	\frac{1}{\Dn} \norm{{R_n}\one_{\widetilde{S}} }_{2} \lesssim  \frac{1}{\Dn} \sup_{ \xi  \in \R: |\xi| \leq \Xin} |R_n(\xi)|  \, 2^{J/2} \leq \frac{1}{\Dn}4 (\lambdaDn r)^2 e^{4\lambdaDn}2^{J/2}.
	\end{align*}
	This shows \eqref{eq:remainderbound}  for $p=2$ and we can now interpolate to get $p\in (1,2)$ and $p\in(2,\infty)$. For the case $p\in(1,2)$ we fix $J\in \Nz$ and $H\geq 0$, and, for a measurable set $B\subseteq \R$, we define $\mathcal{B}_{B}$ to be its Borel $\sigma$-algebra inherited from $\R$, and $\mu_{B}$ to be the Lebesgue measure on $\R$ restricted to $B$. Furthermore, let us write $L^p(\mu_{B})$, $p \in [1,\infty]$, for the usual Banach spaces of measurable functions arising from $\left( B, \mathcal{B}_{B}, \mu_{B} \right)$. Then, we remark that, from the calculations for  $p=1,2$, it follows that the linear operator
	\begin{align*}
	\tildeK_J:L^\infty\left(\mu_{[-\Xin,\Xin]}\right) \to L^1\left(\mu_{[-H,H]}\right) + L^2\left(\mu_{[-H,H]}\right) 
	\end{align*}
	satisfies, for any $f\in  L^\infty\left(\mu_{[-\Xin,\Xin]}\right)$,
	\begin{equation*}
	\norm{\tildeK_Jf }_{L^1\left(\mu_{[-H,H]}\right)} \lesssim  H2^J \norm{f}_{L^\infty\left(\mu_{[-\Xin,\Xin]}\right)},
	\end{equation*}
	and
	\begin{equation*}
	\norm{\tildeK_Jf }_{L^2\left(\mu_{[-H,H]}\right)} \lesssim  2^{J/2} \norm{f}_{L^\infty\left(\mu_{[-\Xin,\Xin]}\right)},
	\end{equation*}
	where the constants hidden in the notation only depend on $K$. 
	Consequently, the Riesz--Thorin interpolation theorem (cf. 6.27 in \cite{F99}) can be used to conclude that, for any $p\in(1,2)$, $J\in \Nz$, $H\geq 0$ and $f\in  L^\infty\left(\mu_{[-\Xin,\Xin]}\right)$,
	\begin{align}\label{eq:boundLpnormremainderterm}
	\norm{\tildeK_Jf }_{L^p\left(\mu_{[-H,H]}\right)} &\lesssim \left(H2^J\right)^{\left(\frac{2}{p}-1\right)} \left(2^{J/2}\right)^{2\left(1-\frac{1}{p}\right)} \norm{f}_{L^\infty\left(\mu_{[-\Xin,\Xin]}\right)}\notag \\
	&=H^{\left(\frac{2}{p}-1\right)}2^{J/p}  \norm{f}_{L^\infty\left(\mu_{[-\Xin,\Xin]}\right)},
	\end{align}
	where the hidden constant depending on $K$ and $p$. Therefore, \eqref{eq:remainderbound} for $p\in (1,2)$  follows by applying the operator $\tildeK_J$ to $R_n$ and noting that we are on $\Omegantilde\left(2^J,r\right)$ with $0\leq r< e^{-2\lambdaDn}/(4\lambdaDn)$, $2^J\geq 0$ and $n\in \N$ so Lemma \ref{lem:remindertermLog} applies, and taking expectations. The conclusion of the proposition then follows because, by analogous but simpler arguments, for $p\in (2,\infty)$ we have that
	\begin{align}\label{eq:boundLpnormremainderterm2}
	\norm{\tildeK_Jf }_{\Lp}
	&\lesssim 2^{J\left(1-\frac{1}{p}\right)}  \norm{f}_{L^\infty\left(\mu_{[-\Xin,\Xin]}\right)},
	\end{align}
	where the hidden constant depending on $K$ and $p$; this can  be shown directly too by using similar arguments to those in Lemma \ref{lem:Omegantildecomplement}, i.e. by the Hausdorff--Young inequality. 
	
\end{proof}

\begin{prop}\label{prop:mainstochastic}
	Assume $\mu=\nu/\lambda\in \Lp$ for some $p\in[1,\infty]$ and,  if  $p\in [1,2)$, suppose it satisfies Assumption \ref{ass2a} 
	for some $\alpha>2/p-1$ and $q=1$. 
	Let $H\geq 0$ if $p\in[1,2)$ and $H=\infty$ if $p\in[2,\infty]$,  and let  $J\in \Nz$ be such that	$2^J<n/2\left(1-e^{-\lambdaDn}\right)$ if $p\in [2,\infty)$ and $J2^J <n/2\left(1-e^{-\lambdaDn}\right)$  if $p=\infty$. Then, for all $r\leq \frac{1}{\lambdaDn}\frac{1-\exp(-\lambdaDn)}{1+\exp(-\lambdaDn)}$ and $n\in \N$, we have that 
	\begin{equation}\label{eq:mainstochasticbound}
	\E \left[\norm{(\nunhat-\nu)^{(MS)}}_p \one_{\Omegantilde\left(2^J, r\right)}  \right] 	 \leq  2 \lambda e^{2\lambdaDn} L_p M_p'\sqrt{\frac{2^J(1+\one_{\{\infty\}}(p))}{\Tlambda}} +\lambda\frac{e^{3\lambdaDn} }{\sqrt{\Tlambda}}\norm{\mu}_p,
	\end{equation}	
	where 
	\begin{equation}\label{eq:defMp'}
	M_p'\!:=\!\left\{\!
	\begin{array}{ll}
	\!e^{\lambdaDn \left( \norm{ \mu }_{1,\alpha/2}-1\right) }\!\left(  \left(e^{\lambdaDn \norm{  \mu}_{1,\alpha}} \!-\!1\right)\!/\!\left(e^{\lambdaDn}\! -\!1\right) \right)^{\!1/2}\!, & \mbox{if } p\!\in\![1,2),\\
	\!\one_{[2,\infty)}(p)+\norm{\mu}_{p/2}^{1/2}\one_{(2,\infty)}(p)+\left(1\vee \norm{\mu}_{\infty} \right)^{1/2}\one_{ \{\infty\}}(p)
	, & \mbox{if } p\!\in\![2,\infty]. 
	\end{array}	
	\right.
	\end{equation}
\end{prop}

\begin{proof}

	By the calculations that led to \eqref{eq:mainstochastic2}, we have that on $\Omegantilde\left(2^J,r\right)$ with $r\leq 1$ and for all $n\in \N$, $J\in \Nz$ and $x\in \R$,
	\begin{equation*}
	\frac{1}{2\pi}\tildeK_J\left(\frac{{\varphinhattilde}{-\varphintilde}}{\Dn\varphintilde} \right)(x) =\frac{1}{\Dn}\int_{\R} \left( \int_{\R} K_J(x, y) \left( \Pnhattilde - \Pntilde\right) (d(y+z))  \right) \FII{ \varphintilde^{-1}(-\cdot)} (dz), 
	\end{equation*}
	and, using Minkowski's inequality for integrals together with the fact that, due to \eqref{eq:FIvarphintilde-1(-cdot)} and \eqref{eq:repPintro},
	\begin{equation}\label{eq:FIvarphintilde-1leqelambdaDnPn}
	\abs{ \FII{ \varphintilde^{-1}}(dz)} \leq e^{\lambdaDn} \Pn(dz), \qquad z\in \R,
	\end{equation} 
	we have that, under the same assumptions and for all $H\geq 0$ and $p\in [1,\infty]$,
	\begin{align}\label{eq:mainstochasticequiv}
	\norm{(\nunhat-\nu)^{(MS)}}_p \leq \frac{e^{\lambdaDn}}{\Dn} \bigintssss_{\R} \norm{\int_{\R} K_J(\cdot, y) \left( \Pnhattilde - \Pntilde\right) (d(y-z)) }_p \Pn(dz).
	\end{align}	
	We now use \eqref{eq:tildePn-tildeP2} and note that, on $\Omegantilde\left(2^J,r\right)$ with $r\leq \frac{1}{\lambdaDn}\frac{1-\exp(-\lambdaDn)}{1+\exp(-\lambdaDn)}$ and by simple algebra,
	\begin{equation}\label{eq:n0boundsOmegantilde}
	n \frac{e^{-\lambdaDn}}{2} 	\leq n_0 \leq n \frac{1+e^{-\lambdaDn}}{2}.
	\end{equation}
	Then, by Fubini's theorem and the lower bound for $n_0$ on the left side, for all $H\geq 0$ if $p\in[1,2)$, and all $n\in \N, J\in \Nz$ and $p\in [1,\infty]$,
	\begin{align}\label{eq:mainstochasticprebound}
	&\!\!\!\!\!\!\!\E \left[\norm{(\nunhat\!-\nu)^{(MS)}}_p \!\one_{\Omegantilde}  \right]  \notag\\
	& \qquad \leq 2\frac{e^{2\lambdaDn}}{\Dnn}\! \!\bigintsss_{\R}  \! \E \left[\left({n\!-\!n_0}\right)  \norm{\int_{\R} K_J(\cdot, y) \left(  \widehat{\Pn \,}^{J)}_{\!n} \!\!-\! \Pn^{J)}\right) \!(d(y\!-\!z)) }_p\!\!\one_{\Omegantilde} 
	\right]\!\Pn(dz) \notag  \\
	&\qquad \quad +2\frac{e^{3\lambdaDn}}{\Dnn} \E \left| {n_0\!-\!\E n_0}\right|  \! \bigintssss_{\R}    \norm{\int_{\R} \!K_J(\cdot, y) 	\Pn^{J)} (d(y\!-\!z)) }_p\! \Pn(dz)  .
	\end{align}	
	We now bound the second term on the right  side: by \eqref{eq:boundnormKJfpbynormfp}, the definition of $ \Pn^{J)}$ in Section \ref{sec:unification}, Young's inequality for convolutions and $\norm{\nu}_1=\lambda$, the $\norm{\cdot}_p$ norm therein is, up to a constant depending on $K$ and for all $ z\in \R, J\in \Nz$ and $p\in [1,\infty]$, 
	bounded above by
	\begin{equation}\label{eq:boundnormp PnJ}
	\norm{ \Pn^{J)} (d(\cdot-z))}_p \leq  \frac{1}{ e^{\lambdaDn} -1} \sum_{i =1}^{\infty} \frac{ \Dnsuper{i} \norm{\nu(d(\cdot-z))}_p  \lambda^{i-1} }{i!} \leq \frac{ \norm{\nu}_p}{\lambda};
	\end{equation}
	then, in view of Jensen's inequality, 
	$n_0 \sim \text{Bin}(n,e^{-\lambdaDn})$  and $1-\exp(-x) \leq \min\{1,x\}$ for all $x \geq 0$, 
	\begin{align}\label{eq:maxineqvarphiparametric}
	\E \left| {n_0\!-\!\E n_0}\right| \!\leq\! \sqrt{\E \left| n_0 \!-\! \E n_0 
		\right|^2} \! \leq\! 
	\sqrt{n{e^{-\!\lambdaDn} (1\!-\!e^{-\!\lambdaDn})}} 
	\!\leq\! 
	\sqrt{\Tlambda\min\left\{\!\frac{1}{\lambdaDn}, 1\right\}} ,		
	\end{align}
	and	the second term on the right hand side of \eqref{eq:mainstochasticprebound} is bounded above by the second summand in \eqref{eq:mainstochasticbound}. To bound the first term, we write the expectation in it as
	\begin{align}\label{eq:EEmainstochastic}
	\E \left[m  \E \left[   \norm{\int_{\R} K_J(\cdot, y) \left(  \widehat{\Pn \,}^{J)}_{n} - \Pn^{J)}\right) \!(d(y-z)) }_p \one_{\Omegantilde}\bigg|_{n_0=n-m}  \right]\right] .
	\end{align}
	Note that, in view of the upped bound on $n_0$ in \eqref{eq:n0boundsOmegantilde}, $m=n-n_0\geq n/2(1-e^{-\lambdaDn})$. Then, due to the present assumptions on $J$ and on $\nu$, together with the fact that Condition \ref{cond:3} summarises all that is required by the empirical process theory results used in case 4 of the proof of Theorem 5.1.5 in \cite{GN16}, we can apply the variance bounds for the standard problem of density estimation from i.i.d. observations (cf. Theorem 5.1.5 and Remark 5.1.6 in \cite{GN16}) to bound the inner expectation; we conclude that, for any $z\in \R, n\in \N, p\in [1,\infty]$ and $J$ as in the statement of the current proposition, the last display is bounded above by
	\begin{align*}
	L_p \, \varUpsilon_p(z)\,  \sqrt{2^J(1+\one_{\{\infty\}}(p))} \E \sqrt{n-n_0},
	\end{align*}
	where, by Condition \ref{cond:2} 
	and the fact that $\nu$ satisfies Assumption \ref{ass2a} for some $\alpha>2/p-1$ and $q=1$, 
	$L_p$ is as defined in Theorem \ref{thm:upperbound} and 
	\begin{equation*}
	\varUpsilon_p(z):=\norm{ \Pn^{J)} (d(\cdot-z)) }_{1,\alpha}^{1/2} \one_{[1,2)}(p) + \one_{[2,\infty]}(p) +\norm{ \Pn^{J)}(d(\cdot-z))}_{p/2}^{1/2} \one_{(2,\infty]}(p).
	\end{equation*}
	%
	Due to \eqref{eq:boundnormp PnJ} it follows that, for any $z\in \R$ and $p\geq 2$, $\varUpsilon_p(z) \leq  M_p'$, where 
	$M_p'$ is as defined in \eqref{eq:defMp'}. Arguing as in \eqref{eq:maxineqvarphiparametric},  
	\begin{equation}\label{eq:Esqrtn-n0bound}
	\E \sqrt{n-n_0} \leq \sqrt{n\left(1-e^{-\lambdaDn}\right)} \leq \sqrt{\Tlambda \min\{1/(\lambdaDn), 1\}},
	\end{equation}
	and this finishes showing that the first term on the right hand side of \eqref{eq:mainstochasticprebound} is bounded by the first summand in \eqref{eq:mainstochasticbound} when $p\geq 2$. In view of the definition of $\Pn^{J)}$  in Section \ref{sec:unification} and using that $(1+|x+y|)\leq (1+|x|)(1+|y|)$ for all $x,y\in \R$ together with Young's inequality for convolutions repeatedly, we have that, for $p\in[1,2)$, $\varUpsilon_p(z)$ is bounded 
	by
	\begin{equation*}
	\left( \frac{\left( 1\!+\!|z|\right)^{\alpha}}{e^{\lambdaDn} \!-\!1}  \sum_{i =1}^{\infty} \frac{\Dnsuper{i} \norm{\nu (d(\cdot\!-\!z))\left( 1\!+\!|\cdot - z|\right)^{\alpha} }_1 \norm{\nu  }_{1,\alpha}^{i-1}}{i!}\right)^{\!\frac{1}{2}} \!\!	 \leq\! \left( 1\!+\!|z|\right)^{\alpha/2}\left(  \frac{e^{\Dn \norm{\nu }_{1,\alpha}} \!-\!1}{e^{\lambdaDn}\! -\!1} \right)^{\!\frac{1}{2}}\!.
	\end{equation*}
	Due to \eqref{eq:Esqrtn-n0bound}, for $p\in[1,2)$ the first summand on the right hand side in \eqref{eq:mainstochasticprebound} is bounded above by
	\begin{equation*}
	2\lambda  e^{2\lambdaDn} M_p' \, e^{ -\Dn \left( \norm{ \nu}_{1,\alpha/2} -\lambda\right) }  \int_{\R} \left(1+|z|\right)^{\alpha/2}\Pn(dz)  \:\sqrt{\frac{2^J(1+\one_{\{\infty\}}(p))}{\Tlambda} },
	\end{equation*}
	which in turn is bounded above by the first summand in \eqref{eq:mainstochasticbound} because, in view of \eqref{eq:repPintro} and arguing as in the second-to-last display, the integral is bounded by $e^{\Dn \left( \norm{ \nu }_{1,\alpha/2} -\lambda\right) }$.
	
\end{proof}

\begin{proof}[Proof of Theorem \ref{thm:upperbound}]
	%
	%
	
	Trivially, $\E \one_{\Omegantilde\left( 2^J, r\right)} \leq 1$ for any $2^J,r\geq0$. Thus, the four terms in \eqref{eq:splitnunhat-nu} are bounded in Lemmata \ref{lem:Omegantildecomplement}, \ref{lem:bias} and Propositions \ref{prop:remainder}, \ref{prop:mainstochastic}, so we simply have to verify that  under the present assumptions and choices of $H,U,2^J$, the conditions therein are satisfied and the bounds they provide give rise to the bound in \eqref{eq:upperbound}. 
	
	Most of the assumptions in such results 
	are assumed here in the same form. For the remaining, we check them in order of appearance. If Assumption \ref{ass2a} is satisfied for some $\alpha>0$ and $q=1$, then $\int_{\R} \log^\beta (\max\{|x|,e\}) \mu(dx)<\infty$ for any $\beta>0$ and, thus, under the conditions on $\nu$ when $p\in [1,2)$, it satisfies Assumption \ref{ass2b}. 
	In Proposition \ref{prop:mainstochastic} we also assumed that $J\geq 0$ if $ p\in [1,2)$,  $2^J<n/2\left(1-e^{-\lambdaDn}\right)$ if $p\in [2,\infty)$, and  $J2^J <n/2\left(1-e^{-\lambdaDn}\right)$  if $p=\infty$. These conditions are trivially satisfied for the choice $2^J=\twoJn=o\left(\Tlambda\right)$ for  $\Tlambda$  large enough because, due to $e^{-x}\leq (1+x)^{-1}$ for all $x\geq 0$,
	\begin{equation}\label{eq:n/2(1-e-lambdaDn)>Tlambda}
	n/2\left(1-e^{-\lambdaDn}\right)\geq  \frac{\Tlambda}{2(1+\lambdasupnDn)}.
	\end{equation}

	Next, we check the convergence rates and constants of the bounds.  With the familiar choice $2^J=\twoJn$, the first terms in \eqref{eq:biasbound} and \eqref{eq:mainstochasticbound} are of order $\varepsilon_{\Tlambda}$. Due to $\supnDn<\infty$, introducing the term $e^{2\lambdaDn}$ in  \eqref{eq:mainstochasticbound} into $M_p'$ defined in \eqref{eq:defMp'}, noticing that $\norm{\mu}_{1,\alpha}>1$ for any $\alpha>0$ and using the same arguments employed to conclude \eqref{eq:CleqbarC}, it follows that $M_p'\leq M_p$. 
	Then, all this justifies the form of the whole term of order $\varepsilon_{\Tlambda}$ in \eqref{eq:upperbound}. 	Furthermore, in view of \eqref{eq:varepsilonm} and \eqref{eq:Hn},  for $p\in [1,2)$,
	\begin{equation*}
	\Hn^{\!\!\!\!-\alpha}= \Tlambda^{ \alpha \left(\frac{1}{\alpha}\frac{s}{2s+1}-  \vartheta_H\right)} \varepsilon_{\Tlambda}=o\left( \varepsilon_{\Tlambda}\right).
	\end{equation*}
	Hence, the last term of order $o\left(\varepsilon_{\Tlambda}\right)$  in \eqref{eq:upperbound} arises from the second summand in \eqref{eq:biasbound}. Due to $\supnDn<\infty$, the last summand in \eqref{eq:mainstochasticbound} gives rise to the first term carrying the quantity $\norm{\mu}_p$. To justify the appearance of the other three terms of order $o\left(\varepsilon_{\Tlambda}\right)$, we need to argue that the bounds in \eqref{eq:remainderbound} and \eqref{eq:Omegantildecomplementbound} indeed give rise to them for some choice of $r=r_{\Tlambda}$ satisfying that, for some $\delta>0$,
	\begin{equation}\label{eq:rbounds}
	\sqrt{\frac{ \log^{1+2\delta}\left(e+\twoJn\right)}{\Tlambda}  }\leq r_{\Tlambda}< \min\left\{ \frac{ e^{-2\lambdaDn}}{4\lambdaDn}, \frac{1}{\lambdaDn}\frac{1-e^{-\lambdaDn}}{1+e^{-\lambdaDn}} \right\}.
	\end{equation}
	Note that, due to $0<\Dn\leq \supnDn<\infty$ and $e^{-x}-1\leq x \left(e^{-\bar x}-1\right)/\bar{x}$ for all $0\leq x\leq \bar x$, the upper bound is bounded from below by
	\begin{equation}\label{eq:rlowerupperbound}
	\min\left\{ \frac{ e^{-2\lambdasupnDn}}{4\lambdasupnDn}, \frac{1-e^{-\lambdasupnDn}}{2\lambdasupnDn} \right\},
	\end{equation}
	so, by the polynomial growth of $2^J=\twoJn$ in $\Tlambda$, the conditions on $r_{\Tlambda}$ are satisfied for $\Tlambda$ large enough by taking
	\begin{equation}\label{eq:rn}
	r_{\Tlambda}:=\Tlambda^{\!\!\vartheta_r} \qquad \mbox{with } -1/2 < \vartheta_r<0.
	\end{equation}
	We deal with the bound in \eqref{eq:remainderbound} first. To show that it is $o\left(\varepsilon_{\Tlambda}\right)$ for $p\in[1,2)$, we need that
	\begin{equation*}
	2\vartheta_r   +\vartheta_H \left(2/p-1\right) + \frac{1/p}{2s+1} < -\frac{s}{2s+1} \quad \Longleftrightarrow \quad \vartheta_r  < -\frac{1}{2} \left(\vartheta_H\left(2/p-1\right)   + \frac{s+1/p}{2s+1}\right).
	\end{equation*}
	In other words, to have room to choose $\vartheta_r$ as in \eqref{eq:rn} and guaranteeing that the bound in \eqref{eq:remainderbound}  is $o\left(\varepsilon_{\Tlambda}\right)$, it is necessary that
	\begin{equation*}
	\vartheta_H\left(2/p-1\right)   + \frac{s+1/p}{2s+1}<1 \quad \Longleftrightarrow \quad \vartheta_H < \frac{1}{2/p-1}\frac{1+s-1/p}{2s+1},
	\end{equation*}
	as assumed in \eqref{eq:Hn}. 
	To show the same conclusion for $p\in[2,\infty]$, we need 
	\begin{equation*}
	2\vartheta_r   +  \frac{1-1/p}{2s+1} < -\frac{s}{2s+1} \quad \Longleftrightarrow \quad \vartheta_r  < -\frac{1}{2} \frac{1+s-1/p}{2s+1}.
	\end{equation*}
	Simple algebra shows that the upper bound on the right hand side is strictly larger than $-1/2$ and, thus, we can choose $\vartheta_r$ as in \eqref{eq:rn} and making the bound in \eqref{eq:remainderbound}  be $o\left(\varepsilon_{\Tlambda}\right)$. 
	Then, due to $\supnDn<\infty$,
	the remainder term gives rise to the second-to-last term of order $o\left(\varepsilon_{\Tlambda}\right)$ in \eqref{eq:upperbound}.	
	Lastly, by the expressions for $\twoJn,\Un,\Hn$ and $r_{\Tlambda}$, the bound in 	\eqref{eq:Omegantildecomplementbound} can be written as
	\begin{equation*}
	\left( \left(\frac{\Tlambda}{(\log \Tlambda)^{\one_{\{\infty\}}(p)}}\right)^{\frac{1}{2s+1}} \Tlambda^{\vartheta_U} \Tlambda^{\vartheta_H/p} +\norm{\nu}_p\right)   e^{- \widetilde L \Tlambda^{1+2\vartheta_r} } \qquad \mbox{ if } p\in[1,2)
	\end{equation*}
	and
	\begin{equation*}
	\left( \left(\frac{\Tlambda}{(\log \Tlambda)^{\one_{\{\infty\}}(p)}}\right)^{\frac{1-1/p}{2s+1}} \Tlambda^{\vartheta_U}  +\norm{\nu}_p\right)   e^{- \widetilde L \Tlambda^{1+2\vartheta_r} } \qquad \mbox{ if } p\in[2,\infty],
	\end{equation*}
	and it is clear that,  under the present conditions $\vartheta_U, \vartheta_H<\infty$ and $\vartheta_r>-1/2$, this term vanishes exponentially fast in $\Tlambda$. Therefore, introducing a harmless extra $\lambda$ factor in the first summand inside the outer brackets, we conclude that these last two displays give rise to the first term of order $o\left(\varepsilon_{\Tlambda}\right)$ in \eqref{eq:upperbound} and to the second carrying the quantity $\norm{\mu}_p$. This ends the proof of Theorem \ref{thm:upperbound}.

\end{proof}

\subsection{Proof of Theorem \ref{thm:adaptation}}\label{sec:proofs:adaptation} 


Let $\Jmax$, 
$\J$, $\JThat$ and $\twop{j}$ be as introduced right before Theorem \ref{thm:adaptation}, and define
\begin{equation}\label{eq:defJ*}
\JTstar
:=\min \left\{ j\in \J:   C^{(B)}\norm{\nu}_{B_{p,\infty}^s} 2^{-sj } \leq \sqrt{{\twop{j}}/{\Tlambda}} \right\},
\end{equation}
where $C^{(B)}$ is as in \eqref{eq:upperbound}, so that
\begin{equation}\label{eq:J*sim}
2^{\JTstar} \sim \left(   C^{(B)}\norm{\nu}_{B_{p,\infty}^s} \right)^{\frac{2}{2s+1}} 2^{J_{\Tlambda} }.
\end{equation}
Then, analogously to the proofs of classical Lepski\u{i}'s method, a key quantity to study is $\Pr\big( \JThat = j \big)$ for any $j>\JTstar$; this is the content of Lemma \ref{lem:JhatT}, which precedes the actual proof of Theorem \ref{thm:adaptation}. 
To prove it, we have to  understand the concentration properties of 
$\norm{\nuThat-\nu}_p$: 
by \eqref{eq:presplitnunhat-nu} with $\hat{\nu}=\nuThat$ and by the triangle inequality
, it can be bounded, on $\Omegantilde\left( 2^J, r\right)$  for all $J\in \Nz$ and $0\leq r < \min\{1,e^{-2\lambdaDn}/(4\lambdaDn)\}$, for any $H\geq0$ if $p\in[1,2)$ 
and for $T$ large enough
, as
\begin{align}\label{eq:allpreconcentration}
\norm{\nuThat - \nu}_p  \leq & \norm{\left(\nuThat - \nu\right)^{(MS)}}_p +  \norm{\left(\nuThat - \nu\right)^{(R)}}_p +  \norm{\left(\nuThat - \nu\right)^{(B)}}_p.
\end{align}
The last term does not carry any stochasticity, and concentration of the first two terms is studied in Proposition \ref{prop:mainstochasticconcentration} and Lemma \ref{lem:remainderconcentration}. To show the former we first introduce 
some preparatory results. 

Recall from \eqref{eq:FIvarphintilde-1(-cdot)} that 
$\FII{ \varphintilde^{-1}(-\cdot)}$ is a finite (signed) measure on $\R$
, and from Section \ref{sec:unification} that $\Pn^{J)}$ is the law of $\ZJ{1}$ and $\widehat{\Pn \,}^{J)}_{\!n}$ is its empirical version on $\Omegantilde\left(2^J,r\right)$ with $J,r$ as in \eqref{eq:n0boundsOmegantilde} due to $n_0<n$ by the second inequality therein. 
The inequalities in \eqref{eq:n0boundsOmegantilde} also imply that $n_0\!
>\!0$ and $n\!-\!n_0\!\geq \!n/2(1\!-\!e^{-\lambdaDn})\!>\!0$. Then, using this together with \eqref{eq:mainstochastic1}, \eqref{eq:tildePn-tildeP2}, the same arguments leading to \eqref{eq:mainstochastic2}, 
and the triangle inequality, 
we have that, on $\Omegantilde\left( 2^J, r\right)$ with $J\in \Nz$ and $r\leq \frac{1}{\lambdaDn}\frac{1-\exp(-\lambdaDn)}{1+\exp(-\lambdaDn)}$, and for all $H\geq0$ when $p\in[1,2)$, and all $x\in \R$,
\begin{align}\label{eq:mainstochasticpreconcentration}
\Dn \abs{  (\nuThat-\nu)^{(MS)} }(x) \leq & {\frac{n-n_0}{n_0}}\abs{\left( \widehat{\Pn \,}^{J)}_{\!n} - \Pn^{J)}\right) \left(\FII{ \varphintilde^{-1}(-\cdot)} \ast K_J(x,\cdot) \right)} \\
&+ \abs{e^{\lambdanhatDn} - e^{\lambdaDn}} \abs{\FII{ \varphintilde^{-1}(-\cdot)} \left( \Pn^{J)} \ast K_J(x,-\cdot) \right)}.\notag
\end{align}
Concentration of the second summand will follow from Lemma \ref{lem:OmegantildesubseteqOmegan}. The first summand, however, is a product of a parametric and a nonparametric quantity and concentration of the latter has to be derived directly
. To do so, we follow the strategy in Section 5.1.2 in \cite{GN16} to derive the concentration properties of the usual convolution and projection density estimators using Talagrand's inequality:
for $p\in[1,\infty)$ we note that, due to the separability of $\Lp$ and the Hahn--Banach theorem, if $f\!\in\! \Lp$ there exists a countable subset $B_0^q$, $q$ conjugate to $p$, of the unit ball of $\L{q}$, such that, for all $f\in \Lp$,
\begin{equation*}
\norm{f}_p=\sup_{g\in B_0^q} \left\lvert \int_{\R} f(x)g(x)dx\right\rvert;
\end{equation*}
this, together with Fubini's theorem, 
implies that, for any $p\in[1,\infty)$,
\begin{equation}\label{eq:representationnormdual}
\left(n-n_0\right)\norm{\left( \widehat{\Pn \,}^{J)}_{\!n} - \Pn^{J)}\right) \left(\FII{ \varphintilde^{-1}(-\cdot)} \ast K_J(\cdot,\cdot) \right)}_p= 
\sup_{G \in \G} \left\lvert \sum_{i=1}^{n-n_0} G\left(\ZJ{i}\right)\right\rvert 
\end{equation}
for the class of functions $\G=\G_p(J)$ indexed by the $g\in B_0^q$ given by
\begin{align*}
\G&
:
=\left\{ x\mapsto \left(\FII{ \varphintilde^{-1}(-\cdot)}  \!\ast\! \!\int_{\R}\! K_J(y, \cdot)g(y)dy\right)\! (x) \right.\\
&\qquad \qquad \,\, \left.-\Pn^{J)} \!\left(\FII{ \varphintilde^{-1}(-\cdot)}  \!\ast \!\!\int_{\R} \!K_J(y, \cdot)g(y)dy \right) \!
\right\}\!;
\end{align*}
in addition, for $p=\infty$, we trivially have that \eqref{eq:representationnormdual} is true for 
\begin{align*}
\G=\G_{\infty}(J):=\left\{ x\mapsto \FII{ \varphintilde^{-1}(-\cdot)} \ast K_J(x,  \cdot) - \Pn^{J)} \!\left(\FII{ \varphintilde^{-1}(-\cdot)}  \ast K_J(x, \cdot)\right) 
\right\}.
\end{align*}
Representation  \eqref{eq:representationnormdual}, together with the following lemma, allows the use of Talagrand's inequality (Bousquet version, cf. Theorem 3.3.9 in \cite{GN16}) to describe the concentration of the quantity therein around its mean. 

\begin{lem}\label{lem:Usigma2bounds}
	Let $\mu=\nu/\lambda \in\Lp$ for some $p\in [1,\infty]$ and let $J\in \Nz$. 
	Then, 
	\begin{equation*}
	\sup_{G \in \G}\norm{G}_{\infty}  \leq 2e^{\lambdaDn} \norm{\barK}_p 2^{J(1-1/p)}  =: \bar G 
	\end{equation*}
	and
	\begin{equation*}
	\sup_{G \in \G} \E G^2\left(\ZJ{1}\right) \leq  e^{\lambdaDn} \norm{\mu}_p  \norm{\barK}_{2(1-1/(1+p))}^2   2^{J(1-1/p)}=:\sigma^2.
	\end{equation*}
\end{lem}

From now on we write $\Prunder{B}(A):=\Pr(A\cap B)$. 

\begin{proof}

	We prove the result for $p\in [1,\infty)$, as the case $p=\infty$ follows by the same (in fact, simpler) arguments.	To control the first quantity 
	we use that, by Young's inequality for convolutions, \eqref{eq:FIvarphintilde-1leqelambdaDnPn} and $\norm{\Pn}_{TV}= \norm{\Pn^{J)}}_{TV}=1$,
	\begin{equation*}
	\sup_{G \in \G}\norm{G}_{\infty}\leq 2e^{\lambdaDn} \sup_{g\in B_0^q}\norm{\int_{\R} |K_J(y, \cdot)|| g(y)|dy}_\infty\leq 2e^{\lambdaDn} \norm{\barK}_p 2^{J(1-1/p)},
	\end{equation*}
	where the last inequality follows by the generalised H{\"o}lder's inequality and Condition \ref{cond:2}. 
	To control the second quantity 
	we remark that $G\left(\ZJ{1}\right)$ is centred so this, together with the same arguments we just used, implies that 
	\begin{align*}
	\sup_{G \in \G} \E G^2\left(\ZJ{1}\right) &\leq \sup_{g\in B_0^q} \bigintssss_{\R} \left( \FII{ \varphintilde^{-1}(-\cdot)}  \ast \int_{\R} K_J(y, \cdot)g(y)dy  \right)^2\! (x)  \Pn^{J)} (x)\,dx\\
	&\leq  \norm{\Pn^{J)}}_p \:\:\sup_{g\in B_0^q} \norm{ \FII{ \varphintilde^{-1}(-\cdot)}  \ast \int_{\R} K_J(y, \cdot)g(y)dy }_{2q}^2 \\
	&\leq \norm{\mu}_p e^{\lambdaDn}  2^{J(1-1/p)} \norm{\barK}_{2p/(1+p)}^2,
	\end{align*}
	where in the last inequality we used \eqref{eq:boundnormp PnJ} with $z=0$ and the general version of Young's inequality for convolutions (cf. Proposition 8.9 in \cite{F99}). Therefore, the conclusion follows by rewriting the subscript of the last quantity.

\end{proof}

\begin{prop}\label{prop:mainstochasticconcentration}
	Let $\supnDn <\!\infty$ and,  
	for some $0<\underline{\lambda} \leq \bar{\lambda}<\infty$ and $B_1,B_2\geq 1$, assume $\nu=\lambda \mu$ with $\lambda\in[\underline{\lambda},\bar{\lambda}]$, $\norm{\mu}_p\leq B_1$ for some $p\in[1,\infty]$ and,  if $p\in[1,2)$, $\norm{\mu}_{1,\alpha} \leq B_2$ for some $\alpha>2/p-1$. 
	Let $\Jmax$, $\J$, $\twop{l}$ and $\mathcal{V}$ be as introduced in \eqref{eq:Jmax}, \eqref{eq:twopl} and \eqref{eq:mathcalV}. 
	Additionally, let $H\geq 0$ if $p\in[1,2)$ and $H=\infty$ if $p\in[2,\infty]$, $r\leq \frac{1}{\lambdaDn}\frac{1-\exp(-\lambdaDn)}{1+\exp(-\lambdaDn)}$ and $M^{(MS)}\geq \max\!\left\{1,C_2^{(MS)}\right\}^2$, where $C_2^{(MS)}(\alpha,\bar \lambda, \barK, B_1,B_2,p)$ is as defined in \eqref{eq:C2MS} and $\barK$ is as introduced in Condition \ref{cond:2}. 
	Then, for any $j,l\!\in\! \J$ such that $j\!\leq\! l$, and any T large enough,
	\begin{align}\label{eq:mainstochasticconcentration}
	\sup_{ \nu\in \mathcal{V}_{p,s,\alpha, \alpha',\beta}\left(\underline{\lambda}, \bar{\lambda}, B_1, B_2,\infty, \infty\right) }&\PrOmegantilde \left( \norm{(\nuThat(j)-\nu)^{(MS)} }_p \geq M^{(MS)}  \sqrt{ \frac{\twop{l}}{T}} \right) \notag \\
	& \qqsub 
	\leq  e^{-\sqrt{M^{(MS)}}2^{j/p}(j+\log T)^{\one_{\{\infty\}}(p)}} +  e^{-C_3^{(MS)} \Tlambdaunderline}  + 2 e^{- C_4^{(MS)} 2^j},
	\end{align}
	where
	\begin{equation*}
	C_3^{(MS)}=C_3^{(MS)}\left(M^{(MS)}\right):=\frac{\left( {M^{(MS)}}/ \left(C_2^{(MS)}\right)^2 - 1\right)^2 }{2+\frac{2}{3}  {M^{(MS)}}/\left(C_2^{(MS)}\right)^2 }
	\end{equation*}
	and  $C_4^{(MS)}=C_4^{(MS)}\left(M^{(MS)}\right)$ with
	\begin{equation*}
	C_4^{(MS)}:=\frac{e^{-5\lambdabarsupnDn} \left({ M^{(MS)}} \right)^2}{\left(2 \norm{\barK}_1  B_1 + \supnDn M^{(MS)}\right)\left(2\bar\lambda \left(2 \norm{\barK}_1 B_1 + \supnDn M^{(MS)}\right)  + \frac{2}{3}{ M^{(MS)}}\right) }.
	\end{equation*}
	
\end{prop}

\begin{proof}
	
	First note that, by Minkowski's inequality for integrals, Condition \ref{cond:2} 
	and \eqref{eq:FIvarphintilde-1leqelambdaDnPn},
	\begin{equation*}
	\norm{\FII{ \varphintilde^{-1}(-\cdot)} \left( K_j(x,-\cdot) \ast \Pn^{J)}\right)}_p \leq e^{\lambdaDn}   \norm{2^j\barK\left(2^j \cdot\right)}_1  \norm{\Pn^{J)}}_p \leq e^{\lambdaDn} \norm{\barK}_1  \norm{\mu}_p,
	\end{equation*}
	where in the last inequality we used  \eqref{eq:boundnormp PnJ} with $z=0$. Then, due to the conditions on $H,j,r$ assumed herein, we trivially use \eqref{eq:mainstochasticpreconcentration}, \eqref{eq:representationnormdual} and the lower bound for $n_0$ in \eqref{eq:n0boundsOmegantilde} to bound the left hand side in the display of the proposition by
	\begin{align*}
	&\PrOmegantilde\left( \sup_{G \in \G_p(j)} \left\lvert \sum_{i=1}^{n-n_0} G\left(\ZJ{i}\right)\right\rvert  \geq \frac{M^{(MS)}}{4e^{\lambdaDn}} \sqrt{ {T \twop{l}}}\right)\\
	&\qquad + \PrOmegantilde \left( \abs{e^{\lambdanhatDn} - e^{\lambdaDn}} \geq \frac{\Dn M^{(MS)}}{2 e^{\lambdaDn}\norm{\barK}_1  \norm{\mu}_p} \sqrt{ \frac{\twop{l}}{T}}\right).
	\end{align*}
	We treat these two terms separately and in reverse order. In view of  \eqref{eq:elambdanhatDn-elambdaDn>stuff}, $l\in \J$ and the bound for $2^{\Jmax}$ in \eqref{eq:Jmax}, the second summand is bounded above by
	\begin{align*}
	\Pr \left( \abs{n_0-\E n_0} \geq \frac{ M^{(MS)}e^{-3\lambdaDn} \sqrt{T \twop{l} } }{2 \norm{\barK}_1 \! \norm{\mu}_p \!+\! \Dn \!M^{(MS)}} \right) \leq 2 \exp\left( -  C_4^{(MS)} 2^j\right)
	\end{align*}
	where in the inequality we applied the Bernstein type-of inequality in the display right after \eqref{eq:elambdanhatDn-elambdaDn>stuff}, the bound for $2^l\leq2^{\Jmax}$ again, $l\geq j$ and $\norm{\mu}_p\leq CB_1$; we remark that the bound in the last display is uniform in $\mathcal{V}$. To bound the first summand, let $A_{n-n_0}=A_{n-n_0}\left( \ZJ{1}, \ldots, \ZJ{n-n_0}\right)$  denote the set inside the large brackets. Then, in view of the present conditions on $j,r$, we can use \eqref{eq:n0boundsOmegantilde} to justify that
	\begin{equation}\label{eq:PrAn-n0leqPrAm}
	\PrOmegantilde\left(A_{n-n_0}\right) \leq  \E \left[\one_{n_0=n-m} \one_{m\geq n/2(1-e^{-\lambdaDn}) 
	} \Pr\left(A_m\right)\right], 
	\end{equation}
	where the expectation and the probability on the right hand side are over $n_0$ and the non-zero increments, respectively. To bound the latter we use Lemma \ref{lem:Usigma2bounds}: in view of the bounds there, we can apply inequality (3.101) in Theorem 3.3.9 in \cite{GN16}  with \begin{align*}
	\nu_m
	:=  m e^{\lambdaDn} 2^{J(1-1/p)} \left( 4  \norm{\barK}_p  \E\sup_{G \in \G_p(j)} \left\lvert \sum_{i=1}^{m} G\left(\ZJ{i}\right)\right\rvert 
	+  \norm{\mu}_p  \norm{\barK}_{2(1-1/(1+p))}^2   \right),
	\end{align*}
	and, using the same elementary simplification of (5.35) resulting in (5.36) in \cite{GN16}, we obtain that, for any $p\in[1,\infty]$ and $x \geq 0$,
	\begin{align}\label{eq:mainstochasticconcentrationpre}
	\Pr \left( \sup_{G \in \G_p(j)} \left\lvert \sum_{i=1}^{m} G \left(\ZJ{i}\right)\right\rvert \right.  &\geq \frac{3}{2}  \E   \sup_{G \in \G_p(j)} \left\lvert \sum_{i=1}^{m} G  \left(\ZJ{i}\right)\right\rvert  \notag \\
	&\qquad \qquad \,\,+  \sqrt{ 2 m e^{\lambdaDn}  \norm{\mu}_p   \norm{\barK}_{2\left(1 - \frac{1}{1+p}\right)}^2   2^{j\left(1 - \frac{1}{p}\right)}  x} \notag \\
	& \qquad \qquad \qquad \qquad \qquad \quad + \left.\frac{14 e^{\lambdaDn} \norm{\barK}_p   }{3}2^{j\left(1-\frac{1}{p}\right)} x\right) \leq e^{-x}.
	\end{align}
	Due to \eqref{eq:representationnormdual}, Fubini's theorem, Minkowski's inequality for integrals and \eqref{eq:FIvarphintilde-1leqelambdaDnPn}, the expectation in the last display is bounded above by
	\begin{align*}
	& m \E \norm{\int_{\R} \left( \int_{\R} K_j(\cdot, y) \left( \widehat{\Pn \,}^{J)}_{\!m} - \Pn^{J)}\right) (d(y+z))  \right) \FII{ \varphintilde^{-1}(-\cdot)} (dz)}_p \\
	&\qquad \leq  e^{\lambdaDn} m \int_{\R} \E \norm{ \int_{\R} K_j(\cdot, y) \left( \widehat{\Pn \,}^{J)}_{\!m} - \Pn^{J)}\right) (d(y-z)) } \Pn(dz)\\
	& \qquad\qquad\leq e^{\lambdaDn} L_p M_p' \sqrt{m 2^j \left(1+j\one_{\{\infty\}}(p)\right)},
	\end{align*}
	where the last inequality holds for $T$ large enough, $\lambda\geq \underline{\lambda}>0$ and $m\geq n/2(1-e^{-\lambdaDn})$, and it follows from the analysis of \eqref{eq:EEmainstochastic} because the only use of the indicator of $\Omegantilde\left(2^J, r\right)$ therein is that, for any $J\in \Nz$, the bound for $m$ on the right hand side of \eqref{eq:PrAn-n0leqPrAm} is satisfied and, furthermore, for $T$ large enough only depending on $\underline \lambda$ and in view of the bounds for $\eta$ in \eqref{eq:Jmax} and \eqref{eq:n/2(1-e-lambdaDn)>Tlambda},  
	\begin{equation*}
	2^j\leq 2^{\Jmax} <n/2\left(1-e^{-\underline\lambda \Dn}\right)\leq n/2\left(1-e^{-\lambdaDn}\right) \quad \mbox{ if } p\in [2,\infty) 
	\end{equation*}
	and
	\begin{equation*}
	j2^j\leq \Jmax2^{\Jmax} <n/2\left(1-e^{-\underline\lambda \Dn}\right)\leq n/2\left(1-e^{-\lambdaDn}\right) \quad  \mbox{ if } p=\infty.
	\end{equation*}
	Hence, taking $x=\sqrt{M^{(MS)}}2^{j\left(\frac{1}{p}-1\right)} \twop{l}$ and recalling that $j\leq l$ and ${M^{(MS)}}\geq 1$, we have that for $T$ large enough, for any  $\lambda\geq \underline{\lambda}>0$ and $m\geq n/2(1-e^{-\lambdaDn})$, the right hand side in the outermost brackets in \eqref{eq:mainstochasticconcentrationpre} is bounded above by
	\begin{align*}
	&\sqrt{M^{(MS)}}e^{\lambdaDn} \sqrt{m \twop{l}} \left(\frac{3}{2}   L_p M_p' +     \sqrt{ 2  \norm{\mu}_p   \norm{\barK}_{2\left(1 - \frac{1}{1+p}\right)}^2 }+ \frac{14 \norm{\barK}_p   }{3}\sqrt{\frac{ \twop{l}  }{m}} \right) 
	\end{align*}
	and, using $l\leq {\Jmax}$ and \eqref{eq:Jmax} together with $m\geq n/2(1-e^{-\lambdaDn})$, $e^{-x}\leq (1+x)^{-1}$ valid for any $x\geq0$,  $\norm{\mu}_p\leq CB_1$, $\norm{\mu}_{1,\alpha} \leq B_2$ if $p\in[1,2)$,  $\norm{\mu}_{p/2} \leq \norm{\mu}_p^{\frac{p-2}{p-1}}$ if $p\in(2,\infty)$ as noted at the end of Remark \eqref{rem:ctsminimax}, $M_p'\leq \bar M_p$, where the latter is defined in \eqref{eq:defbarMp}, and  $\Dn\leq \supnDn$, for $T$ large enough this display is bounded by	
	\begin{equation*}
	\sqrt{m  \twop{l}}\frac{ \sqrt{M^{(MS)}}C_2^{(MS)}}{4\sqrt{\lambda}e^{\lambdabarsupnDn}}.
	\end{equation*}
	Consequently, for any $m\geq n/2(1-e^{-\lambdaDn})$,  $\lambda\geq \underline{\lambda}>0$, $\mu$ satisfying the assumptions herein and $T$ large enough,
	\begin{align*}
	\Pr\left(A_m\right) &= \one_{\left\{m\leq T\left( \sqrt{\lambda}\frac{e^{\lambdabarsupnDn}}{e^{\lambdaDn}}\frac{\sqrt{M^{(MS)}}}{C_2^{(MS)}}\right)^2\right\}}  \Pr \left(A_m\right) + \one_{\left\{m> T\left( \sqrt{\lambda}\frac{e^{\lambdabarsupnDn}}{e^{\lambdaDn}}\frac{\sqrt{M^{(MS)}}}{C_2^{(MS)}}\right)^2\right\}} \Pr \left(A_m\right)\\
	&\leq e^{-\sqrt{M^{(MS)}}2^{j\left(\frac{1}{p}-1\right)} \twop{l}}+ \one_{\left\{m> T\left( \sqrt{\lambda}\frac{e^{\lambdabarsupnDn}}{e^{\lambdaDn}}\frac{\sqrt{M^{(MS)}}}{C_2^{(MS)}}\right)^2\right\}},
	\end{align*}
	and, recalling that $\twop{l}:=2^l(1+(l+\log T)\one_{\{\infty\}}(p))$ and $l\geq j$, and due to $e^{-x}\leq (1+x)^{-1}$ for any $x\geq0$ and $\Dn\leq \supnDn$, the right hand side in \eqref{eq:PrAn-n0leqPrAm} is bounded by
	\begin{equation*}
	e^{-\sqrt{M^{(MS)}}2^{j/p} \left(1+(j+\log T)\one_{\{\infty\}}(p)\right)}+ \Pr\left(En_0-n_0>  \lambda T \left(
	\frac{M^{(MS)}}{\left(C_2^{(MS)}\right)^2} - 1
	\right)\right). 
	\end{equation*}
	Note that, 
	in view of $M^{(MS)}\geq \left(C_2^{(MS)}\right)^2
	$, the quantity on the right hand side of the inequality inside the probability is positive. Then, applying Bernstein's inequality as argued right after \eqref{eq:elambdanhatDn-elambdaDn>stuff} and using that $\lambda\geq   \underline{\lambda}$, we have that the second summand in the last display is bounded by the second summand in \eqref{eq:mainstochasticconcentration}, and this concludes the proof.

\end{proof}

\begin{lem}\label{lem:remainderconcentration}
	Let $\supnDn <\!\infty$ and, 
	for some $0<\underline{\lambda} \leq \bar{\lambda}<\infty$, $\beta>1$ and $B\geq 1$, assume $\nu=\lambda \mu$ with $\lambda\in[\underline{\lambda},\bar{\lambda}]$ and $\norm{ \log^{\beta} (\max\{|\cdot|,e\}) \mu }_1\leq B$. Let $\Jmax$, $\J$ and $\twop{l}$ be as introduced in \eqref{eq:Jmax} and \eqref{eq:twopl}. 
	Additionally,  let $H=\infty$ if $p\in[2,\infty)$ and, otherwise, let it be as in \eqref{eq:Hn'} satisfying, 
	if $p\in(1,2)$, that  $\theta_H-2 \theta_H' -1\leq 0$ and $\theta_H'<\frac{1-1/p}{2(1/p-1/2)}$. 
	Then, if $0 \leq r <e^{-2\lambdaDn}/(4\lambdaDn)$, we have that, for any  $j,l\in \J$ such that $ j\leq l$, 
	all  $M^{(R)}>0$, $p\in[1,\infty]$ and  $T$ large enough depending on $\underline{\lambda}, \bar{\lambda}$ and $B$,
	\begin{align}\label{eq:remainderconcentration}
	\sup_{\stackrel{ \text{\scriptsize $\nu=\lambda \mu: \lambda\in[\underline{\lambda},\bar{\lambda}],$}}{ \norm{ \log^{\beta} (\max\{|\cdot|,e\}) \mu }_1\leq B}}\!\!\PrOmegantilde \! &\left( \norm{(\nuThat(j)-\nu)^{(R)} }_p \geq M^{(R)}  \sqrt{ \frac{\twop{l}}{T}} \right) \notag\\
	& \qqsub \!\! \leq \left\{ \begin{array}{ll}
	3T^{-C_2^{(R)} \left(\log T\right)^{\frac{\eta}{2}\left( 1-\theta_H\right)-1
	}},& \mbox{ if } p=1,\\
	3\exp\left(-C_2^{(R)} T^{\frac{1}{2} - \left(\frac{1}{p}-\frac{1}{2}\right) \left( 2\theta_H' +1\right) }  \right), & \mbox{ if } p\in(1,2),\\
	3\exp\left(-C_2^{(R)} T^{\frac{1}{p}} \right), & \mbox{ if } p\in[2,\infty),\\
	3T^{-C_2^{(R)} \left(\log T\right)^{\frac{1}{2}(\eta -1)}}, & \mbox{ if } p=\infty.
	\end{array}
	\right.
	\end{align}
	where $\eta>2/(1-\theta_H)$ if $p=1$ and $\eta>1$ if $p=\infty$, and, for ${L} $ and 
	$C_1^{(R)}:=C^{(R)}\left(p,K\right)$ as in Theorem \ref{thm:varphivarphin:c} 
	and  
	Proposition \ref{prop:remainder}, respectively,
	\begin{equation*}
	C_2^{(R)}=C_2^{(R)}\left(M^{(R)}\right):= {L} \frac{e^{-7\lambdabarsupnDn}\left(M^{(R)}\right)^2}{4{ \lambdabarsupnDn} C_1^{(R)}}.
	\end{equation*}	
	
\end{lem}

\begin{proof}

	Note that, due to $j\geq 1>0$ and $0\leq r< e^{-2\lambdaDn}/(4\lambdaDn)$, we can use \eqref{eq:boundLpnormremainderterm} and \eqref{eq:boundLpnormremainderterm2} together with Lemma \ref{lem:remindertermLog}, $j\leq l \leq \Jmax$ and the fact that $\inf_{\xi \in \R} |\varphintilde(\xi)|\geq e^{-\lambdaDn}$, to conclude that, on $\Omegantilde\big(2^j, r\big)$,
	\begin{align*}
	\norm{(\nuThat(j)-\nu)^{(R)} }_{p} &\leq C_1^{(R)} \frac{e^{2\lambdaDn}}{\Dn} \left(  H_{T,j}^{\left(\frac{2}{p}-1\right)}2^{\frac{j}{p}} \one_{[1,2)}(p) + 2^{j\left(1-\frac{1}{p}\right)} \one_{[2,\infty]}(p) \right) \\
	& \quad \times \sup_{|\xi|\leq 2^{\Jmax}\Xi} \left| \varphinhattilde(\xi)-\varphintilde(\xi)\right|^2.
	\end{align*}
	By simple algebra and using that $\Dn\leq \supnDn$, it follows that  the  left hand side in \eqref{eq:remainderconcentration} is bounded above by
	\begin{align}\label{eq:remainderintermediateconcentration}
	&\Pr\bigg(    \sup_{|\xi|\leq 2^{\Jmax} \Xi} n\left| {\varphinhattilde(\xi)-\varphintilde(\xi)}\right| \geq 2 e^{\lambdaDn} 
	\Tlambda \frac{M^{(R)}e^{-2\lambdasupnDn}}{2\lambda\sqrt{C_1^{(R)}\supnDn} } \left( T^{-1} \twop{l} \right)^{1/4}  \notag \\
	& \qquad \qquad\qquad\qquad\qquad\qquad\qquad \, \times \left(  H_{T,j}^{-\left(\frac{1}{p}-\frac{1}{2}\right)}2^{-\frac{j}{2p}} \one_{[1,2)}(p) + 2^{\frac{j}{2}\left(\frac{1}{p}-1\right)} \one_{[2,\infty]}(p) \right)\bigg).
	\end{align}
	In the rest of the proof we deal with the cases $p\in[1,2)$ and $p\in[2,\infty]$ separately. In view of \eqref{eq:Hn'}, for $p\in[1,2)$  the right hand side in the inequality (ignoring the constant quotient and the term before it) is
	\begin{align*}
	T^{-\frac{1}{4}}2^{\frac{l}{4}} T^{-\frac{\theta_H}{2}\left(\frac{1}{p}-\frac{1}{2}\right)} 2^{\frac{j}{2}\left( \theta_H - 2\theta_H' \right)\left(\frac{1}{p}-\frac{1}{2}\right)}2^{-\frac{j}{2p}}&= T^{-\frac{1}{2}\left(\frac{1}{2}+ \theta_H\left(\frac{1}{p}-\frac{1}{2}\right)\right)} 2^{\frac{l-j}{4}+\frac{j}{2}\left(\frac{1}{p}-\frac{1}{2}\right) \left( \theta_H - 2\theta_H' -1\right)}\\
	&\geq T^{-\frac{1}{2}\left(\frac{1}{2}+ \theta_H\left(\frac{1}{p}-\frac{1}{2}\right)\right)} 2^{\frac{l}{2}\left(\frac{1}{p}-\frac{1}{2}\right) \left( \theta_H - 2\theta_H' -1\right)}\\
	&=: t(T,l), 
	\end{align*}
	where the inequality follows due to $j\leq l$ and $\theta_H - 2\theta_H' -1\leq  0$, the latter being an assumption if $p\in(1,2)$ and true for $p=1$ because, then, $\theta_H'=0$ and $\theta_H<1$. Notice that, due to $\theta_H - 2\theta_H' -1\leq 0$ again together with the bound for $2^{\Jmax}$ in \eqref{eq:Jmax}, for all $l\in \J$ and $T>0$
	\begin{equation*}
	T^{-\frac{1}{2}\left(\frac{1}{2}+ \left(\frac{1}{p}-\frac{1}{2}\right) \left( 2\theta_H'+1\right)\right)} \left(\log T\right)^{-\frac{\eta}{2}\left(\frac{1}{p}-\frac{1}{2}\right) \left( \theta_H - 2\theta_H' -1\right)  } \leq t(T,l) \leq T^{-\frac{1}{2}\left(\frac{1}{2}+ \theta_H\left(\frac{1}{p}-\frac{1}{2}\right)\right)},
	\end{equation*}
	where $\eta>4$ if $p=1$ and $\eta=0$ if $p\in(1,2)$, and, in view of the bounds for $\theta_H$ and $\theta_H'$ in 
	\eqref{eq:thetaHthetaH'} 
	together with assumption $\theta_H'<\frac{1-1/p}{2(1/p-1/2)}$ for $p\in(1,2)$,
	\begin{equation*}
	\frac{1}{2}+ \left(\frac{1}{p}-\frac{1}{2}\right) \left( 2\theta_H'+1\right)=1 \quad \mbox{ if } p=1, \qquad \frac{1}{2}+ \left(\frac{1}{p}-\frac{1}{2}\right) \left( 2\theta_H'+1\right)<1 \quad \mbox{ if } p\in(1,2),
	\end{equation*}
	and 
	\begin{equation*}
	\frac{1}{2}+ \theta_H\left(\frac{1}{p}-\frac{1}{2}\right)>5/6>0.
	\end{equation*}
	Then,  due to $\theta_H - 2\theta_H' -1=\theta_H -1<0$ and $\eta>2/(1-\theta_H)$ {if} $p=1$ together with 
	$2^{\Jmax}=O(\Tlambda)$, $\lambda\in [\underline{\lambda},\bar \lambda]$ and $\norm{ \log^{\beta} (\max\{|\cdot|,e\}) \mu }_1\leq B$, for $T$ large enough depending on $\underline{\lambda}, \bar{\lambda}$ and $B$ we can use Lemma \ref{lem:OmegantildesubseteqOmegan} to conclude that, for $T$ large enough, the supremum over the measures $\nu$ considered herein of the quantity in \eqref{eq:remainderintermediateconcentration} is bounded above by
	\begin{equation*}
	3\exp\left(  - C_2^{(R)}T (t(T,l))^2\right) \leq \left\{ \begin{array}{ll}
	3T^{-C_2^{(R)} \left(\log T\right)^{\frac{\eta}{2}\left( 1-\theta_H\right)-1
	}},& \mbox{ if } p=1,\\
	3\exp\left(-C_2^{(R)} T^{\frac{1}{2} - \left(\frac{1}{p}-\frac{1}{2}\right) \left( 2\theta_H' +1\right) }  \right), & \mbox{ if } p\in(1,2).
	\end{array}
	\right.
	\end{equation*}
	This shows the result for $p\in[1,2)$. In view of \eqref{eq:remainderintermediateconcentration} and $j\leq l$, for $p\in[2,\infty]$ the right hand side in the inequality (ignoring the constant quotient and the term before it) therein is 
	\begin{align*}
	T^{-\frac{1}{4}}2^{\frac{l}{4}} \left(1+(l+\log T)\one_{\{\infty\}}(p)\right)^{\frac{1}{4}} 2^{\frac{j}{2}\left(\frac{1}{p}-1\right)} &\geq T^{-\frac{1}{4}}2^{\frac{l}{2}\left(\frac{1}{p}-\frac{1}{2}\right)}\left(1+\log T\one_{\{\infty\}}(p)\right)^{\frac{1}{4}} \\
	&=:t(T,l), 
	\end{align*}
	and, 
	in view of the bound for $2^{\Jmax}$ in \eqref{eq:Jmax}, we have that, for all $ l\in \J$ and $T\geq e$, 
	\begin{equation*}
	T^{\frac{1}{2}\left(\frac{1}{p}-1\right)} \left(1+\left(\log T\right)^{\frac{1}{4}\left( \eta+1\right) }\one_{\{\infty\}}(p)\right)\leq t(T,l)\leq T^{-\frac{1}{4}}\left(1+\log (T)\one_{\{\infty\}}(p)\right)^{\frac{1}{4}},
	\end{equation*}
	where $\eta>1$
	. Then, due to this assumption on $\eta$ together with $2^{\Jmax}=O(\Tlambda)$, $\lambda\in [\underline{\lambda},\bar \lambda]$ and $\norm{ \log^{\beta} (\max\{|\cdot|,e\}) \mu }_1\leq B$, for $T$ large enough depending on $\underline{\lambda}, \bar{\lambda}$ and $B$ we can use Lemma \ref{lem:OmegantildesubseteqOmegan} to conclude that, for $T$ large enough, the supremum over the measures $\nu$ considered herein of the quantity in \eqref{eq:remainderintermediateconcentration} is bounded above by
	\begin{equation*}
	3\exp\left(  - C_2^{(R)}T (t(T,l))^2\right) \leq \left\{ \begin{array}{ll}
	3\exp\left(-C_2^{(R)} T^{\frac{1}{p}} \right), & \mbox{ if } p\in[2,\infty),\\
	3T^{-C_2^{(R)} \left(\log T\right)^{\frac{1}{2}(\eta -1)}}, & \mbox{ if } p=\infty.
	\end{array}
	\right.
	\end{equation*}
	
\end{proof}

\begin{lem}\label{lem:JhatT}
	
	Let $\supnDn <\!\infty$ and 
	assume that $\nu\in \mathcal{V}_{p,s,\alpha, \alpha',\beta}\left(\underline{\lambda}, \bar{\lambda}, (B_i)_{i=1}^4\right)$, with $\mathcal{V}$  as defined in \eqref{eq:mathcalV}, for some $p\in[1,\infty]$, $s>0$, $\alpha>2/p-1$, $\alpha'>3/p-3/2$, $\beta>1$, $0<\underline{\lambda} \leq \bar{\lambda}<\infty$, and $B_1,B_2, B_3,B_4\geq 1$. 
	Let $\UT$ be as in \eqref{eq:Un}, $\Jmax, \J$ and $\JThat$ as introduced in \eqref{eq:Jmax} and \eqref{eq:twopl}
	, and $\JTstar$ as defined in \eqref{eq:defJ*}. Additionally,  let $H=\infty$ if $p\in[2,\infty)$ and, otherwise, let it be as in \eqref{eq:Hn'} satisfying, 
	if $p\in(1,2)$, that  $\theta_H-2 \theta_H' -1\leq 0$ and $\theta_H'<\frac{1-1/p}{2(1/p-1/2)}$. Then, if $M:=\tau - 1/{\sqrt{\underline \lambda}} 
	> \max\left\{1,C_2^{(MS)}\right\}^2$, where $C_2^{(MS)}
	$ is as in \eqref{eq:C2MS}
	, we have that, for every $j\in \J$ such that $j>\JTstar$, and any $T$ large enough,
	\begin{align*}\label{eq:hatJT}
	\sup_{ \nu\in  \mathcal{V}}\Pr \left( \JThat=j \right) &\leq  8e^{-\sqrt{M}2^{j/p} (j+\log T)^{\one_{\{\infty\}}(p)}} +  2\log(T)e^{-C_3^{(MS)}  \Tlambdaunderline}  \notag \\
	& \quad + 16 e^{- C_4^{(MS)} 2^j}  + 6 \log(T)e^{- C^{(C)} \Tlambdaunderline   }  \notag \\
	&\quad +6\log(T)\times \left\{ \begin{array}{ll}
	T^{-C_2^{(R)} \left(\log T\right)^{\frac{\eta}{2}\left( 1-\theta_H\right)-1
	}},& \mbox{if } p=1,\\
	\exp\left(-C_2^{(R)} T^{\frac{1}{2} - \left(\frac{1}{p}-\frac{1}{2}\right) \left( 2\theta_H' +1\right) }  \right), & \mbox{if } p\in(1,2),\\
	\exp\left(-C_2^{(R)} T^{\frac{1}{p}} \right), & \mbox{if } p\in[2,\infty),\\
	T^{-C_2^{(R)} \left(\log T\right)^{\frac{1}{2}(\eta -1)}}, & \mbox{if } p=\infty,
	\end{array}
	\right.
	\end{align*}
	where 
	\begin{equation*}
	C^{(C)}:=L\frac{e^{-3\lambdabarsupnDn}}{8(\lambdabarsupnDn)^2}\min\left\{ \frac{ e^{-2\lambdabarsupnDn}}{2}, {1-e^{-\lambdabarsupnDn}}\right\}^2
	\end{equation*}
	and $C_3^{(MS)}=C_3^{(MS)}(M), C_4^{(MS)}=C_4^{(MS)}(M)$  and ${L},C_2^{(R)}=C_2^{(R)}(M), \eta$ are as in Proposition \ref{prop:mainstochasticconcentration} and Lemma \ref{lem:remainderconcentration}, respectively.

\end{lem}

\begin{proof}
	
	Note that, in view of 
	Lemma \ref{lem:bias}, it holds that, for any $T>0$ 
	and $l\in\J$ satisfying $l\geq \JTstar$, 
	\begin{align*}
	\norm{(\nuThat(l)-\nu)^{(B)} }_p &\leq C^{(B)}  \norm{\nu}_{B_{p,\infty}^s} 2^{-ls} +   \norm{ \nu}_{p,\alpha'} H_{T,l}^{-\alpha'} \one_{[1,2)}(p)\\
	&\leq C^{(B)}  \norm{\nu}_{B_{p,\infty}^s} 2^{-\JTstar \!s} +   \norm{ \nu}_{p,\alpha'} \left(\!\sqrt{\frac{2^l}{T}}\:\right)^{\alpha'\theta_H} 2^{-\alpha' l \theta_H'} \one_{[1,2)}(p)\\
	&\leq  \left(\frac{1}{\sqrt{ \lambda}}+  (\log T)^{-\frac{\eta}{2}(\alpha'\theta_H -1)} \norm{ \nu}_{p,\alpha'} \one_{[1,2)}(p)\right)\sqrt{\frac{\twop{l}}{T}},
	\end{align*}
	where in the last inequality we used the definition of $\JTstar$ in \eqref{eq:defJ*}, $ \JTstar\leq l$, and the facts $\alpha'\theta_H>1$,  $\sqrt{2^{\Jmax}/T}\leq (\log T)^{-\eta/2}$ and $\theta_H'\geq 0$, all of which immediately follow from the bounds for these parameters in \eqref{eq:thetaHthetaH'} and \eqref{eq:Jmax}. 
	Then, by the same arguments leading to \eqref{eq:splitnunhat-nu} 
	together with $\norm{ \nu}_{p,\alpha'} \leq \bar \lambda B_3<\infty$, we conclude that, uniformly in such $\nu$, for any $\epsilon>0$, for $T$ large enough and on $\Omegantilde\left(2^j,r\right)$ with $0\leq r<\min\{1, e^{-2\lambdaDn}/(4\lambdaDn)\}$ and  $j,l\in\J$ such that $l\geq \JTstar$, 
	\begin{align*}
	\norm{ \nuThat(l) -\nu}_p \leq & \norm{(\nuThat(l)-\nu)^{(MS)} }_p + \norm{(\nuThat(l)-\nu)^{(R)} }_p
	+ 
	\left(\frac{1}{\sqrt{\lambda}}+  \epsilon \one_{[1,2)}(p)\right)
	\sqrt{\frac{\twop{l}}{T}}.
	\end{align*}
	Furthermore, by the definition of $\JThat$ in \eqref{eq:JThat}, if $\JThat=j$  there is an $l\in \J$ with $l\geq j$ such that 
	\begin{equation*}
	\norm{ \nuThat(j-1) -\nuThat(l)}_p >\tau\sqrt{\frac{\twop{l}}{T}},
	\end{equation*}
	where we remark that $j-1\geq \JTstar$. Then, by a union bound, the triangle inequality, the second-to-last display, the fact that $j-1\leq l$ and defining  $M':=
	\tau - 1/{\sqrt{ \underline{\lambda}}} - \epsilon \one_{[1,2)}(p) 
	$
	, for $T$ large enough it follows that $\Pr \left( \JThat=j \right)$ is bounded above by
	\begin{align*}
	\sum_{l\in\J:l\geq j} & \PrOmegantilde \left( \norm{ \nuThat(j-1) -\nuThat(l)}_p > \tau\sqrt{\frac{\twop{l}}{T}} \right) + \Pr\left( \Omegantildecomplement\left(2^j,r\right)\right)\\
	&\quad\leq  \sum_{l\in\J:l\geq j} \left[ \PrOmegantilde \left( \norm{ (\nuThat(j-1)-\nu)^{(MS)} }_p > M'\sqrt{\frac{\twop{l}}{T}} \right) \right. \\
	&\quad\quad\qquad\qquad+ \PrOmegantilde\left( \norm{ (\nuThat(j-1)-\nu)^{(R)} }_p > M'\sqrt{\frac{\twop{l}}{T}} \right)  \\
	&\quad\quad\quad\qquad\qquad+ \PrOmegantilde \left( \norm{ (\nuThat(l)-\nu)^{(MS)} }_p > M'\sqrt{\frac{\twop{l}}{T}} \right)  \\
	&\quad\quad\quad\quad\qquad\qquad+\left. \PrOmegantilde\left( \norm{ (\nuThat(l)-\nu)^{(R)} }_p > M'\sqrt{\frac{\twop{l}}{T}} \right)\right] \\
	& \qquad + \Pr\left( \Omegantildecomplement\left(2^j,r\right)\right).
	\end{align*} 
	In addition, note that, for any constant $C>0$,  $p\in[1,\infty)$ and $T$ large enough,
	\begin{equation*}
	\sum_{l\in\J:l\geq j } e^{-C2^{l/p}} = \frac{e^{-C2^{j/p}}+e^{-C2^{(j+1)/p}} - e^{-C2^{(j+1)/p}(\Jmax - j +2) } }{1-e^{-C2^{(j+1)/p}}} \leq 4 e^{-C2^{j/p}},
	\end{equation*}
	where in the inequality we have used that $j>\JTstar$ so that, for any $C>0$ and $T$ large enough, $e^{-C2^{(j+1)/p}}\leq 1/2$. By the same arguments, we also have that, for any $C>0$ and $T$ large enough, $\sum_{l\in\J:l\geq j } e^{-Cl} \leq 2e^{-Cj}$. Recall from the proof of Theorem \ref{thm:upperbound} that, when $p\in [1,2)$, if $\mu$ satisfies Assumption \ref{ass2a} for some $\alpha>2/p-1$ and $q=1$, then Assumption \ref{ass2} is automatically satisfied. Moreover, $2^j\leq 2^{\Jmax}\leq T=O(\Tlambda)$ and, in view of $M:=\tau - 1/{\sqrt{\underline \lambda}}> \max\left\{1,C_2^{(MS)}\right\}^2$, $M'=\max\left\{1,C_2^{(MS)}\right\}^2 +\delta	- \epsilon \one_{[1,2)}(p)$ for some $\delta>0$ so, taking $\epsilon\leq \delta$,  $M'\geq \max\left\{1,C_2^{(MS)}\right\}^2$. 
	Consequently, the conclusion  follows by Proposition \ref{prop:mainstochasticconcentration} and Lemmata \ref{lem:remainderconcentration} and \ref{lem:OmegantildesubseteqOmegan:b} by choosing $r$ equal to, say, half of the quantity on the right hand side in \eqref{eq:rbounds}, and bounding this from below as done to arrive at \eqref{eq:rlowerupperbound} using $\lambda \in [\underline{\lambda},\bar \lambda]$.

\end{proof}


We need a definition prior to proving Theorem \ref{thm:adaptation}: for $C^{(R)}$ and $C^{(B)}$  as in Theorem \ref{thm:upperbound}, let
\begin{align}\label{eq:Co}
\!\!C^{(o)}\!:= 
& \left( \left(\frac{5}{4}e^{2\lambdabarsupnDn} +1\right)\one_{[1,\infty)} + 3 \underline{\lambda}^{-1-\vartheta_U}\right) \left( \underline{\lambda}^{-\frac{\theta_H+2\theta_H'}{2p}}\one_{[1,2)}(p)+\bar{\lambda}^{\frac{1}{p}}\one_{[2,\infty]}(p) \right) \\
& + 8\bar{\lambda}\left( \frac{5}{4}e^{2\lambdabarsupnDn} +1\right)2 \bar{\lambda}^{\sqrt{\tau-1/\sqrt{\underline{\lambda}}}} \!\left( C^{(B)} \bar{\lambda} B_1 \right)^{2\frac{\sqrt{\tau-1/\sqrt{\underline{\lambda}}} \log_2 e-1}{2s+1}} + C B_1 \notag \\
& +  \! \bar\lambda	\left(e^{3\lambdabarsupnDn}\!+\!1\right) C B_1\! +\!\bar\lambda^{ 1+\alpha' \frac{\theta_H}{2}+\frac{\alpha' \theta_H}{2s+1}} \underline{\lambda}^{-\frac{2 \alpha' \theta_H'}{2s+1}}\!\left(  C^{(B)} \right)^{\!\frac{\alpha' (\theta_H-2\theta_H')}{2s+1}}\!\!  B_1^{\frac{\alpha' \theta_H}{2s+1}}  B_3   \one_{[1,2)}(p)   \notag  \\
&+ C^{(R)}\bar \lambda^2 \supnDn e^{4\lambdabarsupnDn} \left( \!\underline\lambda^{-\theta_H\big(\frac{1}{p}-\frac{1}{2}\big)}\left( \bar\lambda C^{(B)}B_1 \right)^{\frac{1+(2/p-1)(2\theta_H'-\theta_H +1 )}{2s+1}}  \one_{[1,2)}(p) \right. \notag\\
& \qquad\qquad\qquad\quad \qquad\qquad\qquad \qquad \quad \,\,\quad \left.+\left( \bar\lambda C^{(B)}B_1 \right)^{\frac{2(1-1/p)}{2s+1}} \one_{[2,\infty]}(p) \right)\notag\\
& + \!  \underline\lambda^{-\vartheta_U} \!\left(\underline\lambda^{-\frac{\theta_H}{2p}} \!\left( \bar\lambda C^{(B)}\right)^{\!\big(2+\frac{\theta_H-2\theta_H'}{p}\big)\frac{1}{2s+1}}\!B_1^{\frac{\theta_H}{p(2s+1)}}\one_{[1,2)}(p) \right.\notag\\
& \qquad\qquad\qquad\quad  \qquad \quad \qquad \,\, \left.\left( \bar\lambda C^{(B)} B_1\right)^{\frac{2(1-1/p)}{2s+1}}\one_{[2,\infty]}(p) \right) \!.\notag
\end{align}

\begin{proof}[Proof of Theorem \ref{thm:adaptation}]

	Trivially, 
	\begin{equation}\label{eq:adaptationpre}
	\E\left[\norm{\nuThat(\JThat) - \nu}_p\right]=\E\left[\norm{\nuThat(\JThat) - \nu}_p \one_{ \JThat> \JTstar}\right] + \E\left[\norm{\nuThat(\JThat) - \nu}_p \one_{ \JThat\leq \JTstar}\right],
	\end{equation}
	and we analyse the two terms in order. In view of \eqref{eq:estimatornu} and the triangle inequality, the first is bounded, for any $r\geq0$, by
	\begin{align*}
	&\E\left[\norm{\nuThat(\JThat)}_p \one_{ \Omegantilde\left(2^{\JThat},r\right) \cap \{\JThat> \JTstar\} }\right] +\E\left[\norm{\nuThat(\JThat)}_p \one_{ \Omegan\cap \Omegantildecomplement\left(2^{\JThat},r\right) \cap \{\JThat> \JTstar\}}\right] \\
	& \qquad \qquad \qquad \qquad \qquad \qquad + \norm{\nu}_p \Pr\left( \JThat> \JTstar\right).
	\end{align*}
	Note that proving Lemma \ref{lem:Omegantildecomplement} reduced to finding an upper bound for the $\Lp$-norm on the right hand side of \eqref{eq:Omegantildecomplementboundpre}. Crucially, the bound is valid on the whole of $\Omegan$ and, thus, by the same arguments 
	it follows that for any $j \in \Nz$ and on $\Omegan$
	\begin{align}\label{eq:boundnuThatjLpnorm}
	\norm{\nuThat(j)}_p\leq  2^{j} \UT H_{T, j}^{1/p}\one_{[1,2)}(p) +2^{j(1-{1}/{p})} \UT\one_{[2,\infty]}(p).
	\end{align}
	Then, in view of the expressions for $\UT$ and $H_{T, \JThat}$ arising from \eqref{eq:Un} and \eqref{eq:Hn'}, respectively, 
	noting that $2^{\JThat}\leq \twoJmax \leq  T=O(\Tlambda)$ by \eqref{eq:Jmax} and using Lemma \ref{lem:OmegantildesubseteqOmegan:b}, if $\sqrt{\log^{1+2\delta}(e+T) / \Tlambda } \leq r\leq 1$ 
	for some $\delta>0$ the second term in the penultimate display is bounded above by
	\begin{align*}
	3 \frac{\Tlambda^{\vartheta_U}}{\lambda^{\vartheta_U}} \left( \frac{\Tlambda^{1+(\theta_H+2\theta_H')/(2p)}}{\lambda^{1+(\theta_H+2\theta_H')/(2p)}}\one_{[1,2)}(p)\! +\!\frac{\Tlambda^{1-{1}/{p}}}{\lambda^{1-1/p}} \one_{[2,\infty]}(p) \right)e^{- \widetilde{L} \Tlambda r^2  }. 
	\end{align*}
	We may take, e.g., $r=\Tlambda^{-1/4}$ and, noting that  $\varepsilon_{\Tlambda} \leq \varepsilon_{\Tlambdaunderline}$, it is clear that this display gives rise to part of the first line in $ C^{(o)}$ defined in \eqref{eq:Co}. Let us now bound the third term in the display right after \eqref{eq:adaptationpre}. We define $M:=\tau - 1/{\sqrt{\underline{\lambda}}} 
	$ and note that
	, in view of the present assumption on $\tau$, $M\geq \max\left\{1,C_2^{(MS)}\right\}^2$. Then, by Lemma \ref{lem:JhatT} and using the same arguments as at the end of its proof, 
	\begin{align*}
	\sup_{ \nu\in  \mathcal{V}} \Pr \left( \JThat>\JTstar\right) &\leq  32e^{-\sqrt{M}2^{\JTstar/p} \left(\log T\right)^{\one_{\{\infty\}}(p)}} +  2\log^2(T)e^{-C_3^{(MS)}  \Tlambdaunderline}  \notag \\
	& \quad + 64 e^{- C_4^{(MS)} \twoJTstar}  + 6 \log^2(T)e^{- C^{(C)} \Tlambdaunderline   }  \notag \\
	& \quad + 6\log^2(T)\times \left\{ \begin{array}{ll}
	T^{-C_2^{(R)} \left(\log T\right)^{\frac{\eta}{2}\left( 1-\theta_H\right)-1
	}},& \mbox{if } p=1,\\
	\exp\left(-C_2^{(R)} T^{\frac{1}{2} - \left(\frac{1}{p}-\frac{1}{2}\right) \left( 2\theta_H' +1\right) }  \right), & \mbox{if } p\in(1,2),\\
	\exp\left(-C_2^{(R)} T^{\frac{1}{p}} \right), & \mbox{if } p\in[2,\infty),\\
	T^{-C_2^{(R)} \left(\log T\right)^{\frac{1}{2}(\eta -1)}}, & \mbox{if } p=\infty,
	\end{array}
	\right.
	\end{align*}
	and, in view of \eqref{eq:J*sim}, $M>1> 1/2$ and $\varepsilon_{\Tlambda} \leq \varepsilon_{\Tlambdaunderline}$, the last term in the display  after \eqref{eq:adaptationpre} immediately gives rise to the last summand in the second line of $ C^{(o)}$. To bound the first term in the display  after \eqref{eq:adaptationpre} we first note that, in view of the display right before \eqref{eq:presplitnunhat-nu} 
	and $ r=\Tlambda^{-1/4}< \min\{1,e^{-2\lambdaDn}/(4\lambdaDn)\}$, on $\Omegantilde\big(2^{j},r\big)$  the bound in \eqref{eq:boundnuThatjLpnorm} is valid with $\UT$ replaced by $\lambda\big( \frac{5}{4}e^{2\lambdasupnDn} +1\big)$ too. Then, if  $p\in[1,\infty)$, the first term in the display right after \eqref{eq:adaptationpre} is bounded above by 
	\begin{align*}
	& \lambda\left( \frac{5}{4}e^{2\lambdasupnDn} \!+\!1\right) \!\left( \frac{\Tlambda^{1+(\theta_H+2\theta_H')/(2p)}}{\lambda^{1+(\theta_H+2\theta_H')/(2p)}}\one_{[1,2)}(p)\! +\!\frac{\Tlambda^{1-{1}/{p}}}{\lambda^{1-1/p}} \one_{[2,\infty]}(p) \right)\Pr\!\left( \JThat\!>\! \JTstar\right),
	\end{align*}
	and, due to the second-to-last display, it gives rise to the remaining part of the first line in $ C^{(o)}$. When $p=\infty$, the first term in the display  after \eqref{eq:adaptationpre} can be written as
	\begin{equation*}
	\sum_{ \substack{ j\in \mathcal{J}: \\j>\JTstar}}\E\left[\norm{\nuThat(j)}_p \one_{ \Omegantilde\left(2^{j},r\right) \cap \{\JThat> j\} }\right] \leq 8\lambda\left( \frac{5}{4}e^{2\lambdasupnDn} +1\right) \sum_{ \substack{ j\in \mathcal{J}: \\j>\JTstar}}  2^j	e^{-\sqrt{M} (j+\log T)}+o(\varepsilon_{\Tlambdaunderline}) 
	\end{equation*}
	where the  inequality is valid when $T$ is large enough and follows from \eqref{eq:boundnuThatjLpnorm} with $\UT$ replaced by $\lambda\big( \frac{5}{4}e^{2\lambdasupnDn} +1\big)$ together with Lemma \ref{lem:JhatT}. The sum on the right hand side equals
	\begin{align*}
	&e^{-\sqrt{M} \log T} \sum_{j>\JTstar}  2^{-j(\sqrt{M} \log_2 e -1)}	\leq 2 T^{-\sqrt{M}} 2^{-\JTstar(\sqrt{M} \log_2 e-1)}\\
	&\lesssim 2 \bar{\lambda}^{\sqrt{M}} \!\left( C^{(B)} \bar{\lambda} B_1 \right)^{2\frac{\sqrt{M} \log_2 e-1}{2s+1}} \!\!\Tlambda^{ \frac{1}{2s+1}-\sqrt{M}(1+\frac{\log_2 e}{2s+1})} (\log\! \Tlambda )^{-\sqrt{M}\frac{\log_2 e-1}{2s+1}},
	\end{align*}
	with the first inequality following from $M>1> 1/\log_2 e$ together with the same arguments as at the end of Lemma \ref{lem:JhatT}, and the second inequality being justified by \eqref{eq:J*sim}.  Due to $M>1$ and  $\varepsilon_{\Tlambda} \leq \varepsilon_{\Tlambdaunderline}$, the last display is $o(\varepsilon_{\Tlambdaunderline})$ and the first term in the display  after \eqref{eq:adaptationpre} with $p=\infty$ gives rise to the first summand in the second line in $ C^{(o)}$.
	We are left with bounding the supremum of the second term in \eqref{eq:adaptationpre}. Note that, due to $\lambda\in[\underline{\lambda},\bar{\lambda}]$ and $\norm{ \log^{\beta} (\max\{|\cdot|,e\}) \mu }_1\leq B_4$ for some $\beta>1$ and $B_4\geq 1$, for $T$ large enough depending on $\underline{\lambda},\bar{\lambda}$ and $B_4$, the bounds in \eqref{eq:rbounds}  (the lower bound up to constants) allow to take an intermediate $r$ for all $\lambda\in[\underline{\lambda},\bar{\lambda}]$. Thus, by the triangle inequality, the definition of $\JThat$ in \eqref{eq:JThat},  Propositions \ref{prop:mainstochastic}, \ref{prop:remainder} and Lemmata \ref{lem:bias}, \ref{lem:Omegantildecomplement}, for $T$ large enough the supremum of the second term in \eqref{eq:adaptationpre} is bounded above by
	\begin{align*}
	&\E\left[\norm{\nuThat(\JThat) - \nuThat(\JTstar)}_p \one_{ \JThat\leq \JTstar}\right]+E\left[\norm{\nuThat(\JTstar) - \nu}_p \right] \\
	& \,\,\,\leq \left(\tau\sqrt{\lambda} +2 \lambda e^{2\lambdaDn} L_p M_p\right) \sqrt{\frac{\twop{\JTstar}
		}{\Tlambdaunderline}}  \!+\! C^{(B)}\norm{\nu}_{B_{p,\infty}^s} 2^{-s\JTstar }\\
	& \,\,\, \quad + \norm{\nu}_p \frac{e^{3\lambdaDn}}{\Tlambda}+  C^{(R)}\lambda^2 \Dn e^{4\lambdaDn} r^2 \left(  H_{T,\JTstar}^{\left(\frac{2}{p}-1\right)}2^{\JTstar\!/p} \one_{[1,2)}(p) \!+\! 2^{\JTstar\left(1-\frac{1}{p}\right)} \one_{[2,\infty]}(p) \right) \\
	&\,\,\, \quad + \norm{\nu}_{p,\alpha'} H_{T,\JTstar}^{-\alpha'}\!\! \one_{[1,2)}(p)\\
	&\,\,\, \quad +\left(\!U_T \!\left(\twoJTstar H_{T,\JTstar}^{1/p}\!\! \one_{[1,2)}(p) \!+\!2^{\JTstar\left(1-\frac{1}{p}\right)} \one_{[2,\infty]}(p)\right) \!+\!\norm{\nu}_p\!\right)  \! e^{- \widetilde L \Tlambda r^2  } 	\\
	& \,\,\,\leq \!\left(\tau\sqrt{\lambda} +2 {\lambda }  e^{2\lambdaDn} L_p M_p+1  \right)\left( \lambda C^{(B)}\norm{\mu}_{B_{p,\infty}^s} \right)^{\frac{1}{2s+1}} \sqrt{\frac{ \twop{J_{\Tlambda}}
		}{\Tlambda}} \\
	& \,\,\, \quad +\! C^{(R)}\lambda^2 \Dn e^{4\lambdaDn} r^2  \left( \!\lambda^{-\theta_H(\frac{1}{p}\!-\!\frac{1}{2})}\left( \lambda C^{(B)}\norm{\mu}_{B_{p,\infty}^s} \right)^{\!\frac{1+(2/p-1)(2\theta_H'-\theta_H +1 )}{2s+1}}  \!\!H_{\Tlambda}^{ \left(\frac{2}{p}-1\right)}2^{J_{\Tlambda}/p} 
	\right. \\
	& \qquad\qquad\qquad\qquad\qquad\qquad \times \one_{[1,2)}(p) +\left.\left( \lambda C^{(B)}\norm{\mu}_{B_{p,\infty}^s} \right)^{\frac{2(1-1/p)}{2s+1}} 2^{J_{\Tlambda}\left(1-\frac{1}{p}\right)} \one_{[2,\infty]}(p) \right) \\
	&\,\,\, \quad +  \lambda\norm{\mu}_p \frac{e^{3\lambdaDn}}{\Tlambda}+\lambda^{1+ \alpha' \frac{\theta_H}{2}} \norm{\mu}_{p,\alpha'}  \left( \lambda C^{(B)}\norm{\mu}_{B_{p,\infty}^s} \right)^{\frac{\alpha' (\theta_H-2\theta_H')}{2s+1}} H_{\Tlambda}^{-\alpha'} \one_{[1,2)}(p)\\
	& \,\,\, \quad +  \left(  \lambda^{-\vartheta_U}U_{\Tlambda} \left(  \lambda^{-\frac{\theta_H}{2p}} \left( \lambda C^{(B)}\norm{\mu}_{B_{p,\infty}^s} \right)^{\big(2+\frac{\theta_H-2\theta_H'}{p}\big)\frac{1}{2s+1}} \twoJn H_{\Tlambda}^{\frac{1}{p}} \one_{[1,2)}(p)\right. \right.\\
	&\quad\qquad\qquad\qquad\quad\,\,\,+\!\!\left. \left.\left( \lambda C^{(B)}\norm{\mu}_{B_{p,\infty}^s} \right)^{\frac{2(1-1/p)}{2s+1}} 2^{J_{\Tlambda}\left(1-\frac{1}{p}\right)} \one_{[2,\infty]}(p) \right) \!\!+\!\lambda \norm{\mu}_p\right)   e^{- \widetilde L 
		\Tlambda r^2  } \!,
	\end{align*}
	where in the last equality we used the definition of $\JTstar$ in \eqref{eq:defJ*} together with \eqref{eq:J*sim}, and $H_{\Tlambda}, J_{\Tlambda}, U_{\Tlambda}$ are as in Theorem \ref{thm:upperbound}.
	Then, in view of $1\leq \norm{\mu}_{B_{p,\infty}^s}\leq B_1$, the bounds for $\theta_H,\theta_H'$ and the fact that $M_p'\leq \bar M_p$
	, together with $\varepsilon_{\Tlambda} \leq \varepsilon_{\Tlambdaunderline}$, we can argue as at the end of the proof of Theorem \ref{thm:upperbound} to conclude that $r$ can be chosen so that the last display is bounded above by \eqref{eq:adaptation} except for the first two lines in $C^{(o)}$.
	
\end{proof}

\subsection{Proof of Theorem \ref{thm:varphivarphin} 
}\label{sec:proofs:varphivarphin}

\begin{proof}

	%
	
	\begin{enumerate}

		\item Due to the form of $\Pn$ as an infinite series given in \eqref{eq:repPintro}, we can split a sample $\Zs$ from it into the increments coming from the measure $e^{-\lambdaDn} \delta_{0}$ (no jumps) and those coming from $e^{-\lambdaDn}\sump{1}{}{n}$ (jumps). Recall that $n_0 $ denotes the total number of observations falling into the first type and $\ZJ{1},\ldots \ZJ{n-n_0}$ are the rest of the increments. Then, after some algebra, 
		\begin{align*}
		\varphinhat(\xi)- \varphin(\xi) = & \left(\frac{n_0}{n} - e^{-\lambdaDn}\right) \left( 1- \E  e^{\twopii \xi \ZJ{1}}   \right)  \\
		&+ \frac{n-n_0}{n} \left(\frac{1}{n-n_0} \sum_{i =1}^{n-n_0} \left(  e^{\twopii \xi \ZJ{i}}   - \E  e^{\twopii \xi \ZJ{1}}  \right)   \right)
		\end{align*}
		and, 
		by the triangle inequality, 
		\begin{align}\label{eq:maxineqvarphi}
		\E \sup_{|\xi|\leq \Xi'} n|\varphinhat(\xi) - \varphin(\xi)|  \leq &
		\E \left| n_0 - n e^{-\lambdaDn}\right| \left|1-\E  e^{\twopii \xi \ZJ{1}}\right| \notag \\
		&+ \E \left[m \,E \left[ \sup_{|\xi|\leq \Xi'} \left|\varphinhatJ{m} (\xi) - \varphinJ(\xi) \right| \, \bigg|_{n_0=n-m}  \right]\right],
		\end{align}
		where $\varphinJ$ and $\varphinhatJ{m}$  are the characteristic function of an observation of the second type and its empirical version, respectively. 
		The first term on the right side 
		is bounded above by two times the right side in \eqref{eq:maxineqvarphiparametric}.
		Notice that the second type of observations only depend on $n_0$ through the total number of them $n-n_0$. Hence, for any fixed $m,n\in \N$, the inner expectation in the second term of \eqref{eq:maxineqvarphi} satisfies that for any $\delta>0$
		\begin{align*}
		\E  \sup_{|\xi|\leq \Xi'} \left|\varphinhatJ{m} (\xi) - \varphinJ(\xi) \right|  \leq &\frac{\log(e+\Xi')^{1/2+\delta}}{\sqrt{m}}\\
		&\times \sup_{\widetilde{m} \in \N} \E \left[ \sup_{\xi \in \R}  \sqrt{\widetilde m}  \left|\varphinhatJ{\widetilde m} (\xi) - \varphinJ(\xi) \right| \log(e+|\xi|)^{-1/2-\delta}\right] .
		\end{align*}
		Then, the proof of Theorem 1 in \cite{KR10} can be refined in the exact same way that Theorem 4.1 in \cite{NR09} was refined by Theorem 3.1 in \cite{C18} to give that, if $\beta>1$, the supremum of the expectation in the last display is bounded above by
		\begin{equation*}
		C_1' + C_2' \left( \E \log^{\beta}\left(\max\left\{|\ZJ{1}|,e\right\}\right)\right)^{\frac{1}{2\beta}},
		\end{equation*}
		where $C_1'$ and $C_2'$ are universal constants (so, in particular, independent of $\Dn$, $\lambda$ and $\mu$).  In view of Proposition 25.4 (iii) in \cite{S99}, the function  $\log(\max\{|\cdot|,e\}): \R \to [1,\infty)$ is submultiplicative, i.e. there exists a constant $b>0$ such that for all $x,y \in \R$
		\begin{equation*}
		\log(\max\{|x+y|,e\}) \leq b \log(\max\{|x|,e\}) \log(\max\{|y|,e\});
		\end{equation*}
		thus, so is $\log^\beta(\max\{|\cdot|,e\}): \R \to [1,\infty)$ with constant $b^\beta$.  Then, using Fubini's theorem and applying the submultiplicative property repeatedly,
		\begin{align*}
		\E \log^{\beta}\left(\max\left\{|\ZJ{1}|,e\right\}\right) & = \left(e^{\lambdaDn}-1\right)^{-1} \sum_{i =1}^\infty \frac{\Dnsuper{i}}{i!} \int_{\R} \log^\beta(\max\{|x|,e\}) \nu^{\ast i}(dx) \\
		& \leq 	\frac{ \sum_{i =1}^\infty \frac{(b^\beta \Dn)^i}{i!} \norm{\left( \log^\beta(\max\{|\cdot|,e\}) \nu \right)^{\ast i}}_{L^1(\R)} }{b^\beta \left(e^{\lambdaDn}-1\right) }\\
		& \leq \frac{  \left(e^{a b^\beta  \lambdaDn}-1\right)}{b^\beta  \left(e^{\lambdaDn}-1\right)},
		\end{align*}
		where the last inequality is justified by repeated use of Young's inequality for convolutions. Gathering all these conclusions we have that the second term on the right hand side of \eqref{eq:maxineqvarphi} is bounded above by
		\begin{align*}		
		&\left(C_1'+ C_2' \left( \frac{  \left(e^{a b^\beta \lambdaDn}-1\right)}{b^\beta \left(e^{\lambdaDn}-1\right)} \right)^{\frac{1}{2\beta}}\right) {\log^{1/2+\delta} (e+\Xi')}  \E  \sqrt{n-n_0}. 
		\end{align*}
		In view of \eqref{eq:Esqrtn-n0bound}, the conclusion  follows taking $C_2= C_2'$ and $C_1= 
		2+ C_1'$ because $\Xi'\geq 0$.

		\item For any complex number $z$
		\begin{align*}
		\frac{1}{2} \left( |\Ree{z}| + |\Imm{z}|  \right)  \leq |z| \leq |\Ree{z}| + |\Imm{z}|,
		\end{align*}
		from where it follows that, for any $n\in \N$ and $\Xi', t\geq 0$,
		\begin{align}\label{eq:ineqvarphireim}
		\Pr&\left(\sup_{|\xi|\leq \Xi'}n|\varphinhat(\xi) - \varphin(\xi)| \geq 2\E\sup_{|\xi|\leq \Xi'}n|\varphinhat(\xi) - \varphin(\xi)| + t \right) \\
		&\leq \Pr\left(\sup_{|\xi|\leq \Xi'}|n\Ree{\varphinhat(\xi) - \varphin(\xi)}| \geq \E\sup_{|\xi|\leq \Xi'}|n\Ree{\varphinhat(\xi) - \varphin(\xi)}| + t/2 \right) \notag \\
		&\quad+\Pr\left(\sup_{|\xi|\leq \Xi'}|n\Imm{\varphinhat(\xi) - \varphin(\xi)}| \geq \E\sup_{|\xi|\leq \Xi'}|n\Imm{\varphinhat(\xi) - \varphin(\xi)}| + t/2 \right). \notag
		\end{align}
		
		Trivially, 
		\begin{equation*}
		n\,\Ree{\varphinhat(\xi) - \varphin(\xi)}= \sum_{i=1}^n W_i(\xi) \quad \mbox{ and } \quad n\,\Imm{\varphinhat(\xi) - \varphin(\xi)}= \sum_{i=1}^n W_i'(\xi), 
		\end{equation*}
		where		
		\begin{equation*}
		W_i(\xi):=\cos( \twopi \xi Z_i)  - E\cos( \twopi \xi Z_i) \quad \mbox{ and } \quad W_i'(\xi):=\sin( \twopi  \xi Z_i)  - E\sin( \twopi  \xi Z_i).
		\end{equation*}
		For any $\xi\in \R$, $W_i(\xi)$, $i=1, \ldots,n$, are independent and centred,
		and the same holds for any $W_i'(\xi)$. Furthermore,  for any $\xi\in \R$ and $i=1, \ldots,n$,
		\begin{align*}
		\E W_i(\xi)^2 +\E W_i'(\xi)^2 = & \, 		\E   \cos^2\left(\twopi \xi Z_i\right) -  \left(\E \cos\left(\twopi \xi Z_i\right)\right)^2 \\
		&+ \E\sin^2\left(\twopi \xi Z_i\right) -  \left(\E \sin\left(\twopi \xi Z_i\right)\right)^2  \\
		= &\,1- |\varphin(\xi)|^2 = 1-e^{ 2\Dn \left( \Ree{ \FT \nu (\xi)\right) - \lambda}} \\
		\leq &\,\min\{1,2\Dn \left(  \lambda - \Ree{ \FT \nu (\xi)} \right)\} \leq \min\{1,4 \lambdaDn\}, 
		\end{align*}
		where in the first inequality we have used that $|\text{Re}\left( \FT \nu (\xi)\right)| \leq \lambda$ and that $1-\exp(-x) \leq \min\{1,x\}$ for all $x \geq 0$. Additionally, due to the continuity of $\varphin$ and $\varphinhat$, the suprema on the right hand side of \eqref{eq:ineqvarphireim} do not change if they are taken over a dense and countable subset such as $[-\Xi',\Xi']\cap \mathbb{Q}$, where $\mathbb{Q}$ is the set of rational numbers in $\R$. Therefore, we can apply the upper tail of Talagrand's inequality (cf. expression  (3.100) in Theorem 3.3.9 of \cite{GN16}) to each of the probabilities therein with  upper bounds for the envelope and the variance given, respectively, by $1$ and
		\begin{equation*}
		2\E\sup_{|\xi|\leq \Xi'}n|\varphinhat(\xi) - \varphin(\xi)| + 4 \Tlambda,
		\end{equation*}
		and the conclusion follows immediately.

		\item 
		To show the result we take 
		\begin{equation*}
		t= \Tlambda r - 2 \E\sup_{|\xi|\leq \Xi'}n|\varphinhat(\xi) - \varphin(\xi)|. 
		\end{equation*}
		In view of part (a) and the lower bound on $r$, 
		\begin{equation*}
		t \geq \left( \frac{1}{2\sqrt{c}} -2 C \right)  \sqrt{\Tlambda{\log^{1+2\delta}(e+\Xi')} } + \Tlambda r /2,
		\end{equation*}
		and, due to $r \leq 1 $, we also have that $t\leq \Tlambda$. Notice that the term in brackets is nonnegative by  condition $0<c< (4 C)^{-2}$ and that $\log^{1+2\delta}(e+\Xi') \geq \log(e+\Xi')$ due to $\delta> 0$ and $\Xi' \geq 0$. Then, by part (b) and  applying part (a) again, the left hand side in \eqref{eq:concentrationsupr} is bounded above by
		\begin{align*}
		2   \exp&\left(- \frac{\left(  \left( \frac{1}{2\sqrt{c}} -2 C \right)^2 \Tlambda {\log(e+\Xi')}  + (\Tlambda r)^2/4  \right) }{  16 C  \sqrt{\Tlambda{\log^{1+2\delta}(e+\Xi')} }  + 32\Tlambda+ 8\Tlambda/3  }\right) \\
		&\leq 2   \exp\left(- \frac{  4\left( \frac{1}{2\sqrt{c}} -2 C \right)^2  \log(e+\Xi') + \Tlambda r^2   }{ 64 \sqrt{c} \,C   + 416/3  }\right),
		\end{align*}	
		where the last inequality follows by assumption $\log^{1+2\delta}(e+\Xi') \leq c \Tlambda$. The conclusion then follows by condition $c< (4 C)^{-2}$, which is equivalent to $4 \sqrt{c} \,C < 1$, together with assumption $e+\Xi' \geq 2^{\frac{1}{4L(1/(2\sqrt{c}) -2C)^2}}$.

	\end{enumerate}	
	
\end{proof}

\end{document}